\documentclass{article}

\usepackage{microtype}
\usepackage{graphicx}
\usepackage{subcaption}
\usepackage{booktabs} 

\usepackage{hyperref}

\usepackage[preprint]{icml2026}

\usepackage{amsmath}
\usepackage{amssymb}
\usepackage{mathtools}
\usepackage{amsthm}

\usepackage[capitalize,noabbrev]{cleveref}

\usepackage{hyperref}

\usepackage{enumitem}

\usepackage{subcaption}

\usepackage{longtable}
\theoremstyle{plain}
\newtheorem{theorem}{Theorem}[section]
\newtheorem{proposition}[theorem]{Proposition}
\newtheorem{lemma}[theorem]{Lemma}

\theoremstyle{definition}
\newtheorem{definition}[theorem]{Definition}
\newtheorem{assumption}[theorem]{Assumption}
\theoremstyle{remark}

\newtheorem{examp}[theorem]{example}

\newcommand{\e}[1]{\mathbb{E}\left[#1\right]  } 
\newcommand{\indep}{\perp\!\!\!\perp}

\usepackage[textsize=tiny]{todonotes}

\icmltitlerunning{Conditional Feature Importance revisited: Double Robustness, Efficiency and Inference}

\begin{document}

\twocolumn[
  \icmltitle{Conditional Feature Importance revisited: Double Robustness, Efficiency and Inference}



  \icmlsetsymbol{equal}{*}

  \begin{icmlauthorlist}
    \icmlauthor{Angel Reyero Lobo}{yyy,comp}
    \icmlauthor{Pierre Neuvial}{yyy}
    \icmlauthor{Bertrand Thirion}{comp}
  \end{icmlauthorlist}

  \icmlaffiliation{yyy}{Institut de Mathémathiques de Toulouse, UMR5219, Université de Toulouse, CNRS, Toulouse, France}
  \icmlaffiliation{comp}{Université Paris-Saclay, Inria, CEA, Palaiseau, 91120, France}

  \icmlcorrespondingauthor{Angel Reyero Lobo}{angel.reyero-lobo@inria.fr}

  \icmlkeywords{Variable Importance, Conditional Feature Importance, Controlled Variable Selection}

  \vskip 0.3in
]



\printAffiliationsAndNotice{}  

\begin{abstract}
 Conditional Feature Importance (CFI; \citet{Strobl2008}) is a classical variable importance measure that accounts for the relationship between the studied feature and the others. However, CFI has not yet been studied from a theoretical perspective because the conditional sampling step has generally been considered  a purely practical problem. In this article, we demonstrate that the recent Conditional Permutation Importance (CPI) is indeed a valid implementation of this concept. Under the conditional null hypothesis, we then establish a \textbf{double robustness} property that can be used for variable selection. With either a valid model or a valid conditional sampler, the method correctly identifies null coordinates. 

    Under the alternative hypothesis, we study the theoretical target and link it to the popular Total Sobol Index (TSI). We introduce the \textbf{Sobol-CPI}, which generalizes CPI/CFI, prove that it is nonparametrically efficient, and provide a bias correction. Finally, we propose a consistent and valid \textbf{type-I error} test and present numerical experiments that illustrate our findings.
\end{abstract}

\section{Introduction}

Modern machine learning models can make accurate predictions, so using them can help characterize the dependency structure between variables and thus provide valuable insights for further research.   
For example, if a model can accurately predict a disease based on a set of genes, identifying the important predictive variables can guide specialists on which genes to study further to combat the disease.

There is usually a trade-off between \textit{model transparency} and \textit{complexity}.
Indeed, simple methods such as linear regression are easily interpretable, since each covariate's importance can be directly assessed via its estimated coefficient. However, using a model that poorly fits the underlying distribution may yield misleading results \citep{molnar2021general}.
Therefore, one needs to rely on \textit{model-agnostic} approaches based on flexible models that give information about the data distribution. 
In this way, one can adapt to complex situations by capturing non-linear dependencies, recovering the truly important variables in correlated settings.

The two main model-agnostic approaches to measure the importance of a covariate are \textit{perturbation} and \textit{removal} approaches \citep{covert}. Both of them aim to disable the information given by the covariate.

In perturbation/permutation approaches, such as Permutation Feature Importance (PFI, \citet{Mi2021}) and Conditional Permutation Importance (CPI, \citet{chamma2024statistically}), the information is disabled by computing the loss when permuting the coordinate marginally or conditionally.

The most used removal method is Leave One Covariate Out (LOCO), where the importance of a covariate is estimated by evaluating the performance of a model re-trained without that covariate \citep{lei2018,williamson2021}. 
%
%

LOCO is essentially a direct plug-in estimate of the predictive power of the model compared to the predictive power of the restricted model \citep{williamson2023general}. 
As such, it effectively estimates the Total Sobol Index (TSI), an importance index originating from Sensitivity Analysis \citep{first-total-order-effects}. 
TSI aims to assess the predictive capacity/ explained variance of a coordinate given the others.

Due to extrapolation issues with PFI \citep{Hooker2021}, there has been a shift towards conditional approaches. These include Conditional Feature Importance (CFI) \cite{Strobl2008} and Conditional Model Reliance (CMR; \citet{fisher2019all}), defined as a ratio rather than a difference. However, the conditional sampling step is often \textit{overlooked}; a Gaussian distribution or copula is commonly used but performs poorly in practice due to covariance estimation issues \citep{blain2025knockoffsfaildiagnosingfixing}. Our first contribution is to show that the conditional permutation used in CPI (and CMR) has theoretical grounding, making it an instance of CFI.

These conditional methods have shown strong empirical performance (\citet{chamma2024statistically, chamma2024variable, paillard2025measuringvariableimportanceheterogeneous}). We explain this by a double robustness property, typically sought in causal inference, which we introduce here for variable importance.

Variable selection seeks the smallest set of predictive variables—those not independent of the output \textit{given} the others.
%
%
Assuming that each covariate is not directly a function of the others, this set is well-defined and unique.
This assumption is standard in variable selection \citep{candes2017panninggoldmodelxknockoffs} and feature importance \citep{locoVSShapley}. 
%

Permutation-based approaches do not provide a proper quantification of variable importance.
%
In particular, \citet{scornet2022mda} have shown that PFI does not converge to TSI. Yet, the theoretical quantity targeted by CPI has not been studied. 
In this article, we fill this gap and show that, with a simple adjustment, the method can be adapted to estimate TSI. We call this extension \textit{Sobol-CPI}, an unbiased generalization of CPI, for which we establish nonparametric efficiency and study inference to provide statistical guarantees.

In summary, our main contributions are:
\begin{itemize}
    \item Provide conditions under which the CPI conditional sampler is valid, showing that CPI constitutes a CFI measure.
    \item a \textit{double robustness} property: to detect a null covariate, only one of the two estimates used in CFI needs to be accurate. This explains the empirical performance of CPI. By contrast, LOCO requires both models to be accurate. In linear settings, CPI’s bias decays quadratically, while LOCO’s decays linearly.
    \item \textit{Sobol-CPI}, a new estimator of TSI based on a permutation approach, bridging the gap with removal methods. We prove its asymptotic efficiency.
    \item a consistent and valid test using variance correction, and a nonparametric test with finite-sample guarantees based on the sampler’s accuracy.
    \item extensive numerical experiments to illustrate these findings in simulated and real-data.
\end{itemize}
Proofs and additional experiments are in the appendix. These results hold for \textit{any} conditional sampler.
\section{Framework}
\subsection{Setting}\label{subsec:setting}
We observe $\{(x_i, y_i)\}_{i=1, \ldots, n_\mathrm{test}, \ldots, n_\mathrm{test} + n_\mathrm{train}}$ a test and train set  sampled i.i.d. from $P_0$, a distribution from a class of distributions $\mathcal{M}$, where $X \in \mathcal{X}\subset \mathbb{R}^p$ and $y\in \mathcal{Y}\subset \mathbb{R}$. We are interested in assessing the importance of a coordinate $j\in\{1, \ldots, p\}$. 
We denote by $X^j$ the $j$-th column of $X$ and by $X^{-j}$ the vector $X$ with the $j$-th coordinate excluded. 
%
%
We consider a space of functions $\mathcal{F}$ with a norm $\|\cdot\|_\mathcal{F}$. The notation is summarized in the glossary (Appendix \ref{app:glossary}).

Similarly to the Model-X framework \citep{barber-candes2015,candes2017panninggoldmodelxknockoffs}, we want to avoid making strong assumptions about the relationship between inputs and outputs, as it can be complex. However, we assume that the relationship among the input covariates is simple, typically because they originate from the same generative process (e.g. a measurement device). 
%
%
Then, it can be much more efficient and accurate to study the relationship between the $j$-th covariate, $X^j$, given the rest, $X^{-j}$, rather than directly regressing $y$ on $X^{-j}$. This approach is also advantageous in settings where labeling data is expensive, as there are typically more available samples for the input pair $(X^{-j}, X^j)$ than for the pair $(X^{-j}, y)$. However, CPI’s empirical robustness to this assumption is shown in \cite{paillard2025measuringvariableimportanceheterogeneous}.

We denote by $m(X) := \mathbb{E}[y\mid X]$ (resp. $m_{-j}(X^{-j}) := \mathbb{E}[y\mid X^{-j}]$) the conditional expectation of the input $y$ given $X$ (resp. $X^{-j}$), by $\widehat{m}$ (resp. $\widehat{m}_{-j}$) its estimation.  Similarly, we denote by $\nu_{-j}(X^{-j}) := \mathbb{E}[X^j\mid X^{-j}]$ and $\widehat{\nu}_{-j}$ its estimate. We denote by $\ell : \widehat{\mathcal{Y}} \times \mathcal{Y} \to \mathbb{R}$ the loss used to compare a prediction $\widehat{m}(X)$ with the output $y$. 
These estimations are the usual objective of machine learning models. This is obvious for the quadratic loss, but for many other losses ($R^2$, deviance, classification accuracy, and the area under the ROC curve, see  \citet{williamson2023general}) the minimizer is also a function of this conditional expectation.


\begin{definition}[TSI]\label{def:TSI} Given $(X, y)\sim P_0$ and a loss $\ell$ we define the Total Sobol Index of $j\in\{1, \ldots, p\}$ as
\begin{align*}
    \psi_{TSI}(j, P_0):=\e{\ell\left(m_{-j}(X^{-j}), y\right)}-\e{\ell(m(X), y)}.
\end{align*}
\end{definition}

We define 
this quantity as the (Generalized) Total Sobol Index because, when the loss is the quadratic loss, it corresponds to the unnormalized TSI \citep{first-total-order-effects}. 
It represents the loss decay when the covariate is not used. It has also been referred to as the Generalized ANOVA \citep{williamson2021} and is considered a standard importance measure \citep{lei2018, Rinaldo2019, Hooker2021, scornet2022mda, williamson2023general}. A large decay indicates that the feature is useful, whereas a small one indicates that it lacks predictive power when the remaining covariates are used. \citet{scornet2022mda} considered it as the best quantity for support recovery.

\subsection{Related work}

LOCO consists of a simple plug-in estimate of the TSI: 

\begin{definition}[LOCO] Given $j$, a loss $\ell$, a regressor $\widehat{m}$ of $y$ given $X$, a regressor $\widehat{m}_{-j}$ of $y$ given $X^{-j}$ and a test set $(X_i, y_i)_{i=1,\ldots n_\mathrm{test}}$, LOCO is defined as
\begin{align*}
    \widehat{\psi}_\mathrm{LOCO}^j=\frac{1}{n_{\mathrm{test}}}\sum_{i=1}^{n_{\mathrm{test}}}\ell\left(\widehat{m}_{-j}(x_{i}^{-j}), y_i\right)-\ell\left(\widehat{m}(x_{i}),y_i)\right).
\end{align*}
\end{definition}


Nonparametric theory \citep{kennedy2023semiparametricdoublyrobusttargeted} shows that a simple plug-in estimate does not achieve optimal convergence, typically requiring a one-step correction. However, \citet{williamson2021} demonstrated that this correction is unnecessary with the quadratic loss, and \citet{williamson2023general} extended this to other losses under certain regularity conditions. This provides valid confidence intervals. Similar results will be established for our method in \cref{subsec:asymp_eff}.


A major practical limitation is that it requires retraining a model, $\widehat{m}_{-j}$, for each coordinate $j$. This is computationally intensive and it also introduces optimization errors that do not compensate as desired, as discussed in \cref{subsec:double-robust}.

A first naive \textit{permutation-based} approach consists in comparing the performance of $\widehat{m}$ on a test set where the $j$-th column is shuffled: the formal definition of PFI is given in Appendix \ref{sect:PFI}.
%
Even though it avoids refitting and \citet{Mi2021} report good performance, it suffers from extrapolation bias, as predictions may fall in low-density regions where $\widehat{m}$ was not trained \citep{Hooker2021}. It also fails to control type-I error with highly correlated covariates \citep{chamma2024statistically}. Moreover, \citet{scornet2022mda} showed that it does not target an interpretable quantity: it decomposes into three components, only the first being desirable ($\psi_\mathrm{TSI}$). The other are misleading with correlation.



To address this issue, \citet{Strobl2008} proposed \textit{conditionally} permuting the $j$-th coordinate with respect to the others. In this way, only the information exclusively given by $X^j$ is shuffled, while the relationship with the rest of the coordinates is preserved. As a result, one can make predictions based on the new conditional sample, denoted by $\widetilde{x}'^{(j)}_i$, 
where the apostrophe indicates that it was estimated. 
 More formally, \citet{chamma2024statistically} proposed:

\begin{definition}[CPI] Given $j$, a loss $\ell$, a regressor $\widehat{m}$ of $y$ given $X$ and a test set $(X_i, y_i)_{i=1,\ldots n_\mathrm{test}}$, CPI is defined as
\begin{align*}
    \widehat{\psi}_{\mathrm{CPI}}^j=\frac{1}{n_{\mathrm{test}}}\sum_{i=1}^{n_{\mathrm{test}}}\ell\left(\widehat{m}(\widetilde{x}_{i}'^{(j)}),y_i\right)-\ell\left(\widehat{m}(x_{i}), y_i\right),
\end{align*}
where the $j$-th coordinate is \textit{conditionally} permuted.
\end{definition}

For the conditional permutation, they proposed regressing the $j$-th coordinate with respect to $X^{-j}$ to estimate $\nu_{-j}(X^{-j}) = \mathbb{E}[X^j \mid X^{-j}]$, and then adding the permuted residuals of this regression. Thus, $\widetilde{x}'^{(j)l} = x^l$ for any $l \neq j$, and $\widetilde{x}'^{(j)j} = \widehat{\nu}_{-j}(x^{-j}) + (x^j - \widehat{\nu}_{-j}(x^{-j}))^\mathrm{perm}$, where the permutation is across the individuals. We denote by $\widetilde{X}^{(j)} \sim P_j^\star \in \mathcal{M}$ the theoretical random variable based on the true $\nu_{-j}$, and $\widetilde{X}'^{(j)} \sim P_j'$ based on the estimated $\widehat{\nu}_{-j}$. We omit the superscript $(j)$ when the coordinate is clear.


CPI requires training a separate regressor $\widehat{\nu}_{-j}$ for each covariate, which may seem counterintuitive since permutation approaches aim to avoid fitting $\widehat{m}_{-j}$. However, as discussed in \cref{subsec:setting}, while the relationship between $y$ and $X$ can be complex, between $X^j$ and $X^{-j}$ is simple. Then, unlike with LOCO, we use a simple model to estimate $\nu_{-j}$. 

The goal of CPI is to sample $\widetilde{X}^{(j)}$ such that $\widetilde{X}^{(j)} \sim X, \widetilde{X}^{(j)-j} = X^{-j}$ and $\widetilde{X}^{(j)j} \indep y \mid X^{-j}$. 
However, \citet{chamma2024statistically} did not discuss the assumptions under which this method samples from the target distribution;
we address this in \cref{subsec:cond_sampl}.
Furthermore, CPI does not estimate a known theoretical quantity, so in \cref{sect:sobol-CPI} we propose a modification to target TSI.

Finally, in \citet{chamma2024statistically}, a type-I error control is proposed using asymptotic normality. However, it does not address the fact that, under the conditional null hypothesis, similarly to what happens with LOCO, the influence function vanishes. Consequently, there is a need to correct the estimated variance to ensure type-I error control \citep{williamson2023general, Dai2024, locoVSShapley}. In \cref{subsec:inference}, we address this issue for valid inference.

\section{Theoretical properties of CPI/CFI}
\subsection{Validity of conditional sampling}\label{subsec:cond_sampl}
We observe that there is no need to preserve the same sampler used for CPI. Normalizing flows \citep{papamakarios2021normalizingflowsprobabilisticmodeling} are theoretically sound, but impractical due to the large number of training samples they require. Other alternatives are domain-specific generative models, e.g. \citet{sesia2020}. Since the conditional sampling used in CPI, as well as in \citet{Hooker2021} and \citet{blain2025knockoffsfaildiagnosingfixing}, works well in practice and improves computational efficiency, we study its validity.
%
Let us denote the $i$-th observation in which only the $j$-th coordinate has been conditionally permuted: 
\begin{equation}
 \widetilde{x}^{'(j)l}_{i}=
\begin{cases}
 x^l_i & \text{if } l\neq j \\
  \widehat{\nu}_{-j}(x_i^{-j})+\left\{x^j-\widehat{\nu}_{-j}(x^{-j})\right\}^\mathrm{perm} & \text{if } l=j
\end{cases}\label{eq:cpi-sampling}
\end{equation}

First, we assume that each covariate brings an additive independent innovation, similar to a standard assumption in regression (\ref{ass:regressModel}), but for each covariate.
 \begin{assumption}[Additive innovation]\label{ass:generalAssumpCondSampl} For each $j\in\{1,\ldots,p\}$, there exists a function $\nu_{-j}$ such that $X^j=\nu_{-j}(X^{-j})+\epsilon_j$ with $\epsilon_j\indep X^{-j}$ and $\e{\epsilon_j}=0$
\end{assumption}
 This $\nu_{-j}$ is exactly the conditional expectation because $$\e{X^{j}\mid X^{-j}}=\e{\nu_{-j}\left(X^{-j}\right)+\epsilon_j\mid X^{-j}}=\nu_{-j}\left(X^{-j}\right).$$
For instance, Gaussian data satisfies this assumption.
\begin{lemma}[Gaussian additive noise] For a Gaussian vector $X$, additive innovation (\ref{ass:generalAssumpCondSampl}) is satisfied.\label{lemma:GaussianAss}
\end{lemma}
We also need the estimate $\widehat{\nu}_{-j}$ to be consistent, i.e.
$\e{\left(\widehat{\nu}_{-j}\left(X^{-j}\right)-\nu_{-j}\left(X^{-j}\right)\right)^2}\to 0.$

For instance, the Random Forest is consistent under mild assumptions (\citet{Scornet_2015ConsistencyRF}). 
If we assume that $X$ is Gaussian, a linear model is consistent. 
To deal with high-dimensionality, we generally use a Lasso \cite{lasso}.

Under the previous assumptions, the 2-Wasserstein distance between the estimated conditional distribution and the true distribution converges to 0:

\begin{proposition}[Empirical conditional sampling]
    Under additive innovation (\ref{ass:generalAssumpCondSampl}), if the regressor $\widehat{\nu}_{-j}$ is consistent, then $\widetilde{x}'^{(j)}$ constructed as in \eqref{eq:cpi-sampling}, is sampled from $P'_j$, and $\mathcal{W}_2(P'_j, P^\star_j)\to 0$. \label{prop:empCondSamp}
\end{proposition}
Proposition~\ref{prop:empCondSamp} implies that CPI's sampler is asymptotically drawing from the desired distribution. Then, CPI qualifies as a Conditional Feature Importance \cite{Hooker2021}. However, in the following, we abstract from this choice and provide general results for any empirical sampler $P'_j$.

\subsection{Double robustness}\label{subsec:double-robust}
We have observed that there is a need to estimate two models for both LOCO ($\widehat{m}$ and $\widehat{m}_{-j}$) and CPI ($\widehat{m}$ and $P'_{j}$). 
In this section, we prove a \emph{double robustness} property for detecting conditionally null covariates: to identify a null covariate, it is sufficient that one of the two estimates is consistent. This contrasts with LOCO, where errors in both estimates must compensate.  This property explains the good empirical results obtained by CPI for variable selection \cite{chamma2024statistically,chamma2024variable, paillard2025measuringvariableimportanceheterogeneous}.


We begin by a general result, then illustrate double robustness for the quadratic loss, and finally derive explicit rates in a linear setting: CPI's bias decays quadratically, whereas LOCO's decays only linearly.

Note that under the null hypothesis, the theoretical model satisfies
$m(X)=\mathbb{E}[y \mid X]=\mathbb{E}[y \mid X^{-j}]$ since $y \indep X^j \mid X^{-j}$; hence, it does not depend on the null feature.
We require that the ML model asymptotically does neither depend on unimportant features:
\begin{assumption}[Asymptotic relevance]\label{ass:asymp_relevance}
Denote by $g_j(x, s)$ the vector $x$ with the $j$-th component replaced by $s \in \mathbb{R}$. For $\epsilon>0, x\in\mathcal{X},s\in\mathbb{R}$ and $X^j\indep y\mid X^{-j}$, there exists $n_0$ such that for $n\geq n_0$, $$|\hat{m}_n(x)-\hat{m}_n(g_j(x, s))|\leq \epsilon\text{ a.s}.$$
\end{assumption}
This is a pointwise convergence on the null features, which can be easily verified for any GLM since the estimated null coefficients tend to 0. It is satisfied under the representability assumption for the Lasso \citep{asymp_rel_lasso}. Under standard assumptions on Random Forests, the splits are performed with high probability only along the important features  (Prop. 1 \citet{Scornet_2015ConsistencyRF}). For kernel methods, convergence in the RKHS together with the reproducing property implies pointwise convergence, and thus asymptotic relevance. For Gaussian processes, variable selection can be achieved using dimension-specific scalings \citep{Bhattacharya_2014}. Finally, there are also results for 1-norm penalized SVM \citep{1_pen_svm} and neural networks \citep{NEURIPS2020_1959eb9d}. In Sections~\ref{subsec:asympt_relev} and \ref{app:asymp_rel_real_data}, we show numerically that this property holds across a broad benchmark of ML models on both simulated and real datasets.

In any case, since our goal is to study conditional independence using a model, we are implicitly summarizing the $X$-$y$ relationship through the first-order moment ($\e{Y \mid X}$), and implicitly assuming that the model mainly relies only on the important features. Thus, the assumption is natural and is a formalization of what is expected from $\widehat{m}$ when using ML models to study conditional independence.

\begin{theorem}[Double robustness]\label{th:doubleRobust}
Assume that $X^j \indep y \mid X^{-j}$. Given a conditional sampler $P'_j$ such that $\widetilde{X}'^{(j)} \overset{\mathcal{L}}{\to} \widetilde{X}^{(j)}$, then for any bounded and continuous loss $\ell$ and continuous $\widehat{m}$, $\widehat{\psi}_\mathrm{CPI}\to 0$ a.s. 
On the other hand, given Assumption \ref{ass:asymp_relevance}, for any conditional sampler $P'_j$ and continuous loss $\ell$, $\widehat{\psi}_\mathrm{CPI} \to 0$ a.s.
\end{theorem}


The first implication of the theorem extends Theorem~2 from \citet{Koenig2021Relative} by making explicit the estimation of the conditional distribution. The second relies on the model's implicit variable selection: as the loss optimizer, the model will generally not assign importance to irrelevant features, under the standard assumptions ensuring consistency of the ML model. For unbounded losses, the result holds if the loss is clipped with a sufficiently large constant.

\paragraph{Quadratic loss.} In the rest of this section, we focus on the bias introduced by the need to estimate the regressors, therefore, having $n_\mathrm{train}$ fixed.
As is usual in regression, we assume an independent additive noise:

\begin{assumption}[Additive noise]
    $y=m(X)+\epsilon$ with $\epsilon\perp\!\!\!\perp X$ and $\e{\epsilon}=0$.\label{ass:regressModel}
\end{assumption}

We study the estimation bias of CPI in Proposition \ref{prop:cpi-bias} and of LOCO in Proposition \ref{prop:loco-bias}.

 \begin{proposition}
Assuming $X^j \indep y \mid X^{-j}$ and additive noise (\ref{ass:regressModel}), we have 
\begin{align*}
     \mathbb{E}&\left[\widehat{\psi}_\mathrm{CPI}^j\middle|\mathcal{D}_\mathrm{train}\right]=\mathbb{E}\left[\left(\widehat{m}(\widetilde{X})-\widehat{m}(\widetilde{X}')\right)^2\right.\\&\left.+2\left((m(\widetilde{X})-\widehat{m}(\widetilde{X})))(\widehat{m}(\widetilde{X})-\widehat{m}(\widetilde{X}')\right)\middle|\mathcal{D}_\mathrm{train}\right].
 \end{align*}
 \label{prop:cpi-bias}
\end{proposition}
In this expression, we can observe the mentioned double robustness: it is sufficient to have either an accurate \( \widehat{m} \) or \( P'_j \). Specifically, the first error term will vanish either because \( \widehat{m} \) detected the \( j \)-th coordinate as not important (and therefore the function does not change due to it), or because \( P'_j \) is close to \( P^\star_j \), making \( \widetilde{X}' \) close to \( \widetilde{X} \). The second term is treated similarly.

%
 \begin{proposition}
 Under additive noise (\ref{ass:regressModel}), we have
 \begin{align*}
     \mathbb{E}\left[\widehat{\psi}_\mathrm{LOCO}^j\middle|\mathcal{D}_\mathrm{train}\right]&= \mathbb{E}\left[\left(m_{-j}(X^{-j})-\widehat{m}_{-j}(X^{-j})\right)^2\right.\\&\left.\qquad-(m(X)-\widehat{m}(X))^2\middle|\mathcal{D}_\mathrm{train}\right]. 
 \end{align*}
 \label{prop:loco-bias}
 \end{proposition}
Therefore, there is a bias due to the estimation of both regressors. There are two main reasons why this error does not cancel out. The first one is an \emph{optimization error}, which is more harmful in complex models (\citet{opt_error}). Indeed, it corresponds to a difference of errors from different models that have been independently optimized. Due to the variability of the optimization process, there may be many different solutions and multiple local minima, each with different errors, i.e. given $\widehat{m}_1$ and $\widehat{m}_2$ two optimizations, we do not forcely have $\e{(m(X)-\widehat{m}_1(X))^2}=\e{(m(X)-\widehat{m}_2(X))^2}$.  

The second source of error is an \emph{estimation error}. This arises because the error distributions of both models differ. From statistical learning theory, we know that the distribution typically depends on the dimension of the model. Since the models have different dimensions, they have different error distributions, which prevents the differences from canceling out. For instance, in the linear model example presented below, we show how this estimation error implies that the error remains at the convergence rate of the linear model ($O(1/n)$), rather than achieving a faster rate as for CPI. 


In summary, LOCO compares errors of two independently optimized models over different sets of features, which do not cancel because (i) they are optimized separately and (ii) their error distributions differ. In contrast, Proposition~\ref{prop:cpi-bias} shows that CPI's main error term compares the effect of estimating $\widetilde{X}^{(j)}$ \emph{within the same optimized model}.


\paragraph{Linear model.} 
In the remainder of this section, we focus on a Gaussian linear setting, which helps build intuition for CPI's advantages over LOCO. Indeed, rather than requiring only one model to work to detect a null as before, we show that when both models are accurate, the null is detected faster.
We use the regression-based sampler \eqref{eq:cpi-sampling} to obtain explicit convergence rates.

\begin{assumption}[Linear model]$y=X\beta + \epsilon$ with $\epsilon\overset{\mathrm{i.i.d.}}{\sim}\mathcal{N}(0, \sigma^2)$ and $\beta\in\mathbb{R}^p$.\label{ass:LM}
\end{assumption}
We assume that $X$ is Gaussian. Thus, there is an explicit form for the conditional distribution: $X_{j}\mid X_{-j}=x_{-j}\sim \mathcal{N}(\mu^j_\mathrm{cond}, \Sigma^j_\mathrm{cond})$ 
with $\mu^j_\mathrm{cond}:=\mu_{j}+\Sigma_{j, -j}\Sigma^{-1}_{-j, -j}(x_{-j}-\mu_{-j})$ and $\Sigma^j_\mathrm{cond}:=\Sigma_{j, j}-\Sigma_{j,-j}\Sigma^{-1}_{-j, -j}\Sigma_{-j, j}$. In particular,
$\nu_{-j}(X^{-j}):=\e{X^j|X^{-j}}$ is a linear model.

As shown in Lemma~\ref{lemma:LMX_without_j}, $y \mid X^{-j}$ is also linear. Consequently, $\widehat{m}$, $\widehat{m}_{-j}$, and $\widehat{\nu}_{-j}$ are linear, so LOCO and CPI have the same computational complexity. However, LOCO’s bias decays linearly, whereas CPI’s decays quadratically.

\begin{lemma}\label{lemm:doubl-rob-LM}
Assuming $X^j\indep y\mid X^{-j}$, linear model (\ref{ass:LM}), Gaussian input covariates, and $\widehat{\nu}_{-j}$ and $\widehat{\beta}$ trained in different samples, then $\e{\widehat{\psi}_\mathrm{CPI}(j)}=O(1/n_\mathrm{train}^2).$
\end{lemma}
%
%
We note that assuming two training samples has negligible impact on this convergence result: splitting a single sample reduces the rate by a factor of 4, which remains quadratic in $n_\mathrm{train}$. Finally, for LOCO:
%
\begin{lemma}
 Assuming $X^j\indep y\mid X^{-j}$, linear model (\ref{ass:LM}) and Gaussian covariates then $\e{\widehat{\psi}_\mathrm{LOCO}(j)}=O(1/n_\mathrm{train}).$
\label{lemma:bias-LOCO-linear}
\end{lemma}
Therefore, in this linear setting the double robustness translates as a faster convergence rate for CPI.

\section{Sobol-CPI}\label{sect:sobol-CPI}
We aim to obtain a stable estimate of TSI (\ref{def:TSI}). As noted in \cref{subsec:setting}, the link between $y$ and $X^{-j}$ can be complex, often requiring computationally intensive models. Furthermore, as discussed in \cref{subsec:double-robust}, optimization errors do not accumulate efficiently for retraining methods like LOCO. Therefore, we develop a CPI-based approach to estimate TSI. The key idea is to leverage the law of total expectation, because
\begin{align*}
    m_{-j}(X^{-j})&=\e{y\mid X^{-j}}=\e{\e{y\mid X}\mid X^{-j}}\\&=\e{m(X)\mid X^{-j}}=\e{m(\widetilde{X}^{(j)})\mid X^{-j}},
\end{align*}
 where $\widetilde{X}^{(j)}\sim \mathcal{L}(X\mid X^{-j})$.
 We propose to compute
$
1/n_\mathrm{cal} \sum_{i=1}^{n_\mathrm{cal}} \widehat{m}\left( \widetilde{x}^{'(j)}_i \right),
$
where $\{\widetilde{x}^{'(j)}_i\}_{i=1, \dots, n_\mathrm{cal}}$ are sampled from the estimated conditional distribution $P'_j$. This can be easily achieved using CPI's conditional sampling, proven valid in \cref{subsec:cond_sampl}. We thus only need to train a regressor $\widehat{\nu}_{-j}$ and add $n_\mathrm{cal}$ residuals to obtain the estimate.

This idea is not limited to regression settings. Indeed, as discussed in \citet{williamson2023general}, most Bayes predictors are functions of the conditional expectation. 
For example, under the classical 0-1 loss, the Bayes classifier is given by $\mathbb{I}_{\mathbb{E}[y|X] \geq 0.5}$. 
In this case, for the restricted model, we propose
$
\mathbb{I}_{1/n_\mathrm{cal} \sum_{i=1}^{n_\mathrm{cal}} \widehat{m}\left( \widetilde{x}^{(j)}_i \right) \geq 0.5},
$
rather than refitting another model for $\widehat{m}_{-j}$.
Then, we define the general Sobol-CPI as this TSI plug-in combination of the conditional sampling total's expectation estimate of $m_{-j}$:
\begin{definition}[Sobol-CPI] Given a coordinate $j$, a calibration set size $n_\mathrm{cal}$, a regressor $\widehat{m}$ of $y$ given $X$ and a test set $\{(X_i, y_i)\}_{i=1,\ldots n_\mathrm{test}}$ , Sobol-CPI is defined as
\begin{align*}
\hspace*{-2mm}
    \widehat{\psi}_{\mathrm{SCPI}}^j&=\frac{1}{n_{\mathrm{test}}}\sum_{i=1}^{n_{\mathrm{test}}}\ell\left(\frac{1}{n_\mathrm{cal}}\sum_{k=1}^{n_\mathrm{cal}}\widehat{m}(\widetilde{x}^{'(j)}_{i,k}),y_i\right)-\ell\left(\widehat{m}(x_i), y_i\right)
\end{align*}
where the $j$-th coordinate of $\widetilde{x}'^{(j)}_{i,k}$ is sampled from an estimated conditional distribution $P'_j$, with the remaining coordinates fixed to $x_i^{-j}$.
\end{definition}
We observe that using the law of total expectation is related to the \textit{marginalization} of the conditional SAGE value functions (cSAGEvf, \citet{covert2020}), while making explicit the fact that both the \textit{conditional density} and the \textit{conditional expectation} are unknown and must estimated. Regarding the first issue, we study in particular the sampling step from the CPI and propose an asymptotic efficiency result. We further observe that, in order to extend this result to cSAGEvf, it is necessary to specify the sampler explicitly. For the second issue, we propose a bias-correction procedure that yields an unbiased estimator without incurring additional cost.

\subsection{Asymptotic efficiency}\label{subsec:asymp_eff}
%
%


Under the assumptions of \citet{williamson2023general} (discussed in \cref{sect:proofTheo}) and an additional Lipschitz condition ensuring local robustness to small input changes, we achieve the same Sobol-CPI's nonparametric efficiency.

\begin{theorem}\label{th:asymp_eff}
Assuming additive innovation (\ref{ass:generalAssumpCondSampl}), that $m$ is Lipschitz and assumptions A1-A4 and B1-B3 in \cref{sect:proofTheo},  $\widehat{\psi}_\mathrm{SCPI}^j$ is nonparametric efficient.
\end{theorem}


Among the assumptions, those on the loss function have been shown by \citet{williamson2023general} to hold for standard losses, such as the quadratic loss or classification accuracy. We also require the usual $O(n^{-1/4})$ convergence rate for $\widehat{m}$ and $\widehat{\nu}_{-j}$, which is standard in semiparametric inference. Note that this result holds for any conditional sampler $P'_j$ such that $\|\widetilde{X}'^{(j)}-\widetilde{X}^{(j)}\|=O_P(n^{-1/4})$. 


%
%
\subsection{Fixing $n_\mathrm{cal}$ for the quadratic loss $\ell_2$}\label{subsec:fix-n-cal}
In practice, we need to decide on a finite value for $n_\mathrm{cal}$. Doing so introduces a trade-off between \textit{variable selection} and \textit{variable importance}. Indeed, as discussed in \cref{subsec:exp-n-cal}, with a small $n_\mathrm{cal}$, we recover the double robustness of CPI, achieving better results for variable selection but losing nonparametric efficiency, leading to slightly worse results for variable importance. This is because the optimality assumption (Assumption A1 in \cref{sect:proofTheo}), which controls the first-order bias due to the need to estimate $m$, no longer holds.
In any case, fixing $n_\mathrm{cal}$, as is commonly done in practice for SAGE \cite{covert}, induces a bias:

\begin{proposition}[Bias of $\widehat{\psi}_\mathrm{SCPI}$]\label{prop:fix_n_cal} For $n_\mathrm{cal}<\infty$, under additive noise (\ref{ass:generalAssumpCondSampl}) and consistency of $\widehat{m}$ and $\widehat{\nu}_{-j}$, then $$\widehat{\psi}_{\mathrm{SCPI}}^j\xrightarrow{n_\mathrm{train}, n_\mathrm{test}\to\infty}\left(1+\frac{1}{n_{\mathrm{cal}}}\right)\psi_\mathrm{TSI}(j, P_0).$$
\end{proposition}
This bias can be corrected by scaling by $(1+1/n_\mathrm{cal})^{-1}$:
\begin{definition}[Sobol-CPI($n_\mathrm{cal}$)] Given a coordinate $j$, a fixed $n_\mathrm{cal}$, a regressor $\widehat{m}$ of $y$ given $X$ and a test set $\{(X_i, y_i)\}_{i=1,\ldots n_\mathrm{test}}$ , Sobol-CPI($n_\mathrm{cal}$) is defined as
\begin{align*}
    \widehat{\psi}&_{\mathrm{SCPI}(n_\mathrm{cal})}^j=\frac{n_\mathrm{cal}}{n_\mathrm{cal}+1}\frac{1}{n_{\mathrm{test}}}*\\&\sum_{i=1}^{n_{\mathrm{test}}}\left[\left(\frac{1}{n_\mathrm{cal}}\sum_{k=1}^{n_\mathrm{cal}}\widehat{m}(\widetilde{x}^{'(j)}_{i,k})-y_i\right)^2-\left(\widehat{m}(x_i)- y_i\right)^2\right],
\end{align*}
where the $j$-th coordinate of $\widetilde{x}'^{(j)}_{i,k}$ is sampled from an estimated conditional distribution $P'_j$, with the remaining coordinates fixed to $x_i^{-j}$.
%
\end{definition}
In particular, by taking $n_\mathrm{cal}=1$, we recover the standard CPI, but divided by 2. With this simple correction, the method works effectively for variable selection, as we directly recover double robustness. Moreover, it converges to TSI, even though it is not efficient (see \cref{subsec:exp-n-cal}). This result also corrects Theorem 2 from \citet{Hooker2021}, which claimed that Dropped and Conditional Variable Importance coincide, highlighting the gap in the literature between the two approaches.

\section{Confidence intervals and Inference}\label{subsec:inference}
As stated in Theorem \ref{th:asymp_eff}, there is no need for a one-step estimate to correct any first-order bias. 
The same holds for LOCO, which does not require such corrections when estimated as a plug-in estimate in the difference of generalized ANOVA \cite{williamson2021} or as a difference of predictiveness measures more generally \cite{williamson2023general}. 
Whenever the importance is not null, it is possible to construct confidence intervals using the influence functions. 
Indeed, the variance is given by $0 < \tau_j^2 := \mathbb{E}[\varphi_j^2(z)] < \infty$, where $\varphi_j$ is the influence function of $\psi_\mathrm{TSI}(j, P_0)$. 
Therefore, it is possible to estimate it using a plug-in estimate. 
Moreover, as shown in Appendix \ref{sect:MSE_if}, when taking the MSE as the predictiveness measure, the variance estimated with this plug-in method matches the natural empirical variance. 

Nevertheless, under the null hypothesis, the influence function vanishes, preventing a Gaussian asymptotic distribution. This poses a challenge, as it hinders variable selection with direct statistical guarantees—an essential component of reliable scientific discovery.
This issue is the same as that faced by LOCO and Shapley due to the quadratic functionals \cite{locoVSShapley}.

To address this problem, \citet{williamson2023general} noted that even if the influence function of LOCO vanishes, excluding extreme cases, the influence functions of the predictiveness measures with and without the covariate do not vanish. 
They attempted to leverage this observation by computing each predictiveness measure on different data splits. 
However, as observed in numerical experiments in \cref{sect:exps}, this is ineffective because it increases both variability and bias, resulting in a loss of power and poor estimates.


Other alternatives include inflating the confidence interval by an additive term of the order $O_P(\sqrt{n})$ (\citet{Dai2024,locoVSShapley}).
Validity can then be obtained using Chebyshev's inequality, resulting in a confidence interval of the form $(-\infty, \widehat{\psi}_n + z_\alpha \mathrm{se}_n + c / \sqrt{n}]$, where $\mathrm{se}_n$ is the empirical standard error and $c$ is any constant. 
In practice, \citet{locoVSShapley} takes $c$ as the standard error of the output. 

We first propose a similar variance correction improving the rates using our previous results, and then introduce a finite-sample Wilcoxon test that also accounts for the estimation of the conditional distribution.

\paragraph{Variance correction:} Using \cref{th:asymp_eff}, we establish the consistency (power tending to $1$) of this conditional hypothesis test with the additive correction (see Appendix~\ref{app:inf_lin}). Moreover, using Lemma~\ref{lemm:doubl-rob-LM}, we propose applying Markov's inequality to construct a more powerful test, achieving a rate of \( c/n^{\gamma} \) for \( \gamma < 2 \) instead of the standard \( c/\sqrt{n} \) (Appendix \ref{app:inf_lin}). The limitations regarding the extent of interval expansion are discussed in Appendix~\ref{subsec:app_inf}, both for linear and nonlinear settings.
%

\paragraph{Nonparametric tests:} Variance-correction procedures rely on asymptotic convergence rates, which are typically not tight and do not provide finite-sample guarantees. To address this, in addition to the corrections motivated by our convergence results from previous sections, as detailed in Algorithm~\ref{alg:scpi_wcx}, we propose using a paired nonparametric test, such as the Wilcoxon or Sign test, between
\[
        \left\{l\bigl(\widehat{m}(X_{i}), y_i\bigr)\right\}_i^{n_\mathrm{test}}
        \text{ and }
        \left\{l\bigl(\widehat{m}(\widetilde{X}'^{(j)}_{i}), y_i\bigr)\right\}_i^{n_\mathrm{test}}.
    \]
This yields a p-value $p'$. Any conditional sampler may be used; if the theoretical conditional was available, this would provide exact finite-sample type-I error. Since the conditional must be estimated in practice, we provide a bound on the error based on the Total Variation distance $d_\mathrm{TV}$ between the estimated and true conditional distributions. For more details on this distance, see e.g.  \citet{Berrett}.
\begin{theorem}\label{th:wcx_TV} Let $P^\star_j$ denote the conditional distribution $\mathcal{L}(X^j \mid X^{-j})$, estimated by any $P'_j$. Then, for any level $\alpha \in [0,1]$, the p-value $p'$ produced by Algorithm~\ref{alg:SCPI_wcx} satisfies
\begin{align*}
\mathbb{P}_{\mathcal{H}_0}(&p'\leq \alpha\mid \{X^{-j}_i, y_i\}_i^{n_\mathrm{test}}, \mathcal{D}_\mathrm{train})\\&\leq \alpha +d_\mathrm{TV}(P'_j(\cdot\mid X^{-j}), P^\star_j(\cdot\mid X^{-j})).    
\end{align*}
\end{theorem}
%
%
\section{Experiments}\label{sect:exps}
%
We study both simulated and real-data settings. In the former, the TSI can be computed explicitly (see \cref{sect:expl_TSI}). All simulations are run over at least $50$ repetitions. The code is available at
\url{https://github.com/AngelReyero/Sobol-CPI}.


All the proposed theory applies to any conditional sampler. To illustrate this, we implement Sobol-CPI with the regression sampler (\ref{subsec:cond_sampl}), which we denote by \texttt{Sobol-CPI} and with a Gaussian sampler from the \texttt{fippy} library: \texttt{CFI} \cite{Strobl2008} and the conditional SAGE value function (\texttt{scSAGEvf}, \citet{covert}). These correspond exactly to \texttt{Sobol-CPI(1)} and \texttt{Sobol-CPI(50)}, respectively, but without the bias correction from Proposition~\ref{prop:fix_n_cal}, which we have added. Thus, they are equivalent to \texttt{Sobol-CPI} with a different sampler.  We also compare to \texttt{LOCO-W} from \citet{williamson2023general}, and \texttt{LOCO-HD} from \citet{locoVSShapley}. Our focus is on bias in estimating important (efficiency) and non-important covariates (double robustness), as well as on enabling powerful, type-I-error-controlled variable selection. For a broader discussion on how $n_\mathrm{cal}$ balances nonparametric efficiency and double robustness, see \cref{subsec:exp-n-cal}.

\paragraph{Complex learners:} in Figure \ref{fig:comp-learner} we study a nonlinear setting similar to \citet{scornet2022mda}: $ y = X_0 X_1 \mathbb{I}_{X_2 > 0} + 2 X_3 X_4 \mathbb{I}_{X_2 < 0} $, with $X \sim \mathcal{N}(\mu, \Sigma) $, $ \Sigma_{i,j} = 0.6^{|i-j|}$, $p = 50 $, and $\mu = \mathbf{0}$. 
We use Lasso for $\widehat{\nu}_{-j}$ and a SuperLearner \cite{van2007super} of Random Forests, Lasso, Gradient Boosting, and SVM for $\widehat{m}$ and $\widehat{m}_{-j}$. 
For computational reasons, in the data-splitting \texttt{LOCO-W}, we use only Gradient Boosting. We apply Sobol-CPI with $n_\mathrm{cal}=1$ and $100$.

In this setting, since $X$ is Gaussian, Sobol-CPI with a Gaussian sampler (\texttt{CFI} and \texttt{scSAGEvf}) performs very well. Even using a complex model to estimate $\widehat{m}_{-j}$ but a simpler one for $\widehat{\nu}_{-j}$, Sobol-CPI not only achieves more computationally efficient estimates but also outperforms LOCO.  

From the first plot, perturbation (\texttt{Sobol-CPI(1)}/\texttt{CFI}), removal (\texttt{LOCO-W}/\texttt{HD}), and marginalization (\texttt{scSAGEvf}/ \texttt{Sobol-CPI(100)}), converge to the TSI, justifying their comparison and the study of their statistical properties.  

From the second plot, which highlights a null feature, and the fourth plot, where all null-feature importances are aggregated, the double robustness of Sobol-CPI is evident: refitting-based methods exhibit higher variability. This plot also shows \texttt{PFI}, which is omitted from the first plot since it does not estimate TSI. \texttt{PFI} correctly assigns no importance to the null features (second and fourth plots) and correctly identifies important features (third plot). These results shows that, even if the sampler is misspecified, the asymptotic relevance assumption (\ref{ass:asymp_relevance}) still holds. 
%
%
%
%
\begin{figure}
    \centering
    \begin{minipage}{0.5\textwidth}
        \includegraphics[width=\textwidth]{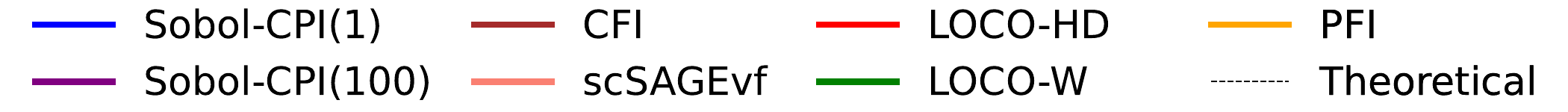}
    \end{minipage}
    \makebox[\textwidth][l]{%
    \hspace{-0.5em}%
    \begin{minipage}{0.5\textwidth}
        \includegraphics[width=\textwidth, height=0.27\textwidth]{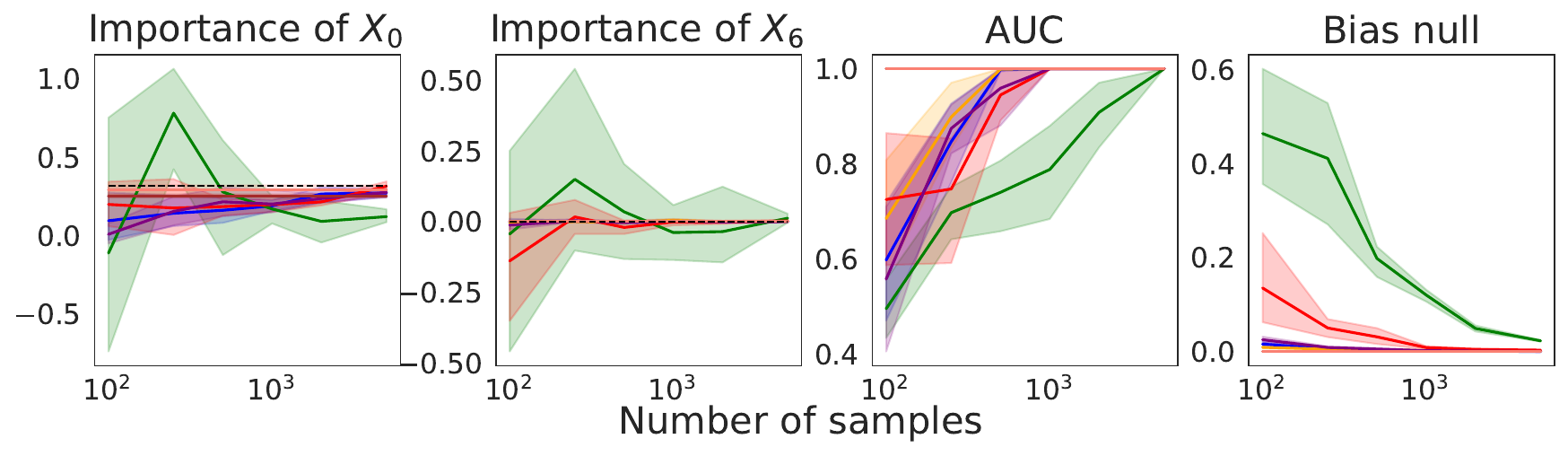}
    \end{minipage}
    }
    \caption{\textbf{Double robustness for complex learners:} left to right: TSI estimates for an important covariate ($X_0$) and a null covariate ($X_6$); AUC for an importance-based variable selection; bias for null covariates.
    }
    \label{fig:comp-learner}
\end{figure}
\paragraph{Inference:} we study a linear setting with important covariates uniformly sampled with sparsity $0.25$, $X \sim \mathcal{N}(\mu, \Sigma) $, where $ \Sigma_{i,j} = 0.6^{|i-j|}$, $p = 100 $, and $\mu = \mathbf{0}$.

As discussed in  \cref{subsec:inference}, to test the null hypothesis, a correction is necessary to address the variance decrease under the null hypothesis. 
%
%
%
%
In \cref{fig:inf}, we present the power with a linear correction term. 
Other corrections, along with analyses of type-I error, computation time, and additional settings (e.g., different correlations and polynomial settings), are discussed in \cref{subsec:app_inf}. 
%
Note that \texttt{LOCO-W} lacks power and requires a large sample size for type-I error control. 
In addition, Sobol-CPI also reduces the bias for the non-null covariates and leverages its double robustness property to achieve the largest power. 
Finally,  the type-I error is controlled at the target level $0.05$ across all the procedures but the data-splitting version of \citet{williamson2023general}. 
\begin{figure}[htbp]
    \centering
     \centering
    \begin{minipage}{0.5\textwidth}
        \includegraphics[width=\textwidth]{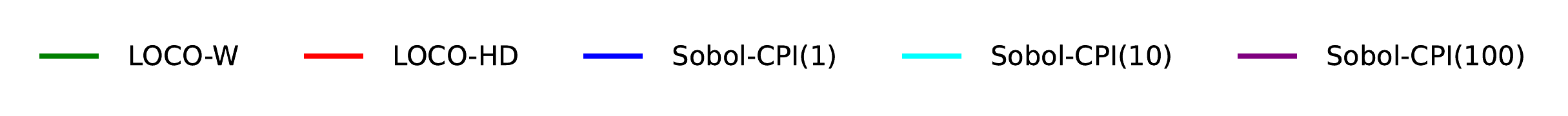}
    \end{minipage}
    \makebox[\textwidth][l]{%
    \hspace{-0.5em}%
    \begin{minipage}{0.5\textwidth}
        \includegraphics[width=\textwidth, height=0.27\textwidth]{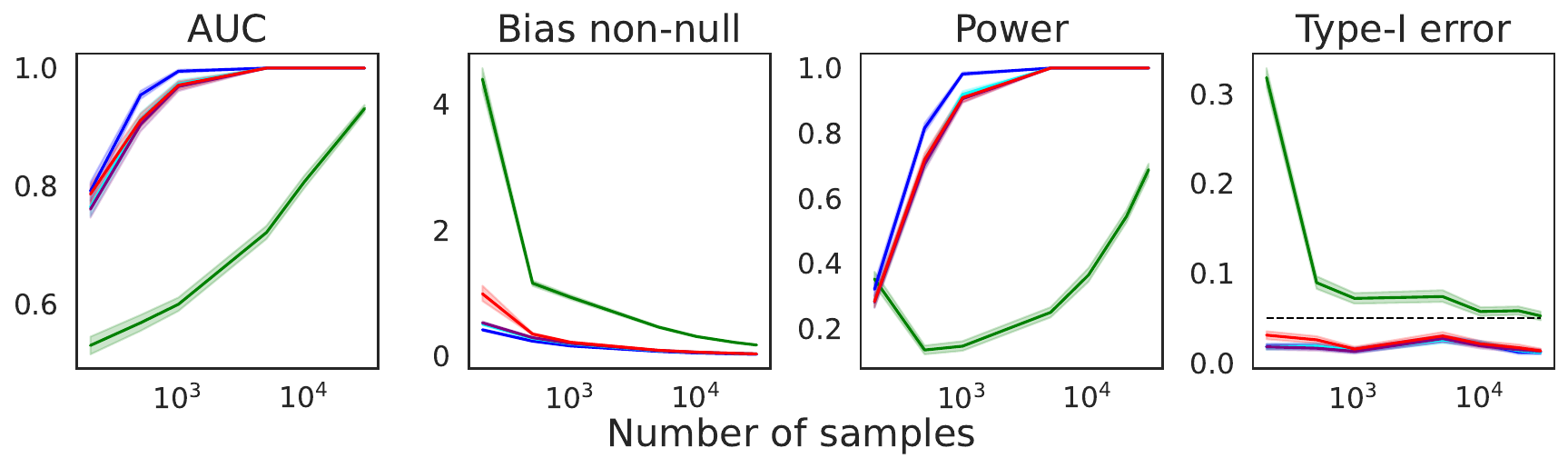}
    \end{minipage}
    }    
    \caption{\textbf{Statistical Inference on variable importance} in a linear setting: AUC for variable selection accuracy, bias in non-null TSI estimation, power and  type-I error. Sobol-CPI(1) provides the most powerful test. Using the $O(1/n)$-corrected variance, the type-I error is controlled.}
    \label{fig:inf}
\end{figure}%
\paragraph{Real data:} we evaluate several standard ML models and datasets (see Section~\ref{app:datasets}); since conclusions were consistent, we report additional results in Appendix~\ref{app:real_data}. Here we present the Wisconsin Breast Cancer (WBDC) dataset with a SuperLearner. As the set of true important features is unknown in real data, we add an artificial null feature constructed as a noisy function of the original features, and vary its correlation with them (x-axis). This allows us to estimate null importance and type-I error on the artificial feature, while performing discoveries on the true inputs.  

In Figure \ref{fig:real_data} we compare the $p$-value from \texttt{LOCO-W}, Sobol-CPI(1) with the Wilcoxon test, and the other procedures with a linear correction. This correction is conservative and yields no discoveries, whereas Sobol-CPI(1)-Wilcoxon \texttt{SCPI(1)-wcx} is powerful while still controlling the error. Also, observe that the double robustness holds in practice.

\begin{figure}
    \centering
     \centering
    \begin{minipage}{0.5\textwidth}
        \includegraphics[width=\textwidth]{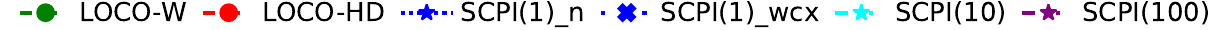}
    \end{minipage}
    \makebox[\textwidth][l]{%
    \hspace{-0.5em}%
    \begin{minipage}{0.51\textwidth}
        \includegraphics[width=\textwidth, height=0.27\textwidth]{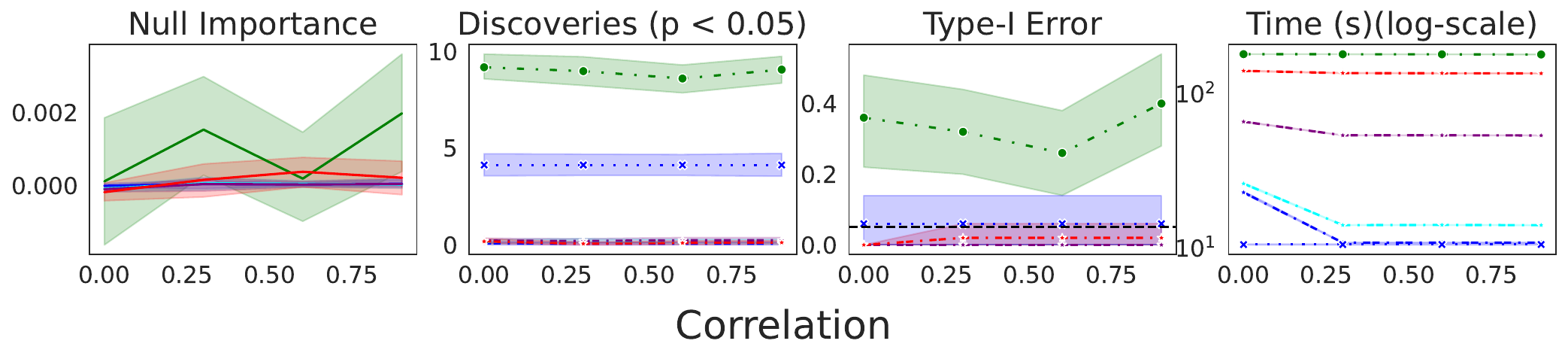}
    \end{minipage}
    }
    \caption{\textbf{Double robustness and inference on real data:} Using WBDC dataset with an added artificial null feature of varying correlation (x-axis). SCPI-based methods assign no importance to the null. With Wilcoxon test (Alg.~\ref{alg:SCPI_wcx}), discoveries are made while controlling the error and remaining computationally efficient.}
    \label{fig:real_data}
\end{figure}%
\section{Conclusion}


In this article, we revisited Conditional Feature Importance (CFI). We first studied the widely used regression-based estimator of the conditional distribution. We then analyzed both the null and alternative regimes. To explain the strong empirical performance of CFI in correctly identifying null features, we uncover a \textbf{double robustness} property: the method benefits jointly from the conditional sampler and the predictive model. Asymptotically, the sampler draws from the input distribution, so the two sources of error compensate by sharing the same distribution, while the model typically induces implicit variable selection during optimization.

Under the alternative regime, we modify CFI and introduce Sobol-CFI, clarifying its population target and proving \textbf{nonparametric efficiency}. We further address bias by proposing a correction that establishes a unified link between LOCO (refitting), CFI (perturbation), and SAGE value functions (marginalization). Finally, we show how to perform valid \textbf{inference} within this framework. Overall, CFI emerges as a powerful, statistically controlled variable importance measure, and this work significantly advances the understanding of the effect of conditional sampler estimation.


\section*{Impact Statement}

This paper presents work whose goal is to advance the field of Machine
Learning. There are many potential societal consequences of our work, none
which we feel must be specifically highlighted here.


\bibliography{biblio}
\bibliographystyle{icml2026}

\newpage
\appendix
\onecolumn

\section{Notation Glossary}\label{app:glossary}

\begin{longtable}{p{4cm} p{12cm}}
\textbf{Symbol} & \textbf{Description} \\
\hline
$X\in \mathbb{R}^p$ & Input\\
$X^j\in \mathbb{R}$ & $j$-th input covariate\\
$X^{-j}\in \mathbb{R}^{p-1}$ & Input with the $j$-th covariate excluded\\
$y\in \mathbb{R}$ & Output\\
$x_i$ & $i$-th individual\\
$x_i^{(j)}$ & $i$-th individual with permuted $j$-th covariate\\
$m(X)$ (resp. $m_{-j}(X^{-j})$) & $\e{y\mid X}$ (resp. $\e{y\mid X^{-j}}$)\\
$\widehat{m}$ (resp. $\widehat{m}_{-j}$) & Estimation of $m$ (resp. of $m_{-j}$) \\
$\nu_{-j}(X^{-j})$ & $\e{X^j\mid X^{-j}}$\\
$\widehat{\nu}_{-j}$ & Estimation of $\nu_{-j}$\\
$P^\star_j$ & Conditional distribution of $X^j$ given $X^{-j}$\\
$P'_j$ & Estimated conditional distribution of $X^j$ given $X^{-j}$\\
$\widetilde{X}^{(j)}$ & Random variable drawn according to $P^\star_j$\\
$\widetilde{X}^{'(j)}$ & Random variable drawn according to $P'_j$\\
$\widetilde{x}^{'(j)}_i$ & $i$-th observation of $\widetilde{X}^{'(j)}$\\
$\widetilde{x}^{'(j)}_{i,k}$ & $i$-th observation with the added residual from $x_k$\\
$\ell$ & Loss function\\
$\mathcal{W}_2$ & 2-Wasserstein distance\\
$\mathcal{F}$ & Generic space of functions\\
$\overset{\mathcal{L}}{\to}$ & Convergence in law\\
$\mathcal{D}_\mathrm{train}$ & Train data set\\
$n_\mathrm{train}$ & Size of train set\\
$n_\mathrm{test}$ & Size of test set\\
$n_\mathrm{cal}$ & Size of calibration set\\
$O(R_n)$ & Bounded at a rate of $R_n$. 
\end{longtable}

The superscript $(j)$ is omitted to avoid index overload when $j$ can be inferred from the context. Also, some notation can be combined; for instance, $\widetilde{x}^{'(j),l}_{i,k}$ denotes the $l$-th coordinate of $\widetilde{x}^{'(j)}_{i,k}$.

\section{Permutation Feature Importance (PFI)}\label{sect:PFI}

We denote by $x_{i}^{(j)}$ the vector where all coordinates come from the $i$-th observation, except for the $j$-th coordinate, which is taken from another observation, as the column has been completely shuffled. Thus, it is given by:

\begin{definition}[PFI] Given a coordinate $j$, a loss $\ell$, a regressor $\widehat{m}$ of $y$ given $X$ and a test set $(X_i, y_i)_{i=1,\ldots n_\mathrm{test}}$, PFI is defined as
\begin{align*}
    \widehat{\psi}_\mathrm{PFI}^j=\frac{1}{n_{\mathrm{test}}}\sum_{i=1}^{n_{\mathrm{test}}}\ell\left(\widehat{m}(x_{i}^{(j)}),y_i\right)-\ell\left(\widehat{m}(x_{i}),y_i\right),
\end{align*}
where the $j$-th coordinate is permuted.
\end{definition}

\section{Conditional sampling proofs}
\subsection{Proof of Lemma \ref{lemma:GaussianAss}}
We begin by decomposing each column as $$X^j=\e{X^j\middle|X^{-j}}+\left(X^j-\e{X^j\middle|X^{-j}}\right).$$ 
 We can denote $\nu_{-j}(X^{-j}):=\e{X^j\middle|X^{-j}}$ and $\epsilon_j:=X^j-\e{X^j\middle|X^{-j}}$. First, note that they are both Gaussian as $\e{X^j|X^{-j}}=\mu_{j}+\Sigma_{j, -j}\Sigma^{-1}_{-j, -j}(X^{-j}-\mu_{-j})$, therefore they are linear combinations of coordinates of a Gaussian vector.
 
 Then, note that $\epsilon_j$ is centered. Finally, to see that they are independent, as they are both Gaussian variables, we just need to prove that their covariance is null:
 \begin{align*}
     \mathbb{E}&\left[\left(\nu_{-j}\left(X^{-j}\right)-\e{X^{j}}\right)\left(X^j-\e{X^j\middle|X^{-j}}\right)\right]\\&=
     \e{\e{\left(\nu_{-j}\left(X^{-j}\right)-\e{X^{j}}\right)\left(X^j-\e{X^j\middle|X^{-j}}\right)\middle|X^{-j}}}\\&=\e{\left(\nu_{-j}\left(X^{-j}\right)-\e{X^{j}}\right)\e{X^j-\e{X^j\middle|X^{-j}}\middle|X^{-j}}}=
     0.
 \end{align*}

\subsection{Proof of Proposition \ref{prop:empCondSamp}}

\begin{proof}
   In this proof, for simplicity of notation, instead of referring to two random variables $X_i$ and $X_k$, we will denote them by $X$ and $X'$. We will first observe that, under \cref{ass:generalAssumpCondSampl}, $P^\star_j$ can be decomposed as the theoretical counterpart of \eqref{eq:cpi-sampling}. Then, using the consistency of $\widehat{\nu}_{-j}$, we show that the Wasserstein distance vanishes.
     We first observe that for $X'\overset{\mathrm{i.i.d}}{\sim}X$ we have that $X'^j-\nu_{-j}\left(X'^{-j}\right)\overset{\mathrm{i.i.d}}{\sim}\epsilon_j$.
 
 Also, we have that $X^j|\left(X^{-j}=x^{-j}\right)= \epsilon_j+\nu_{-j}(x^{-j})\overset{{\mathrm{i.i.d}}}{\sim}\left(X'^j-\nu_{-j}\left(X'^{-j}\right)\right)+\nu_{-j}(x^{-j})$, which is exactly the theoretical CPI sampling step, i.e. first we compute the conditional expectation $\nu_{-j}(x^{-j})$ with the regressor $\nu_{-j}$ and then we add a permuted residual $\left(X'^j-\nu_{-j}\left(X'^{-j}\right)\right)$. Then, for $\widetilde{X}^{(j)}\sim P^\star_j$, we have that $X^j\overset{\mathrm{i.i.d.}}{\sim}\widetilde{X}^{(j)}|X^{-j}$. We obviously have $X^{-j}=\widetilde{X}^{(j)-j}$. Finally, to ensure that $\widetilde{X}^{(j)} \indep y \mid X^{-j}$, we can rely on the fact that, similar to the construction of knockoffs in \citet{candes2017panninggoldmodelxknockoffs}, $\widetilde{X}^{(j)}$ is constructed without using $y$.

    To show that the estimated distribution converges to the theoretical one, first recall that the usual 2-Wasserstein distance is given by 
    $$\left(\underset{P_\theta\in \Theta (\mu, \nu)}{\mathrm{inf}}\int \|x-y\|^2\mathrm{d}P_\theta (\mathrm{dx,dy})\right)^{\frac{1}{2}},$$
    where $\Theta(\mu, \nu)$ is the set of distributions with marginals $\mu$ and $\nu$. 
    
    We have that $P^\star_j$ has the same distribution as $\e{X^j\middle|X^{-j}}+\left(X'^{j}-\e{X'^j|X'^{-j}}\right)$ and the empirical couterpart is given by $P'_j:=\widehat{\nu}_{-j}\left(X^{-j}\right)+X'^{j}-\widehat{\nu}_{-j}\left(X'^{-j}\right)$. Now, we bound the distance between them:
    
    \begin{align*}
        W_2(P^\star_j, P'_j)&=\left(\underset{P_\theta\in\Theta(P^\star_j, P'_j)}{\mathrm{inf}}\int_{\mathbb{R}^{2p}\times \mathbb{R}^{2p}}(x-y)^2P_\theta(\mathrm{dx, dy})\right)^{\frac{1}{2}}\\
        &\leq \left(\int_{\mathbb{R}^{2p}}\left(\e{X^j\middle|X^{-j}}+\left(X'^j-\e{X'^j\middle|X'^{-j}}\right)-\widehat{\nu}_{-j}\left(X^{-j}\right)\right.\right.\\&\left.\left.\qquad-\left(X'^j-\widehat{\nu}_{-j}\left(X'^{-j}\right)\right)\right)^2P_X(\mathrm{dx})P_{X'}(\mathrm{dx}')\right)^{\frac{1}{2}}\\
        &= \left(\int_{\mathbb{R}^{2(p-1)}}\left(\nu_{-j}\left(X^{-j}\right)-\nu_{-j}\left(X'^{-j}\right)\right.\right.\\&\left.\left.\qquad-\widehat{\nu}_{-j}\left(X^{-j}\right)+\widehat{\nu}_{-j}\left(X'^{-j}\right)\right)^2P_{X^{-j}}(\mathrm{dx}^{-j})P_{X'^{-j}}(\mathrm{dx}'^{-j})\right)^{\frac{1}{2}}\\
        &\leq \left(\e{\left(\nu_{-j}\left(X^{-j}\right)-\widehat{\nu}_{-j}\left(X^{-j}\right)\right)^2}\right)^{\frac{1}{2}}+\left(\e{\left(\nu_{-j}\left(X'^{-j}\right)-\widehat{\nu}_{-j}\left(X'^{-j}\right)\right)^2}\right)^{\frac{1}{2}}.
    \end{align*}
    We conclude using that both terms converge to $0$ by the consistency of the regressor. 
\end{proof}

\section{Proofs of the double robustness}

\subsection{Proof of Theorem \ref{th:doubleRobust}}
We assume that $X^j \indep y \mid X^{-j}$. We start by assuming that $\widetilde{X}'^{(j)} \overset{\ell}{\to} \widetilde{X}^{(j)}$. 
We have that for any bounded and continuous $\ell$ and continuous function $f$, 
\begin{align*}
\left|\widehat{\psi}_\mathrm{CPI}(j)\right|&=\left| \frac{1}{n_{\mathrm{test}}}\sum_{i=1}^{n_{\mathrm{test}}}\ell\left(f(\widetilde{x}_{i}'^{(j)}),y_i\right)-\ell\left(f(x_{i}), y_i\right)\right|\\
&\leq \left| \frac{1}{n_{\mathrm{test}}}\sum_{i=1}^{n_{\mathrm{test}}}\ell\left(f(\widetilde{x}_{i}'^{(j)}),y_i\right)-\ell\left(f(x_{i}), y_i\right)-\e{\ell\left(f(\widetilde{X}'^{(j)}),y\right)-\ell\left(f(X), y\right)}\right|\\
&\qquad+\left|\e{\ell\left(f(\widetilde{X}'^{(j)}),y\right)-\ell\left(f(X), y\right)}\right|.
\end{align*}
The first term converges to 0 due to the Law of Large Numbers. The second term converges to 0 because of Portmanteau lemma, because under the null hypothesis $j\in \mathcal{H}_0$, $X^j \indep y \mid X^{-j}$, then $(\widetilde{X}^{(j)},y)\sim (X, y)$. In particular $$\ell\left(f(X), y\right)\sim \ell\left(f(\widetilde{X}^{(j)}), y\right).$$

On the other hand, we observe that as $X^j\indep y\mid X^{-j}$, then $m(X)$ does not depend on $X^{j}$. Let $\epsilon>0$. First, using the continuity of the loss, we have that for any $\epsilon$, there exists $\delta>0$ such that $$|\ell(\hat{m}(\tilde{x}'_i),y_i)-\ell(\hat{m}(\tilde{x}_i),y_i)|<\epsilon$$ if $|\hat{m}(\tilde{x}'_i)-\hat{m}(\tilde{x}_i)|<\delta$.

Second, using the asymptotic relevance (\ref{ass:asymp_relevance}), for any $\delta>0$, as $j\in \mathcal{H}_0$, there exists $n_0$ such that for any $n>n_0$, $$|\hat{m}_{n}(\tilde{x}'_i)-\hat{m}_{n}(\tilde{x}_i)|<\delta.$$ Therefore, for any $n>n_0$,
\begin{align*}
    \left|\frac{1}{n_\mathrm{test}}\sum\ell(\hat{m}(\tilde{x}'_i),y_i)-\ell(\hat{m}(\tilde{x}_i),y_i)\right|\leq \frac{1}{n_\mathrm{test}}\sum|\ell(\hat{m}(\tilde{x}'_i),y_i)-\ell(\hat{m}(\tilde{x}_i),y_i)|\leq \epsilon.
\end{align*}
 As this was proven for any $\epsilon>0$, we have that $\hat{\psi}_\mathrm{CPI}\to 0$ a.s. 




\subsection{On the asymptotic relevance}\label{subsec:asympt_relev}

In this section, we numerically explore the asymptotic relevance assumption (\ref{ass:asymp_relevance}). This assumption can be interpreted as requiring a model that does not extrapolate under the null hypothesis. The theoretical model does not depend on the $j$-th coordinate because $X^j \indep Y \mid X^{-j}$. This is easily illustrated for the linear model, where the coefficients of the null coordinates converge to zero; hence, imputations on these coordinates are irrelevant, since for sufficiently large $n_0$ the multiplied terms vanish. 

In contrast, most models are not as transparent as the linear model, and one cannot simply verify the absence of dependence on the $j$-th coordinate by inspecting an estimated coefficient. However, this can be assessed model-agnostically using the Permutation Feature Importance (PFI, see Section~\ref{sect:PFI}). The theoretical PFI index was proven to be zero under the null hypothesis in \citet{reyerolobo2025principledapproachcomparingvariable}, as it satisfies their proposed minimal axiom. In Figure~\ref{fig:asymp_relevance}, we show that for a large benchmark of ML models, the empirical PFI converges to zero. This indicates that there is no extrapolation under the null hypothesis, as the models indeed do not rely on the null coordinate.

The data are generated according to
\[
Y = X_1^2 + \epsilon, 
\quad \epsilon \sim \mathcal{N}(0,0.2), 
\quad X \sim \mathcal{N}(0,\Sigma),
\]
and we study the effect of varying the correlation between the null coordinate $X_2$ and the important coordinate $X_1$. In particular, we consider $\Sigma_{1,2} \in \{0, 0.3, 0.6, 0.9\}$. The sample size is $n=1000$, with $n_\mathrm{train}=800$ and $n_\mathrm{test}=200$. All models are implemented using the default specifications of \texttt{scikit-learn} (\cite{scikit-learn}). 

We first observe that linear regression and the Lasso perform poorly, as the quadratic relationship with the output is not captured by their linear specifications. Across all other models, the PFI tends to zero, confirming the absence of functional dependence on the null coordinate. The only exceptions arise at very high correlations, where Support Vector Regression (SVR) and $k$-Nearest Neighbors (KNN) exhibit slight model dependence on the null coordinate.

\begin{figure}[htbp]
    \centering
        \includegraphics[width=1.0\textwidth]{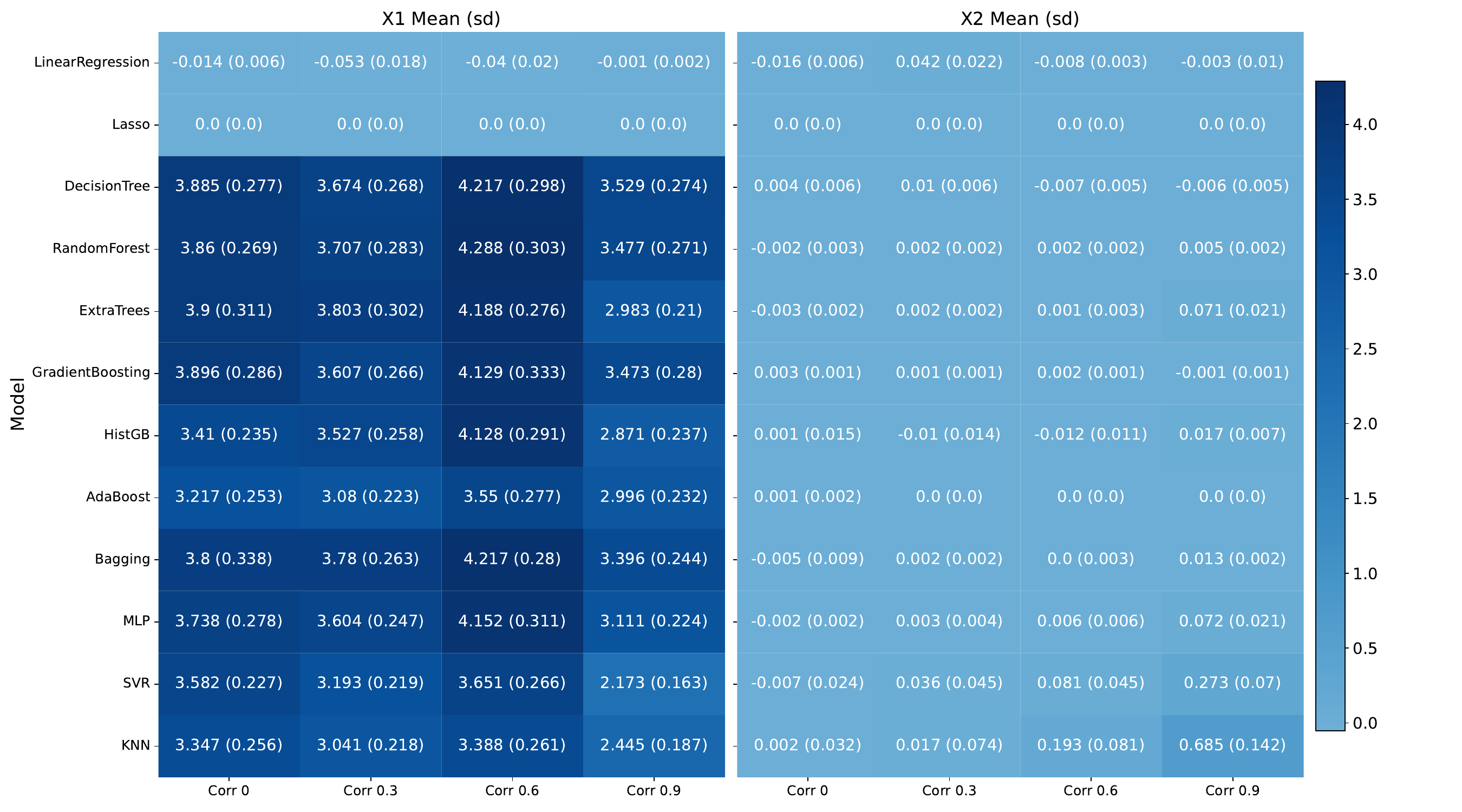}
    \caption{\textbf{Asymptotic relevance on standard ML models:}  
        Permutation Feature Importance (PFI) mean and standard deviation in parenthesis for $X_1$ (left) and $X_2$ (right) across different correlation levels and models, where $Y = X_1^2 + \epsilon, \quad \epsilon \sim \mathcal{N}(0,0.2), \quad X \sim \mathcal{N}(0,\Sigma),$
        with $\Sigma_{1,2} \in \{0, 0.3, 0.6, 0.9\}$. The sample size is $n=1000$, with 80\% used for training and 20\% for testing. The experiment was repeated 100 times. 
        }
    \label{fig:asymp_relevance}
\end{figure}

In Figure~\ref{fig:asymp_relevance_n} we fix the correlation level and observe that, as the sample size increases, all methods reduce their reliance on the null feature. This illustrates an asymptotic property: in practice, with finite samples and real data, we do not recommend using this sampler and instead advocate for conditional samplers. We further explore this assumption on real data in Appendix~\ref{app:asymp_rel_real_data}.

\begin{figure}[htbp]
    \centering
        \includegraphics[width=1.0\textwidth]{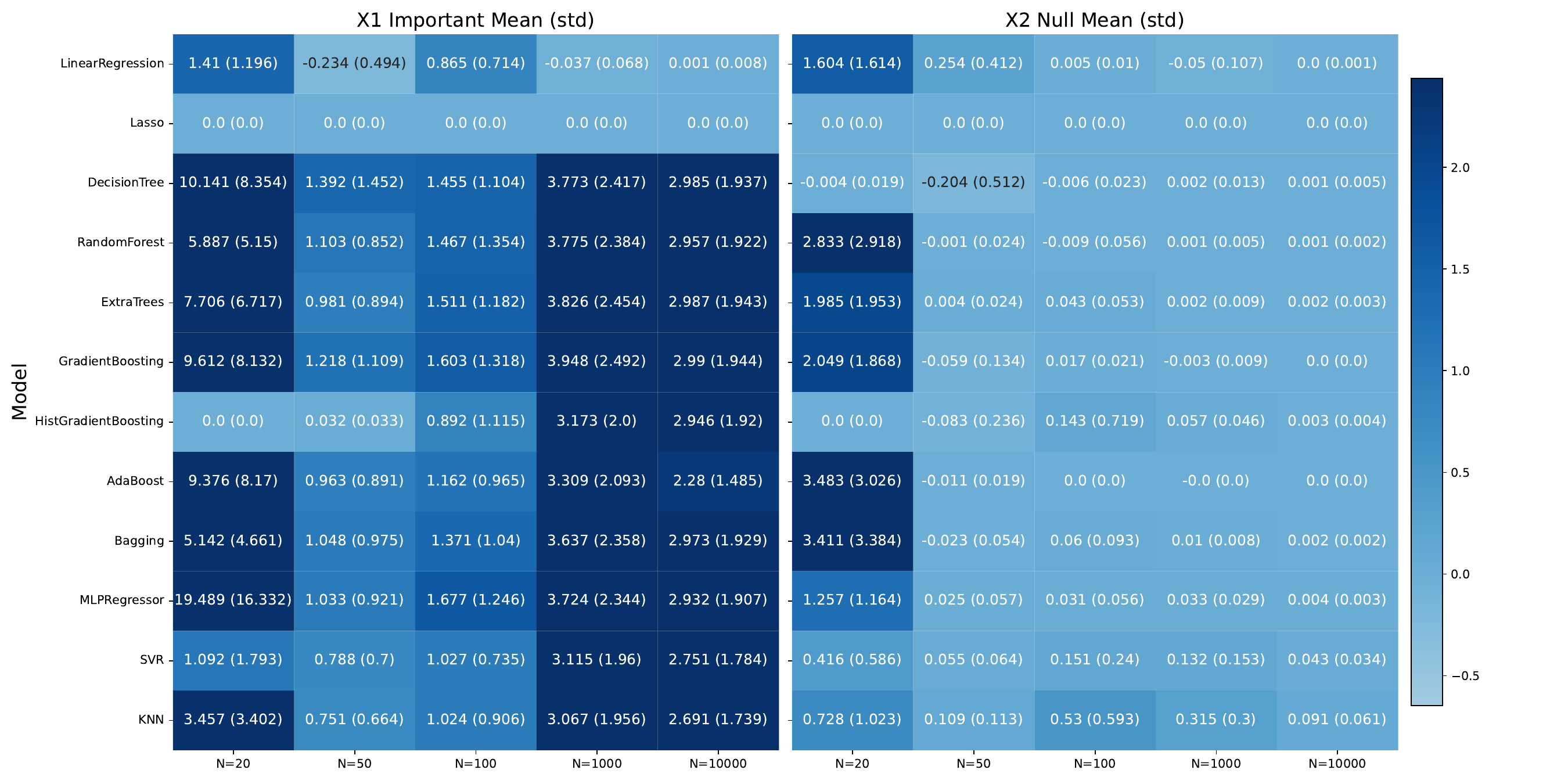}
    \caption{\textbf{Asymptotic relevance on standard ML models:}  
        Permutation Feature Importance (PFI) mean and standard deviation in parenthesis for $X_1$ (left) and $X_2$ (right) across different sample sizes and models, where $Y = X_1^2 + \epsilon, \quad \epsilon \sim \mathcal{N}(0,0.2), \quad X \sim \mathcal{N}(0,\Sigma),$
        with $\Sigma_{1,2}=0.6$. The 80\% of the sample size is used for training and 20\% for testing. The experiment was repeated 100 times. 
        }
    \label{fig:asymp_relevance_n}
\end{figure}

\subsection{Proof of Proposition \ref{prop:cpi-bias}}

\begin{proof}

    As we are under the conditional null hypothesis, using \cref{lemma:nullHyp} we have that there exists a function $m_{-j}\in \mathcal{F}_{-j}$ such that  $m(X)=m_{-j}(X^{-j})$. 
    
    To estimate TSI, as seen in \cref{subsec:fix-n-cal}, we begin by dividing CPI by 2.
    We observe that 
    
    \begin{align}
         \notag\e{\frac{1}{2}\widehat{\psi}_\mathrm{CPI}(j, P_0)\middle|\mathcal{D}_\mathrm{train}}&= 
         \e{\frac{1}{2n_\mathrm{test}}\sum_{i=1}^{n_\mathrm{test}}\left(y_i-\widehat{m}(\widetilde{X}'_i)\right)^2-\left(y_i-\widehat{m}(X_i)\right)^2\middle|\mathcal{D}_\mathrm{train}}\\
         &=\frac{1}{2}\e{\left(y-\widehat{m}(\widetilde{X}')\right)^2-\left(y-\widehat{m}(X)\right)^2\middle|\mathcal{D}_\mathrm{train}}.\label{eq:bias-cpi1}
     \end{align}

    We first note that using \cref{ass:regressModel}, we have that

    \begin{align*}
        \mathbb{E}&\left[\left(y-\widehat{m}(\widetilde{X}')\right)^2-\left(y-\widehat{m}(X)\right)^2\middle|\mathcal{D}_\mathrm{train} \right]\\&=\e{\left(m(X)-\widehat{m}(\widetilde{X}')\right)^2-\left(m(X)-\widehat{m}(X)\right)^2\middle|\mathcal{D}_\mathrm{train}}.
    \end{align*}
    
     We can develop the first term as 
    \begin{align*}
        \mathbb{E}&\left[\left(m(X)-\widehat{m}(\widetilde{X}')\right)^2\middle|\mathcal{D}_\mathrm{train}\right]\\&= \e{\left(m(X)-m(\widetilde{X})\right)^2} + \e{\left(m(\widetilde{X})-\widehat{m}(\widetilde{X}')\right)^2\middle|\mathcal{D}_\mathrm{train}}\\&+2\e{(m(X)-m(\widetilde{X}))(m(\widetilde{X})-\widehat{m}(\widetilde{X}')))\middle|\mathcal{D}_\mathrm{train}}.
    \end{align*}
    We observe that the last crossed-term is null because using the conditional null hypothesis we have that $m(X)=m(\widetilde{X})$. 
    We note that the first term is exactly twice the Total Sobol Index, which under the conditional null hypothesis it is also null. The second term can be developed as
    \begin{align*}
        \mathbb{E}\left[\left(m(\widetilde{X})-\widehat{m}(\widetilde{X}')\right)^2\middle|\mathcal{D}_\mathrm{train}\right]&=\e{(m(\widetilde{X})-\widehat{m}(\widetilde{X}))^2\middle|\mathcal{D}_\mathrm{train}}+\e{(\widehat{m}(\widetilde{X})-\widehat{m}(\widetilde{X}'))^2\middle|\mathcal{D}_\mathrm{train}}\\&+2\e{(m(\widetilde{X})-\widehat{m}(\widetilde{X})))(\widehat{m}(\widetilde{X})-\widehat{m}(\widetilde{X}'))\middle|\mathcal{D}_\mathrm{train}}.
    \end{align*}

     Then, using that $\widetilde{X}\overset{\mathrm{i.i.d.}}{\sim}X$, we have that $\e{(m(\widetilde{X})-\widehat{m}(\widetilde{X}))^2\middle|\mathcal{D}_\mathrm{train}}=\e{(m(X)-\widehat{m}(X))^2\middle|\mathcal{D}_\mathrm{train}}$, so this first term will cancel with the second term in \eqref{eq:bias-cpi1}.
     
     Finally, combining all the previous we have that 
     \begin{align*}
         \mathbb{E}&\left[\frac{1}{2}\widehat{\psi}_\mathrm{CPI}(j, P_0)\middle|\mathcal{D}_\mathrm{train}\right]\\&=\frac{1}{2}\e{\left(y-\widehat{m}(\widetilde{X}')\right)^2\middle|\mathcal{D}_\mathrm{train}}-\e{\left(y-\widehat{m}(X)\right)^2\middle|\mathcal{D}_\mathrm{train}}\\&=\psi_\mathrm{TSI}(j, P_0)+\frac{1}{2}\e{\left(\widehat{m}(\widetilde{X})-\widehat{m}(\widetilde{X}')\right)^2\middle|\mathcal{D}_\mathrm{train}}\\&\qquad+ \e{\left(m(\widetilde{X})-\widehat{m}(\widetilde{X})))(\widehat{m}(\widetilde{X})-\widehat{m}(\widetilde{X}')\right)\middle|\mathcal{D}_\mathrm{train}}\\&=\frac{1}{2}\e{\left(\widehat{m}(\widetilde{X})-\widehat{m}(\widetilde{X}')\right)^2\middle|\mathcal{D}_\mathrm{train}}+ \e{\left(m(\widetilde{X})-\widehat{m}(\widetilde{X})))(\widehat{m}(\widetilde{X})-\widehat{m}(\widetilde{X}')\right)\middle|\mathcal{D}_\mathrm{train}}.
     \end{align*}
     
\end{proof}

\subsection{Proof of Proposition \ref{prop:loco-bias}}

 \begin{proof}
     \begin{align*}
         \e{\widehat{\psi}_\mathrm{LOCO}(j)\middle|\mathcal{D}_\mathrm{train}}&= 
         \e{\frac{1}{n_\mathrm{test}}\sum_{i=1}^{n_\mathrm{test}}\left(y_i-\widehat{m}_{-j}(X_i^{-j})\right)^2-\left(y_i-\widehat{m}(X_i)\right)^2\middle|\mathcal{D}_\mathrm{train}}\\
         &=\e{\left(y-\widehat{m}_{-j}(X^{-j})\right)^2-\left(y-\widehat{m}(X)\right)^2\middle|\mathcal{D}_\mathrm{train}}.
     \end{align*}
     On the one hand, we have that 
     \begin{align}
         \e{\left(y-\widehat{m}_{-j}(X^{-j})\right)^2\middle|\mathcal{D}_\mathrm{train}}&= \sigma^2+\e{\left(m(X)-\widehat{m}_{-j}(X^{-j})\right)^2\middle|\mathcal{D}_\mathrm{train}}\text{ using \cref{ass:regressModel}}\notag\\
         &= \sigma^2+\e{\left(m_{-j}(X^{-j})-\widehat{m}_{-j}(X^{-j})\right)^2\middle|\mathcal{D}_\mathrm{train}}\notag\\&\qquad+\e{\left(m(X)-m_{-j}(X^{-j})\right)^2}\label{eq:bias-loco1},
     \end{align}
     where we have used that 
     \begin{align*}
         \mathbb{E}&\left[(m_{-j}(X^{-j})-\widehat{m}_{-j}(X^{-j}))(m(X)-m_{-j}(X^{-j}))|\mathcal{D}_\mathrm{train}\right]\\&=\e{(m_{-j}(X^{-j})-\widehat{m}_{-j}(X^{-j}))\e{(m(X)-m_{-j}(X^{-j}))|X^{-j},\mathcal{D}_\mathrm{train}}|\mathcal{D}_\mathrm{train}}=0.
     \end{align*}
     We note that $\e{\left(m(X)-m_{-j}(X^{-j})\right)^2}$ is exactly $\psi_\mathrm{TSI}(j, P_0)$.
     On the other hand, we also have that
     \begin{align}
         \e{\left(y-\widehat{m}(X)\right)^2\middle|\mathcal{D}_\mathrm{train}}&=\sigma^2+\e{(m(X)-\widehat{m}(X))^2\middle|\mathcal{D}_\mathrm{train}}.\label{eq:bias-loco2}
     \end{align}
     Combining \eqref{eq:bias-loco1} and \eqref{eq:bias-loco2}, we have that 
     \begin{align*}
         \mathbb{E}&\left[\widehat{\psi}_\mathrm{LOCO}(j)\middle|\mathcal{D}_\mathrm{train}\right]\\&= \psi_\mathrm{TSI}(j, P_0) + \e{\left(m_{-j}(X^{-j})-\widehat{m}_{-j}(X^{-j})\right)^2\middle|\mathcal{D}_\mathrm{train}}-\e{(m(X)-\widehat{m}(X))^2\middle|\mathcal{D}_\mathrm{train}}. 
     \end{align*}
 \end{proof}

\subsection{Proofs of double robustness in Gaussian linear setting}
As $X$ is Gaussian, as recalled in the main text, there is an explicit form for the distribution of a coordinate given the others: $X_{j}|X_{-j}=x_{-j}\sim \mathcal{N}(\mu^j_\mathrm{cond}, \Sigma^j_\mathrm{cond})$ 
with $\mu^j_\mathrm{cond}:=\mu_{j}+\Sigma_{j, -j}\Sigma^{-1}_{-j, -j}(x_{-j}-\mu_{-j})$ and with $\Sigma^j_\mathrm{cond}:=\Sigma_{j, j}-\Sigma_{j,-j}\Sigma^{-1}_{-j, -j}\Sigma_{-j, j}$.
For simplicity, we assume that $\mu=0$. Then, $\e{X^j| X^{-j}} = \Sigma_{j, -j}\Sigma^{-1}_{-j, -j}X_{-j}$, so it is a linear model. We will then denote it as $\nu_{-j}(X^{-j})=X^{-j}\gamma_{-j}$ for $\gamma_{-j}\in \mathbb{R}^{p-1}$. 

We will denote by $\Delta\beta = \beta - \widehat{\beta}$ where $\widehat{\beta}$ was obtained by regression $y$ given $X$. Similarly, $\Delta \beta '= \beta' -\widehat{\beta}' $, where $\widehat{\beta}'$ was trained only over $X^{-j}$ and $\Delta\gamma_{-j}=\gamma_{-j}-\widehat{\gamma}_{-j}$.

We start by showing that the bias of LOCO decreases linearly with the training sample, and then the quadratic decay for CPI.

\subsubsection{Linear decay of LOCO}

We decompose this proof into two parts. First, we prove \cref{prop:LM-LOCO-bias}, which rigorously decomposes the bias into the bias of estimating each linear model $\widehat{m}$ and $\widehat{m}_{-j}$. Then, in \cref{lemma:bias-LOCO-linear}, we use the underlying distribution to compute the expectation and prove the linear decay.

\begin{proposition}[LM LOCO bias]
    Under \cref{ass:LM} with $X$ Gaussian, we have that \begin{align*}
        \e{\widehat{\psi}_\mathrm{LOCO}(j, P_0)\middle|\mathcal{D}_\mathrm{train}}&=\psi_{TSI}(j, P_0)+\Delta\beta'^\top\Sigma_{-j}\Delta\beta'-\Delta\beta^\top\Sigma\Delta\beta\\&=\psi_{TSI}(j, P_0)+\|\Delta\beta'\|^2_{\Sigma_{-j}}-\|\Delta\beta\|^2_\Sigma. 
    \end{align*}\label{prop:LM-LOCO-bias}
\end{proposition}

\begin{proof}
First we note that $$y-\widehat{m}(X)=X\beta+\epsilon-X\widehat{\beta}=X\Delta\beta+\epsilon\sim \mathcal{N}(0, \sigma^2+\Delta\beta^\top\Sigma\Delta\beta).$$

Similarly, using \cref{lemma:LMX_without_j}, we have that $$y-\widehat{m}_{-j}(X^{-j})=X^{-j}\Delta\beta'+\epsilon'\sim\mathcal{N}(0, \Delta\beta'^\top\Sigma_{-j}\Delta\beta'+\sigma^2 + \Sigma_\mathrm{cond}{\beta^j}^2).$$

Therefore, we have that 
\begin{align*}
    \e{\widehat{\psi}_\mathrm{LOCO}(j, P_0)\middle|\mathcal{D}_\mathrm{train}}&= \e{(y-\widehat{m}_{-j}(X^{-j}))^2\middle|\mathcal{D}_\mathrm{train}}-\e{(y-\widehat{m}(X))^2\middle|\mathcal{D}_\mathrm{train}}\\&=
    \mathbb{V}\left(y-\widehat{m}_{-j}(X^{-j})\middle|\mathcal{D}_\mathrm{train}\right) - \mathbb{V}\left(y-\widehat{m}(X)\middle|\mathcal{D}_\mathrm{train}\right)\\ 
    &= \sigma^2+{\beta^j}^2\Sigma_\mathrm{cond}+\Delta\beta'^\top\Sigma_{-j}\Delta\beta'-\sigma^2-\Delta\beta^\top\Sigma\Delta\beta\\
    &=\psi_{\mathrm{TSI}}(j, P_0)+\|\Delta\beta'\|^2_{\Sigma_{-j}}-\|\Delta\beta\|^2_\Sigma,
\end{align*}
where we have used that $\psi_{\mathrm{TSI}}(j, P_0)={\beta^j}^2\Sigma_\mathrm{cond}$.

\end{proof}

For an initial intuition, let's assume we have independent training samples for each regressor (which is even proposed in \citet{williamson2023general} to control the type-I error).

Therefore, we observe that the bias is distributed as a linear combination of independent chi-squared covariates. Indeed,  using that in a linear model $\widehat{\beta}\sim\mathcal{N}(0, \sigma^2(X_{\mathrm{tr}}^\top X_\mathrm{tr})^{-1})$, we have that
$\Delta\beta'\sim (\sigma^2+\beta^2_j\Sigma_\mathrm{cond})\mathcal{N}(0, (X_{\mathrm{tr}, -j}^\top X_{\mathrm{tr}, -j})^{-1})$ and $\Delta\beta\sim\mathcal{N}(0, \sigma^2(X^\top_\mathrm{tr}X_\mathrm{tr})^{-1})$. Therefore, for independent estimates, we would diagonalize and obtain a linear combination of independent chi-squared variables. This decreases linearly with the number of samples, as the error rate of the linear model. Intuitively, $\Sigma^{-1/2}X^\top_{\mathrm{tr}}X_\mathrm{tr}\Sigma^{-1/2}$ is \textit{almost} the identity matrix multiplied by a coefficient decreasing in probability linearly with $n_\mathrm{train}$. Similarly for $\Sigma_{-j}^{-1/2}X_{-j}^\top X_{-j}\Sigma_{-j}^{-1/2}$. Therefore, we have that 
\begin{align*}
    \|\Delta\beta'\|^2_{\Sigma_{-j}}-\|\Delta\beta\|^2_\Sigma \approx \left((\sigma^2+{\beta^j}^2\Sigma_\mathrm{cond})\chi^2(p-1)\right)O_P(1/n_\mathrm{train})-\left(\sigma^2\chi^2(p)\right)O_P(1/n_\mathrm{train}),
\end{align*}
with $O_P(1/n)$ meaning that it decreases in probability linearly with $n$. We observe that in this example both errors follow the same distribution. This is because, as shown in \cref{lemma:LMX_without_j}, the restricted model also satisfies the linear model. We observe that this is not generally the case: \citet{lobo2024primerlinearclassificationmissing} showed that even under the global logistic assumption, the restricted model is not logistic. 

Even if both errors are not independent, there is no reason that the error made in a given coordinate would be compensated with the error made in the same coordinate in the global model. Therefore, this is also one difference with the error made by CPI, as will be seen later: in CPI, we will only take into account the error made to estimate the $ \beta^j$ coordinate.

Under the null hypothesis, we have that $\beta^j = 0$. 

We show that under the null hypothesis the expectation of the error decreases linearly with the number of samples:

\begin{lemma}
Under the conditional null hypothesis, Assumption \ref{ass:LM} and Gaussian covariates, we have that 

\begin{align*}
    \e{\widehat{\psi}_\mathrm{LOCO}(j, P_0)}=\psi_{\mathrm{TSI}}(j, P_0)+O(1/n_\mathrm{train})=O(1/n_\mathrm{train}).
\end{align*}
\end{lemma}

Therefore, in any case, assuming two training sets or one training set, the bias is of the order of $n_\mathrm{train}$.

\begin{proof}
Using Proposition \ref{prop:LM-LOCO-bias}, we only need to compute $\e{\|\Delta\beta'\|^2_{\Sigma_{-j}}}$ and $\|\Delta\beta\|^2_\Sigma$. We compute the second one and the first one is done equivalently.

We first observe that $\Delta\beta| X_\mathrm{tr}\sim\mathcal{N}(0, \sigma^2(X_\mathrm{tr}^\top X_\mathrm{tr})^{-1})$. We denote by $$A:=\Sigma^{1/2}\Delta\beta/\sigma^2\sim \mathcal{N}(0, (\Sigma^{-1/2}X_\mathrm{tr}^\top X_\mathrm{tr}\Sigma^{-1/2})^{-1}).$$ We have that 

\begin{align*}
    \e{\|\Delta\beta\|^2_{\Sigma}}&= \sigma^4\e{A^\top A} = \sigma^4 \e{\mathrm{tr}(A^\top A)}= \sigma^4 \e{\mathrm{tr}(A A^\top)} = 
    \sigma^4 \mathrm{tr}(\e{A A^\top})\\&=  \sigma^4 \mathrm{tr}(\e{\e{A A^\top|X_\mathrm{tr}}})=  \sigma^4 \mathrm{tr}(\e{(\Sigma^{-1/2}X_\mathrm{tr}^\top X_\mathrm{tr}\Sigma^{-1/2})^{-1}}).
\end{align*}

We remark that $X_{tr}^i\Sigma^{-1/2}\sim \mathcal{N}(0, I_{p})$. 
Therefore, $(\Sigma^{-1/2}X_\mathrm{tr}^\top X_\mathrm{tr}\Sigma^{-1/2})^{-1}$ is distributed as a Inverse-Wishart with $I_d$ the scale matrix and $n_\mathrm{train}$ the degrees of freedom. Therefore, its expectation is $I_d/(n_\mathrm{train}-p-1)$. Therefore, we have that 
\begin{align*}
    \e{\|\Delta\beta\|^2_{\Sigma}} = \sigma^4\frac{1}{n_\mathrm{train}-p-1}\mathrm{tr}(I_p)=\sigma^4\frac{p}{n_\mathrm{train}-p-1}.
\end{align*}
Similarly for $\e{\|\Delta\beta'\|^2_{\Sigma_{-j}}}$, we have that 
\begin{align*}
    \e{\|\Delta\beta'\|^2_{\Sigma_{-j}}} = \sigma^4\frac{p-1}{n_\mathrm{train}-p-2}.
\end{align*}
Therefore, we conclude that 
\begin{align*}
    \e{\|\Delta\beta'\|^2_{\Sigma_{-j}}} - \e{\|\Delta\beta\|^2_{\Sigma}}&=\sigma^4\frac{p-1}{n_\mathrm{train} -p-2}- \sigma^4\frac{p}{n_\mathrm{train}-p-1}\\&=\sigma^4\frac{-n_\mathrm{train}+2p+1}{(n_\mathrm{train}-p-1)(n_\mathrm{train}-p-2)}=O(1/n_\mathrm{train}).
\end{align*}

\end{proof}

\subsubsection{Quadratic decay of CPI}

In this section, we prove that CPI benefits from double robustness, achieving a faster convergence rate for the conditional null hypothesis. 
For notational simplicity, we suppress the superscript in $\widetilde{X}^{(j)}$.

Indeed, we first note that in the linear Gaussian setting, the bias will consist of the product of errors.

\begin{proposition}[LM CPI bias] Under Assumption \ref{ass:LM} with X Gaussian and $j\in\mathcal{H}_0$, we have that 
\begin{align}
    \e{\frac{1}{2}\widehat{\psi}_\mathrm{CPI}(j)\middle| \mathcal{D}_\mathrm{train}}&\approx \psi_\mathrm{TSI}(j, P_0)+\widehat{\beta}_j^2 \Delta \gamma_{-j}^\top\Sigma_{-j}\Delta\gamma_{-j}\\
    &= \psi_\mathrm{TSI}(j, P_0)+\widehat{\beta}_j^2 \|\Delta \gamma_{-j}\|^2_{\Sigma_{-j}}.
\end{align}
    \label{prop:LM-CPI-bias}
\end{proposition}
 
We first note that $\|\Delta \gamma_{-j}\|^2_{\Sigma_{-j}}$ represents the error resulting from using a finite sample to estimate the function $\nu_{-j}$. Since this function is linear, the error will decrease linearly with the size of the training sample. We also observe that, in this case, contrary to what occurs in the LOCO approach, the only estimation error accounted for in the global model is in the estimation of $\beta_j$. In general, this error will converge to the coefficient, leading to the usual linear convergence rate in the product.

Nevertheless, under the conditional null hypothesis, this coefficient is zero. Consequently, it will also decrease linearly with the training sample size. Since both errors are multiplied, we obtain a quadratic convergence rate, meaning that CPI will detect more quickly when a covariate is null, resulting in fewer false positives than when using LOCO for variable selection with statistical guarantees.

\begin{proof}
We first note that since $X$ is Gaussian, Assumption \ref{ass:generalAssumpCondSampl} is satisfied. Therefore, from Proposition \ref{prop:cpi-bias} we have that
\begin{align*}
     \mathbb{E}&\left[\frac{1}{2}\widehat{\psi}_\mathrm{CPI}(j)\middle|\mathcal{D}_\mathrm{train}\right]\\&=\frac{1}{2}\e{\left(\widehat{m}(\widetilde{X})-\widehat{m}(\widetilde{X}')\right)^2\middle|\mathcal{D}_\mathrm{train}}+ \e{\left(m(\widetilde{X})-\widehat{m}(\widetilde{X})))(\widehat{m}(\widetilde{X})-\widehat{m}(\widetilde{X}')\right)\middle|\mathcal{D}_\mathrm{train}}.
 \end{align*}
 
We start by noting that the second bias term is negligible:

\begin{align*}
    \mathbb{E}&\left[\left(m(\widetilde{X})-\widehat{m}(\widetilde{X})))(\widehat{m}(\widetilde{X})-\widehat{m}(\widetilde{X}')\right)\middle|\mathcal{D}_\mathrm{train}\right] = \e{(\Delta\beta^\top\widetilde{X})(\widehat{\beta}^\top(\widetilde{X}-\widetilde{X}'))\middle|\mathcal{D}_\mathrm{train}}\\
    &= \Delta\beta^\top\e{\widetilde{X}(\widehat{\beta}^j(\nu_{-j}(X^{-j})+(X'^j-\nu_{-j}(X'^{-j}))-\widehat{\nu}_{-j}(X^{-j})-(X'^j-\widehat{\nu}_{-j}(X'^{-j}))))\middle|\mathcal{D}_\mathrm{train}}
    \\
    &= \widehat{\beta}^j\Delta\beta^\top\e{\widetilde{X}(\nu_{-j}(X^{-j})-\widehat{\nu}_{-j}(X^{-j}))-(\nu_{-j}(X'^{-j})-\widehat{\nu}_{-j}(X'^{-j})))\middle|\mathcal{D}_\mathrm{train}}
    \\
    &= \widehat{\beta}^j\Delta\beta^\top\e{\widetilde{X}(\Delta\gamma_{-j}^\top X^{-j}-\Delta\gamma_{-j}^\top X'^{-j})\middle|\mathcal{D}_\mathrm{train}}
    \\
    &= \widehat{\beta}^j\Delta\beta^\top\e{\widetilde{X}(X^{-j}- X'^{-j})^\top\middle|\mathcal{D}_\mathrm{train}}\Delta\gamma_{-j}.
\end{align*}
Then, we observe that it depends both in the estimation error of $\beta$ and of $\gamma_{-j}$, which are both consistent. Moreover, as $j\in\mathcal{H}_0$, then $\widehat{\beta}_j\to0$. This third-order term is negligible.

For the first bias term we have that

\begin{align*}
    \mathbb{E}&\left[\left(\widehat{m}(\widetilde{X})-\widehat{m}(\widetilde{X}')\right)^2\middle|\mathcal{D}_\mathrm{train}\right]\\&= \widehat{\beta}^2 \e{\left(\nu_{-j}(X^{-j})+(X'^j-\nu_{-j}(X'^{-j})))-\widehat{\nu}_{-j}(X^{-j})-(X'^j-\widehat{\nu}_{-j}(X'^{-j})))\right)^2\middle|\mathcal{D}_\mathrm{train}}
    \\&=\widehat{\beta}^2 \e{\left(\nu_{-j}(X^{-j})-\widehat{\nu}_{-j}(X^{-j})-(\nu_{-j}(X'^{-j})-\widehat{\nu}_{-j}(X'^{-j}))\right)^2\middle|\mathcal{D}_\mathrm{train}}
     \\&=\widehat{\beta}^2 \e{\left(\Delta \gamma_{-j}^\top X^{-j}-\Delta\gamma_{-j}^\top X'^{-j}\right)^2\middle|\mathcal{D}_\mathrm{train}}
     \\&=\widehat{\beta}^2 \Delta\gamma_{-j}^\top\e{\left( X^{-j}- X'^{-j}\right)\left( X^{-j}- X'^{-j}\right)^\top\middle|\mathcal{D}_\mathrm{train}}\Delta \gamma_{-j}\\&=2\widehat{\beta}_j^2 \|\Delta \gamma_{-j}\|^2_{\Sigma_{-j}}.
\end{align*}

\end{proof}

 \begin{lemma}
Under the conditional null hypothesis, Assumption \ref{ass:LM} and Gaussian covariates, assume $\widehat{\nu}_{-j}$ and $\widehat{\beta}$ trained in different samples, we have that 

\begin{align*}
    \e{\widehat{\psi}_\mathrm{CPI}(j)}=O(1/n_\mathrm{train}^2).
\end{align*}
\end{lemma}

\begin{proof}
     We observe that $\widehat{\beta}_j\sim\mathcal{N}(\beta_j,\sigma^2({X_\mathrm{tr}}^\top X_\mathrm{tr})^{-1}_j)$ and that  $\Delta\gamma_{-j}\sim\mathcal{N}(0,(\sigma^2+\Sigma_\mathrm{cond}) ({X^{-j}_\mathrm{tr}}^\top X^{-j}_\mathrm{tr})^{-1})$. Therefore, assuming the independence between both of them because they have been trained in independent training sets, we have that under the conditional null hypothesis the variance of the distribution of the bias will decrease quadratically. To remark so, we can use similar arguments as the ones used in the proof of Lemma \ref{lemma:bias-LOCO-linear} because $\mathrm{tr}(\Sigma^{1/2}(X_\mathrm{tr}^\top X_\mathrm{tr})^{-1}\Sigma^{1/2})\approx p/n_\mathrm{train}$.
 \end{proof}

\section{Assumptions and proof of Theorem \ref{th:asymp_eff} (Nonparametric Efficiency)}\label{sect:proofTheo}

In this section we present the assumptions taken from \citet{williamson2023general} to proof the asymptotic efficiency. To do so, we first introduce the same notations as in this paper. 

Denote $\mathcal{R}:=\{c(P_1-P_2):c\in [0, \infty), P_1, P_2\}$ the linear space of finite signed measures generated by the class of distributions $\mathcal{M}$. In $\mathcal{R}$, the supremum norm $\|\cdot\|_\infty$ is the supremum difference of their distribution functions. The Gâteaux derivative of $P\mapsto\mathbb{E}_P\left[\ell(f(X), y)\right]$ at $P_0$ in the direction $h\in \mathcal{R}$ is denoted by $\dot{V}(f, P_0, h)$. Let's also define the random function $g_n:(X,y)\mapsto \dot{V}(\widehat{m}, P_0; \delta_{(X,y)}-P_0)-\dot{V}(m, P_0;\delta_{(X,y)}-P_0)$, where $\delta_{(X,y)}$ is the degenerate distribution on $\{(X,y)\}$. The assumptions are either deterministic (A) or stochastic (B):

\begin{enumerate}[label=A\arabic*]
    \item (\textit{optimality}) there exists some constant $C>0$ such that, for each sequence $\widehat{m}_1, \widehat{m}_2, \ldots\in \mathcal{F}$ such that $\|\widehat{m}_n-m\|_\mathcal{F}\to 0$, $|\e{\ell(\widehat{m}_n(X), y)}-\e{\ell(m(X), y)}|\leq C\|\widehat{m}_n-m\|^2_\mathcal{F}$ for each $n$ large enough.
    \item (\textit{differentiability}) there exists some constant $\delta>0$ such that for each sequence $\epsilon_1, \epsilon_2, \ldots\in \mathbb{R}$ and $h, h_1, h_2, \ldots\in \mathcal{R}$ satisfying that $\epsilon_j\to 0$, and $\|h_j-h\|-\infty\to 0$, it holds that 
    \begin{align*}
        \underset{f\in\mathcal{F}:\|f-m\|_\mathcal{F}<\delta}{\mathrm{sup}}\left|\frac{\mathbb{E}_{P_0+\epsilon_jh_j}\left[\ell(f(X), y)\right]-\mathbb{E}_{P_0}\left[\ell(f(X), y)\right]}{\epsilon_j}-\dot{V}(f, P_0; h_j)\right|\to 0.
    \end{align*}
    \item (\textit{continuity of optimization}) $\|m_{P_0+\epsilon h}-m\|_\mathcal{F}=O(\epsilon)$ for each $h\in \mathcal{R}$ where $m_{P_0+\epsilon h}$ is the minimizer in $\mathcal{F}$ of $f\mapsto \mathbb{E}_{P_0+\epsilon h}[\ell(f(X), y)]$.
    \item (\textit{continuity of derivative}) $f\mapsto \dot{V}(f, P_0; h)$ is continuous at $m$ relative to $\|\cdot\|_\mathcal{F}$ for each $h\in \mathcal{R}$.
\end{enumerate}

\begin{enumerate}[label=B\arabic*]
    \item (\textit{minimum rate of convergence of $m$}) $\|\widehat{m}_n-m\|_\mathcal{F}=o_P(n^{-1/4})$.
    \item (\textit{weak consistency}) $\int (g_n(X,y))^2dP_0(X,y)=o_P(1)$.
\end{enumerate}

First, we note that there is no need to include the limited complexity assumption, as we are already using a cross-fitted version. This is because the importance is estimated on a separate test dataset from the training set.  

We also recall from \citet{williamson2023general} that Assumption~A1 is referred to as \emph{optimality} because we require a development that depends only on the second order, as the derivative tends to zero given that the model under consideration is a minimizer of the loss.

Similarly, as done in \citet{williamson2023general}, we will prove the asymptotic efficiency of each empirical predictiveness measure with respect to its theoretical counterpart separately. This means that both the predictiveness measure using the information of the $j$-th covariate and the one without it will converge efficiently to their respective theoretical predictiveness measures. Therefore, all the previous assumptions need to be fulfilled for the restricted model. Additionally, we will require the minimum convergence rate not only of $\widehat{m}$, which was already required for the complete model convergence, but also for the model $\widehat{\nu}_{-j}$:

\begin{enumerate}[label=B3]
    \item (\textit{minimum rate of convergence of $\nu_{-j}$}) $\|\widehat{\nu}_{n,-j}-\nu_{-j}\|_\mathcal{F}=o_P(n^{-1/4})$.
\end{enumerate}

We observe that all the above assumptions are exactly the ones required for the asymptotic efficiency of LOCO in \citet{williamson2023general}, except for replacing the required minimum rate of convergence of $\widehat{m}_{-j}$ with that of $\widehat{\nu}_{-j}$. 

We observe that these convergence rates are achieved by parametric models. They are standard in semiparametric inference. They are also achievable for nonparametric models under additional assumptions on dimensionality and regularity, and for standard machine learning models under mild assumptions.

We recall that, in our setting—the knockoffs framework (\citet{candes2017panninggoldmodelxknockoffs})—it is easier to fulfill the latter condition than the former. This is because the relationship between the input and the output is expected to be more complex than the relationship among the inputs.

For a detailed discussion of these assumptions, see \citet{williamson2023general}.

For this result, we also need local robustness to small changes in the argument of the function $m$. This condition is expressed by requiring $m$ to be Lipschitz continuous. This ensures that the error made by estimating the argument does not explode.

Finally, we require Assumption \ref{ass:generalAssumpCondSampl} to hold for the validity of the conditional sampling step.

\begin{theorem}
Under Assumption \ref{ass:generalAssumpCondSampl}, assuming that $m$ is Lipschitz and assumptions A1-A4 and B1-B3,  $\widehat{\psi}_\mathrm{SCPI}(j)$ is nonparametric efficient.
\end{theorem}

\begin{proof}
Similarly as done in \citet{williamson2023general}, the asymptotic efficiency will be established by decomposing each predictiveness measure. Indeed, we are going to first prove that $$\frac{1}{n_\mathrm{test}}\sum_{i=1}^{n_\mathrm{test}}\ell(\widehat{m}(x_i), y_i)\to\e{\ell(m(X), y)}.$$
This comes directly from Theorem 2 of \citet{williamson2023general}. Then, we need to proof this efficient asymptotic convergence $$\frac{1}{n_\mathrm{test}}\sum_{i=1}^{n_\mathrm{test}}\ell\left(\frac{1}{n_\mathrm{cal}}\sum_{l=1}^{n_\mathrm{cal}}\widehat{m}(\widetilde{x}_{i,l}^{(j)'}), y_i\right)\to\e{\ell(m_{-j}(X^{-j}), y)}.$$

To achieve this, we will apply the same theorem. For this, we need to prove that the regressor converges sufficiently fast to $m_{-j}$, specifically at the rate $o_P(n^{-1/4})$. Observe that
\begin{align*}
    \left\|\frac{1}{n_\mathrm{cal}}\sum_{l=1}^{n_\mathrm{cal}}\widehat{m}(\widetilde{X}_{l}^{(j)'})-m_{-j}(X^{-j})\right\|&\leq \left\|\frac{1}{n_\mathrm{cal}}\sum_{l=1}^{n_\mathrm{cal}}\widehat{m}(\widetilde{X}_{l}^{(j)'})-\frac{1}{n_\mathrm{cal}}\sum_{l=1}^{n_\mathrm{cal}}m(\widetilde{X}_{l}^{(j)'})\right\| \\&\qquad+ \left\|\frac{1}{n_\mathrm{cal}}\sum_{l=1}^{n_\mathrm{cal}}m(\widetilde{X}_{l}^{(j)'})-\frac{1}{n_\mathrm{cal}}\sum_{l=1}^{n_\mathrm{cal}}m(\widetilde{X}_{l}^{(j)})\right\|\\&\qquad+\left\|\frac{1}{n_\mathrm{cal}}\sum_{l=1}^{n_\mathrm{cal}}m(\widetilde{X}_{l}^{(j)})-m_{-j}(X^{-j})\right\|.
\end{align*}

We first note that the last quantity exhibits the desired convergence rate by applying the CLT, and because, under Assumption \ref{ass:generalAssumpCondSampl}, $\widetilde{X}^{(j)}$ follows the desired distribution. For the first term, we have that
$$ \left\|\frac{1}{n_\mathrm{cal}}\sum_{l=1}^{n_\mathrm{cal}}m\left(\widetilde{X}_{l}^{(j)'}\right)-\frac{1}{n_\mathrm{cal}}\sum_{l=1}^{n_\mathrm{cal}}m(\widetilde{X}_{l}^{(j)})\right\|\leq \mathrm{max}_{l}\|\widehat{m}(\widetilde{X}_l^{(j)'})-m(\widetilde{X}_l^{(j)'})\|.$$ Then, it has the correct convergence rate by using the convergence rate of $\widehat{m}$. Finally, for the second term, using that $m$ is assumed $L$-Lipschitz, we have that 
\begin{align*}
    \left\|\frac{1}{n_\mathrm{cal}}\right.&\left.\sum_{l=1}^{n_\mathrm{cal}}m(\widetilde{X}_{l}^{(j)'})-\frac{1}{n_\mathrm{cal}}\sum_{l=1}^{n_\mathrm{cal}}m(\widetilde{X}_{l}^{(j)})\right\|\\&\leq \mathrm{max}_{i}\|m(\widetilde{X}_{l}^{(j)'})-m(\widetilde{X}_{l}^{(j)'})\|\\&\leq L\mathrm{max}_{i}\|\widetilde{X}_{l}^{(j)'}-\widetilde{X}_{l}^{(j)'}\|\\
    &=L\mathrm{max}_{i}\| \widehat{\nu}_{-j}(X^{-j})+(X_i-\widehat{\nu}_{-j}(X_i^{-j}))-\nu_{-j}(X^{-j})-(X_i-\nu_{-j}(X_i^{-j}))\|\\
    &\leq L\left(\|\widehat{\nu}_{-j}(X^{-j})-\nu_{-j}(X^{-j})\|-\mathrm{max}_i\|\widehat{\nu}_{-j}(X_i^{-j})-\nu_{-j}(X_i^{-j})\|\right). 
\end{align*}
We conclude by using Assumption B3. 
\end{proof}

\section{Fixing $n_\mathrm{cal}$: Proof of Proposition \ref{prop:fix_n_cal}}

\begin{proof}
The first part of the proof consists on applying the consistency of the estimates, continuous mapping theorem, and finally the Law of Large Numbers. Then, we develop the given expectation as
\begin{align*}
    \mathbb{E}&\left[\left(y-\frac{1}{n_{\mathrm{cal}}}\sum_{i=1}^{n_{\mathrm{cal}}}m(\widetilde{X}_i^{(j)})\right)^2\right]-\e{(y-m(X))^2}\\&=\e{\left(m(X)-\frac{1}{n_{\mathrm{cal}}}\sum_{i=1}^{n_{\mathrm{cal}}}m(\widetilde{X}^{(j)}_i)\right)^2}\\
    &=\frac{1}{n_{\mathrm{cal}}^2}\e{\left(\sum_{i=1}^{n_{\mathrm{cal}}}(m(X)-m(\widetilde{X}^{(j)}_i))\right)^2}\\
    &=\frac{1}{n_{\mathrm{cal}}^2}\sum_{i=1}^{n_\mathrm{cal}}\e{(m(X)-m(\widetilde{X}^{(j)}_i))^2}+\frac{2}{n_{\mathrm{cal}}^2}\sum_{i<k}\e{(m(X)-m(\widetilde{X}^{(j)}_i))(m(X)-m(\widetilde{X}^{(j)}_k))}\\
    &=\frac{1}{n_{\mathrm{cal}}}\e{(m(X)-m(\widetilde{X}^{(j)}_1))^2}+\frac{2}{n_{\mathrm{cal}}^2}\sum_{i<k}\e{(m(X)-m(\widetilde{X}^{(j)}_i))(m(X)-m(\widetilde{X}^{(j)}_k))}.
\end{align*}
For the second part, we observe that 
\begin{align*}
    \e{\right.&\left.(m(X)-m(\widetilde{X}^{(j)}_i))(m(X)-m(\widetilde{X}^{(j)}_k))}\\&=\e{m(X)(m(X)-m(\widetilde{X}^{(j)}_k))}-\e{m(\widetilde{X}^{(j)}_i)(m(X)-m(\widetilde{X}^{(j)}_k))}.
\end{align*}
The second term vanishes:
\begin{align*}
    \e{m(\widetilde{X}^{(j)}_i)(m(X)-m(\widetilde{X}^{(j)}_k))}&=\e{\e{m(\widetilde{X}^{(j)}_i)(m(X)-m(\widetilde{X}^{(j)}_k))|X^{-j}}}\\
    &=\e{\e{m(\widetilde{X}^{(j)}_i)|X^{-j}}\e{(m(X)-m(\widetilde{X}^{(j)}_k))|X^{-j}}}\\
    &=0.
\end{align*}
Now we observe that the first term is exactly the Total Sobol Index:

\begin{align*}
    \e{m(X)(m(X)-m(\widetilde{X}^{(j)}_k))}&=\e{m(X)^2-m(X)m(\widetilde{X}^{(j)}_k)}\\
    &=\e{m(X)^2}-\e{\e{m(X)m(\widetilde{X}^{(j)}_k)|X^{-j}}}\\
    &=\e{m(X)^2}-\e{\e{m(X)|X^{-j}}\e{m(\widetilde{X}^{(j)}_k)|X^{-j}}}\\&=\e{m(X)^2}-\e{m_{-j}(X^{-j})^2},
\end{align*}
and 
\begin{align*}
    \psi_\mathrm{TSI}(j, P_0)&=\e{(m(X)-m_{-j}(X^{-j}))^2}\\&=\e{m(X)^2}-2\e{m(X)m_{-j}(X^{-j})}+\e{m_{-j}(X^{-j})^2}\\
    &=\e{m(X)^2}-2\e{\e{m(X)|X^{-j}}m_{-j}(X^{-j})}+\e{m_{-j}(X^{-j})^2}\\
    &=\e{m(X)^2}-\e{m_{-j}(X^{-j})^2}.
\end{align*}

Therefore we have 
\begin{align*}
    \widehat{\psi}_\mathrm{SCPI}(j)&\xrightarrow{n_{\mathrm{train}, n_{\mathrm{test}}} \to \infty}\frac{1}{n_{\mathrm{cal}}}\e{(m(X)-m(\widetilde{X}^{(j)}_1))^2}\\&\qquad+\frac{2}{n_{\mathrm{cal}}^2}\sum_{i<k}\e{(m(X)-m(\widetilde{X}^{(j)}_i))(m(X)-m(\widetilde{X}^{(j)}_k))}\\
    &=\frac{1}{n_{\mathrm{cal}}}2\psi_\mathrm{TSI}(j, P_0)+\frac{2}{n_{\mathrm{cal}}^2}\sum_{i<k}\psi_\mathrm{TSI}(j, P_0)\\
    &=\frac{1}{n_{\mathrm{cal}}}2\psi_\mathrm{TSI}(j, P_0)+\frac{2}{n_{\mathrm{cal}}^2}\psi_\mathrm{TSI}(j, P_0)\frac{n_{\mathrm{cal}}(n_{\mathrm{cal}}-1)}{2}\\
    &=\left(1+\frac{1}{n_{\mathrm{cal}}}\right)\psi_\mathrm{TSI}(j, P_0).
\end{align*}

\end{proof}

\section{Inference}

\subsection{Variance correction: type-I error and power}\label{app:inf_lin}
\paragraph{Type-I error:}In general, \citet{locoVSShapley} introduced an additional square root term to ensure control of the type-I error. Indeed, due to the quadratic nature of the statistic, under the null hypothesis, the variance vanishes as $\mathbb{V}(\widehat{\psi}) = O(n^{-\gamma})$ with $\gamma > 1$. Therefore, to maintain a valid type-I error rate, we observe that under the null hypothesis, using Chebyshev's inequality, we have that
\begin{align*}
    \mathbb{P}_{\mathcal{H}_0}\left(\widehat{\psi}\geq z_\alpha \mathrm{se}_n+c/\sqrt{n}\right)&\leq \frac{\mathbb{V}(\widehat{\psi})}{\left(z_\alpha \mathrm{se}_n+c/\sqrt{n}\right)^2}\to 0.
\end{align*}

\paragraph{Improving the corrective term in linear settings:}We observe that using Markov's inequality we have that 

\begin{align*}
    \mathbb{P}_{\mathcal{H}_0}\left(\widehat{\psi}\geq z_\alpha \mathrm{se}_n+c/\sqrt{n}\right)&\leq \frac{\mathbb{E}(\widehat{\psi})}{\left(z_\alpha \mathrm{se}_n+c/\sqrt{n}\right)}.
\end{align*}

From Lemmas~\ref{lemm:doubl-rob-LM} and \ref{lemma:bias-LOCO-linear}, we observe that in the linear setting, it is possible to use this last inequality to derive a more powerful valid test by changing the additive term $c/\sqrt{n}$ by  $c/n^{-\gamma}$, with $\gamma < 1$ for LOCO and $\gamma < 2$ for CPI.

\paragraph{Power:}We observe that under the alternative hypothesis, using Theorem~\ref{th:asymp_eff}, we have $\widehat{\psi} \xrightarrow{\mathrm{a.s.}} \psi > 0$. In particular, this implies that

\begin{align*}
    \mathbb{P}_{\mathcal{H}_1}\left(\widehat{\psi}\geq z_\alpha \mathrm{se}_n+c/\sqrt{n}\right)\underset{n\to \infty}{\to} 1.
\end{align*}
Thus, the test is \textit{consistent}, i.e. it has power approaching to one. 
\paragraph{Standard deviation estimation:}For $\mathrm{se}_n$ in the numerical experiments, we also explored using empirical variances derived from multiple estimations via bootstrapping on the test set, without refitting the model. Specifically, we computed several means over bootstrap samples on the test set and then estimate the variance among them. Also, the sample variance (which in this case coincides with the variance estimation with influence function as seen in Section \ref{sect:MSE_if}) divided by the sample size for the others, using the relation
\begin{align*}
\mathrm{var}(\widehat{\psi}_\mathrm{SCPI}) &= \mathrm{var}\left(\frac{n_\mathrm{cal}}{n_\mathrm{cal}+1} \cdot \frac{1}{n_{\mathrm{test}}} \sum_{i=1}^{n_{\mathrm{test}}} \left[ \left( \frac{1}{n_\mathrm{cal}} \sum_{k=1}^{n_\mathrm{cal}} \widehat{m}(\widetilde{x}^{\prime(j)}_{i,k}) - y_i \right)^2 - \left( \widehat{m}(x_i) - y_i \right)^2 \right] \right) \\
&= \frac{1}{n_{\mathrm{test}}} \cdot \mathrm{var} \left( \frac{n_\mathrm{cal}}{n_\mathrm{cal}+1} \left[ \left( \frac{1}{n_\mathrm{cal}} \sum_{k=1}^{n_\mathrm{cal}} \widehat{m}(\widetilde{x}^{\prime(j)}_{i,k}) - y_i \right)^2 - \left( \widehat{m}(x_i) - y_i \right)^2 \right] \right).
\end{align*}

\subsection{Sobol-CPI-Wilcoxon algorithm}

In Algorithm~\ref{alg:scpi_wcx}, we provide an explicit regression-based conditional sampler to give a concrete, implementable procedure; however, any sampler could be used, and the type-I error control guarantees from Theorem~\ref{th:wcx_TV} (based on the Total Variation distance) still hold.

\begin{algorithm}[H]
\caption{Nonparametric Sobol-CPI (\texttt{Sobol-CPI-ST/Wilcox})}
\label{alg:scpi_wcx}
\begin{algorithmic}[1]
\item[] \textbf{Input:} Model $\widehat{m}$, test data $\{(X_i,y_i)\}_i^{n_{\mathrm{test}}}$, control level $\alpha$ and a feature $j$
    \item[] \quad \textbf{(Regression)} Fit $\widehat{\nu}_j \approx \mathbb{E}[X^j \mid X^{-j}]$
    \item[] Compute the residuals $\widehat{\epsilon}_{j} = X^j -\widehat{\nu}_j(X^{-j})$
    \item[] Construct 
    \[
        \widetilde{X}'^{(j)}
        = \widehat{\nu}_j(X^{-j})
        + \widehat{\epsilon}_{j}^\mathrm{permuted}
    \]
    \item[] Compute nonparametric paired test (Sign-test/Wilcoxon) between 
    \[
        \left\{l\bigl(\widehat{m}(X_{i}), y_i\bigr)\right\}_i^{n_\mathrm{test}}
        \text{ and }
        \left\{l\bigl(\widehat{m}(\widetilde{X}'^{(j)}_{i}), y_i\bigr)\right\}_i^{n_\mathrm{test}}
    \]
\item[] \textbf{Return} the computed p-value
\end{algorithmic}
\label{alg:SCPI_wcx}
\end{algorithm}

\subsection{Proof of Theorem \ref{th:wcx_TV} (Wilcoxon and Sign-test)}

We start by recalling that the \emph{total variation} for any probability measures $P_1$ and $P_2$ defined on the same probability space is given by the supremum of the difference of the measures over all measurable sets:
\begin{align*}
    d_\mathrm{TV}(P_1, P_2):=\mathrm{sup}_A |P_1(A)-P_2(A)|.
\end{align*}

We present the proof explicitly for the Sign test for readability, but the argument for the Wilcoxon test is entirely analogous. Following \citet{Berrett}, we observe that the total variation, conditioned on all the random variables used for the test, is exactly the distance between the estimated and original conditional distributions.
\begin{align*}
    d_\mathrm{TV}&\left\{\left(\widetilde{X}'^{(j)j}, X_1^j, \ldots, X_{n_\mathrm{test}}^j)\mid \{X^{-j}_i, y_i\}_i^{n_\mathrm{test}}, \mathcal{D}_\mathrm{train}\right), \left(\widetilde{X}^{(j)j}, X_1^j, \ldots, X_{n_\mathrm{test}}^j)\mid \{X^{-j}_i, y_i\}_i^{n_\mathrm{test}}, \mathcal{D}_\mathrm{train}\right)\right\}\\
    &=d_\mathrm{TV}\left\{\left(\widetilde{X}'^{(j)j})\mid \{X^{-j}_i\}_i^{n_\mathrm{test}}, \mathcal{D}_\mathrm{train}\right), \left(\widetilde{X}^{(j)j})\mid \{X^{-j}_i\}_i^{n_\mathrm{test}}, \mathcal{D}_\mathrm{train}\right)\right\}\\
    &=d_\mathrm{TV}\left\{P'^{n_\mathrm{test}}_j, P^{\star n_\mathrm{test}}_j\right\},
\end{align*}
where we have used the assumption that the conditional sampler was estimated from a sample independent of the training sample used to fit $\widehat{m}$.

Let us define the set 
\begin{align*}
A:= \left\{\{\widetilde{X}^{(j)j}_i\}_i\in \mathcal{X}^{\times n_\mathrm{test}}\mid \sum l\left(\widehat{m}(\widetilde{X}^{(j)}_i, y_i)-l(\widehat{m}(X_i), y_i)\right)\geq q^{1-\alpha}_{\mathrm{Bin}(n_\mathrm{test}, 0,5)}\right\}    
\end{align*}
where $q^{1-\alpha}_{\mathrm{Bin}(n_\mathrm{test}, 0.5)}$ denotes the $(1-\alpha)$-quantile of a binomial distribution with parameters $n_\mathrm{test}$ and success probability $0.5$. Note that this corresponds to the statistic of the Sign test and can be generalized to the Wilcoxon test.

Finally, we conclude by noting that
\begin{align*}
    \mathbb{P}_{\mathcal{H}_0}&\left(\widehat{\phi}_\mathrm{ST}=1\mid \{X^{-j}_i, y_i\}_i^{n_\mathrm{test}}, \mathcal{D}_\mathrm{train}\right)
    \\&\leq P^\star_j(A)+d_\mathrm{TV}\left\{\left(\widetilde{X}'^{(j)j}, X_1^j, \ldots, X_{n_\mathrm{test}}^j)\mid \{X^{-j}_i, y_i\}_i^{n_\mathrm{test}}, \mathcal{D}_\mathrm{train}\right), \left(\widetilde{X}^{(j)j}, X_1^j, \ldots, X_{n_\mathrm{test}}^j)\mid \{X^{-j}_i, y_i\}_i^{n_\mathrm{test}}, \mathcal{D}_\mathrm{train}\right)\right\}
    \\&\leq  \alpha + d_\mathrm{TV}\left\{P'^{n_\mathrm{test}}_j, P^{\star n_\mathrm{test}}_j\right\}.
\end{align*}

\section{MSE variance estimation with influence functions}\label{sect:MSE_if}

From Theorem \ref{th:asymp_eff}, it is possible to estimate the variance using the influence function. Indeed, the variance is given by 
\[
\tau_j^2 := \mathbb{E}\left[\varphi_j^2(X, y)\right],
\]
where $\varphi_j$ is the influence function of $\psi_j$, which is given by 
\[
\varphi_j: (X,y) \mapsto \dot{V}(m_{-j}, P_0, \delta_{(X,y)} - P_0) - \dot{V}(m, P_0, \delta_{(X,y)} - P_0).
\]
Here, $\dot{V}(f, P_0; h)$ stands for the Gâteaux derivative of $P \mapsto \mathbb{E}_P\left[\ell(f(X), y)\right]$ in the direction $h \in \mathcal{R}$, with 
\[
\mathcal{R} := \{c(P_1 - P_2) : c \in [0, \infty), P_1, P_2 \in \mathcal{M}\}
\]
being the finite signed measures generated by $\mathcal{M}$. Therefore, it is possible to estimate the variance as a simple plug-in:
\[
\widehat{\tau}^2_j := \frac{1}{n}\sum_{i=1}^n \left[\dot{V}(\widehat{m}_{-j}, P_n; \delta_{(X^{-j}_i, y_i)} - P_n) - \dot{V}(\widehat{m}, P_n; \delta_{(X_i,y_i)} - P_n)\right]^2.
\]

In Appendix A of \citet{williamson2023general}, they propose a simple method to compute the influence function when the predictiveness measure comes from standardized V-measures. We observe that, in general, computing the empirical variance and the variance using this plug-in version do not coincide. Nevertheless, in this section, we easily prove that they do coincide when using the MSE. Indeed, we observe that
 
\begin{align*}
    \widehat{\tau}^2_j&=\frac{1}{n}\sum_{i=1}^n\left[\dot{V}(\widehat{m}_{-j}, P_n; \delta_{(x_i, y_i)}-P_n)-\dot{V}(\widehat{m}, P_n; \delta_{(x_i, y_i)}-P_n)\right]^2\\
    &=\frac{1}{n}\sum_{i=1}^n\left[\left(y_i-\widehat{m}_{-j}(x^{-j}_i)\right)^2-\frac{1}{n}\sum_{i=1}^n\left(y_i-\widehat{m}_{-j}(x^{-j}_i)\right)^2-\left(y_i-\widehat{m}(x_i)\right)^2\right.\\ &\left.\qquad+\frac{1}{n}\sum_{i=1}^n\left(y_i-\widehat{m}(x_i)\right)^2\right]^2
    \\
    &=\frac{1}{n}\sum_{i=1}^n\left[\left(y_i-\widehat{m}_{-j}(x_i^{-j})\right)^2-\left(y_i-\widehat{m}(x_i)\right)^2\right.\\ &\left.\qquad-\frac{1}{n}\sum_{i=1}^n\left[\left(y_i-\widehat{m}_{-j}(x_i^{-j})\right)^2-\left(y_i-\widehat{m}(x_i)\right)^2\right]\right]^2\\
    &= \widehat{\mathrm{var}}\left(\ell(y, \widehat{m}_{-j}(x_i^{-j}))-\ell(y, \widehat{m}(x_i)\right).
\end{align*}

\section{Some lemmas used in the proofs}

\begin{lemma}[Conditional null hypothesis]\label{lemma:nullHyp} Under Assumption \ref{ass:regressModel}, the j-th covariate is independent of the output $y$ conditionally on the rest of covariates if and only if there exists a measurable function $m_{-j}\in\mathcal{F}_{-j}$ such that $m(X)=m_{-j}(X^{-j})$.    
\end{lemma}

\begin{proof}
    Firstly, we assume that $m(X)=m_{-j}(X^{-j})$, or equivalently, that $Y=m(X)+\epsilon=m_{-j}(X^{-j})+\epsilon$. Therefore, using that $\epsilon$ is independent from $X$ and that $m_{-j}(X^{-j})$ is constant conditionally on $X^{-j}$, then $y\perp\!\!\!\perp X^{j}|X^{-j}$.

    To prove the other way, we first observe that 
    \begin{align*}
        \e{y^2|X^{-j}}=\e{(m(X)+\epsilon)^2|X^{-j}}=\e{m(X)^2|X^{-j}}+\sigma^2,
    \end{align*}
    using that $\epsilon$ is centered and independent of $X$.
    On the other hand, we observe that using the conditional independence and also that $\epsilon$ is centered and independent of $X$ that 
    \begin{align*}
        \e{y^2|X^{-j}}&=\e{y(m(X)+\epsilon)|X^{-j}}\\&=\e{y|X^{-j}}\e{m(X)|X^{-j}}+\e{y\epsilon|X^{-j}}\\&=\e{m(X)|X^{-j}}^2+\sigma^2.
    \end{align*}
    Then, we obtained that as both quantities are equivalent that $\e{m(X)^2|X^{-j}}=\e{m(X)|X^{-j}}^2$. We observe that Jensen's inequality with an strict convex function is only achieved with degenerate distributions. Therefore, $m(X)$ is $\sigma(X^{-j})$-measurable and therefore there exists a measurable function that we denote $m_{-j}$ such that $m(X)=m_{-j}(X^{-j})$. 
    
\end{proof}

\begin{lemma}[LM for $y|X^{-j}$]
Under Assumption \ref{ass:LM} and Gaussian covariate, we have that $y=X^{-j}\beta'+\epsilon'$ with $\epsilon'$ independent from $X^{-j}$ and $\epsilon'\overset{\mathrm{i.i.d.}}{\sim}\mathcal{N}(0, \sigma^2+\Sigma_{\mathrm{cond}}{\beta^j}^2)$.\label{lemma:LMX_without_j}
\end{lemma}
\begin{proof}
Indeed, we can write 
    \begin{align*}
        y&=X\beta+\epsilon\\
        &= X^{-j}\beta_{-j}+X^j\beta^j+\epsilon\\
        &=  X^{-j}\beta_{-j}+(X^j-\nu_{-j}(X^{-j})+\nu_{-j}(X^{-j}))\beta^j+\epsilon\\
        &= X^{-j}\beta_{-j}+\nu_{-j}(X^{-j})+\left((X^j-\nu_{-j}(X^{-j}))\beta^j+\epsilon\right).
    \end{align*}
    Using that $X$ is Gaussian we have that $\nu_{-j}(X^{-j})=X^{-j}\gamma_{-j}$, therefore, we can write $\beta'=\beta_{-j}+\gamma_{-j}$. Finally, we have that $(X^j-\nu_{-j}(X^{-j}))$ is independent of $X^{-j}$ using the Gaussianity and that $(X^j-\nu_{-j}(X^{-j}))\sim \mathcal{N}(0, \Sigma_\mathrm{cond})$. Therefore, we have that $\epsilon'=(X^j-\nu_{-j}(X^{-j}))\beta^j+\epsilon\sim \mathcal{N}(0, \sigma^2+ (\beta^j)^2\Sigma_\mathrm{cond})$.
\end{proof}

\section{Additional experiments}

\subsection{Double Robustness of Sobol-CPI and nonnull bias}

In this section, we study a linear Gaussian setting. In this case, both methods need to implement computationally similar procedures because $\widehat{m}$, $\widehat{m}_{-j}$, and $\widehat{\nu}_{-j}$ are all linear. We have $y = X\beta + \epsilon$ with $\beta_0 = 0$, $\beta_{1:p-1} \sim \mathcal{N}(0, 25I_{p-1})$, $\epsilon \sim \mathcal{N}(0, 1)$, and $X \sim \mathcal{N}(0, \Sigma)$ where $\Sigma_{i,j} = 0.6^{|i-j|}$ and $p = 20$, we estimate $\widehat{\psi}_\mathrm{Sobol-cpi(1)}$, $\widehat{\psi}_\mathrm{Sobol-cpi(100)}$(i.e. with $n_\mathrm{cal}=100$) and $\widehat{\psi}_\mathrm{LOCO}$ using linear models trained with $n_\mathrm{train} = 50$ and averaged over $n_\mathrm{test} = 100000$.
On the top of Figure \ref{fig:append_doub-rob}, the bias when estimating the importance on the null coordinate, while on the bottom, it corresponds to a nonnull coordinate. On the left, we use the same training set to estimate $\widehat{m}$, $\widehat{m}_{-j}$, and $\widehat{\nu}_{-j}$. Conversely, on the right, we use different training samples while maintaining $n_\mathrm{train} = 50$. The histograms are based on 1000 runs.

\cref{fig:append_doub-rob} illustrates that CPI converges faster under the conditional null hypothesis due to double robustness, and that a significant loss is incurred by using the data splitting method proposed in \citet{williamson2023general}. However, we also observe that by increasing $n_\mathrm{cal}$, the double robustness of CPI can still be preserved. Additionally, the bias distribution for the non-null covariates is similar across methods. Consequently, there is improved variable selection without sacrificing variable importance.

\begin{figure}[htbp]
  \centering

  \begin{subfigure}{0.45\textwidth}
    \centering
    \includegraphics[width=\linewidth]{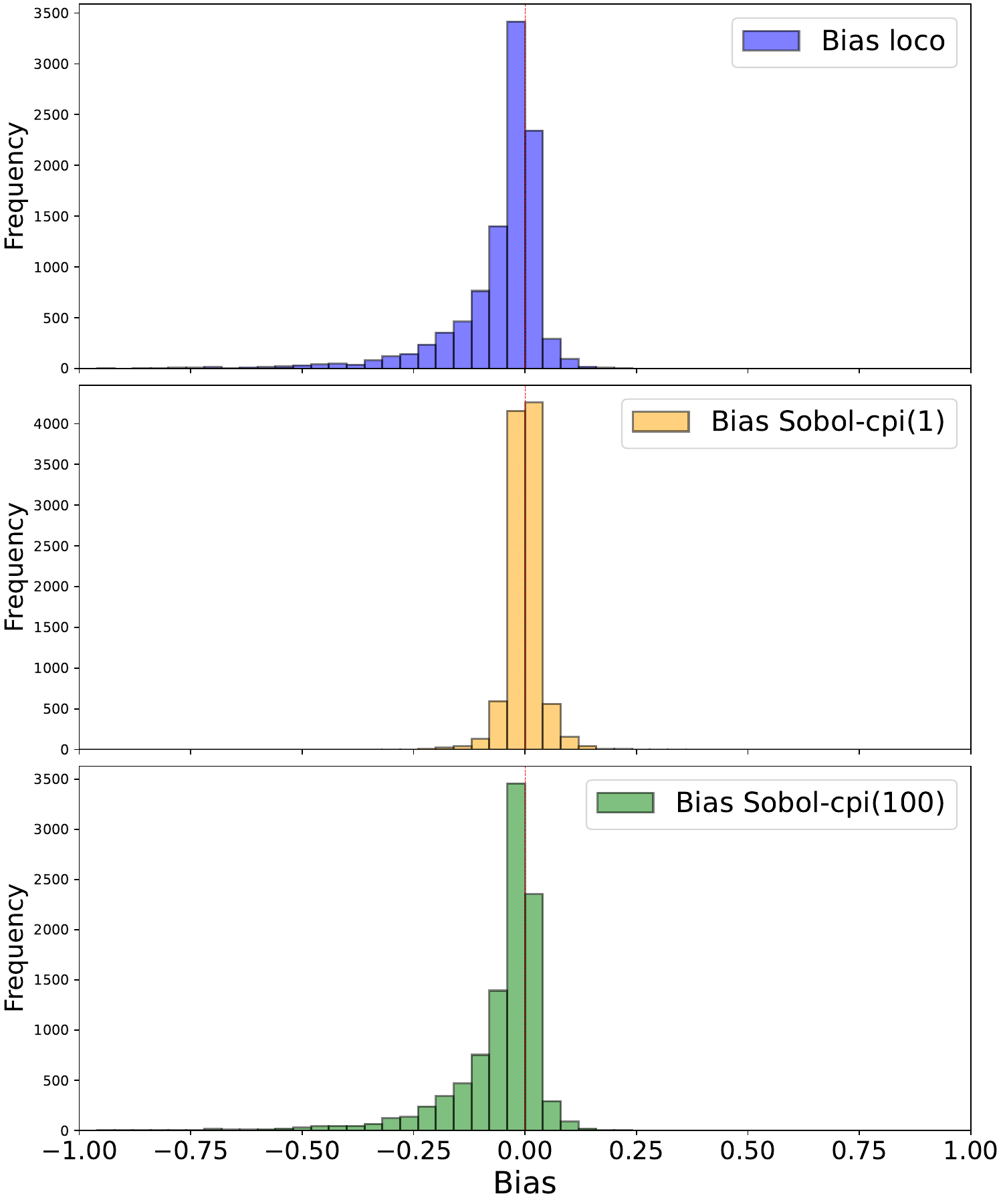}
    \caption{Same training sample, null covariate}
  \end{subfigure}
  \hfill
  \begin{subfigure}{0.45\textwidth}
    \centering
    \includegraphics[width=\linewidth]{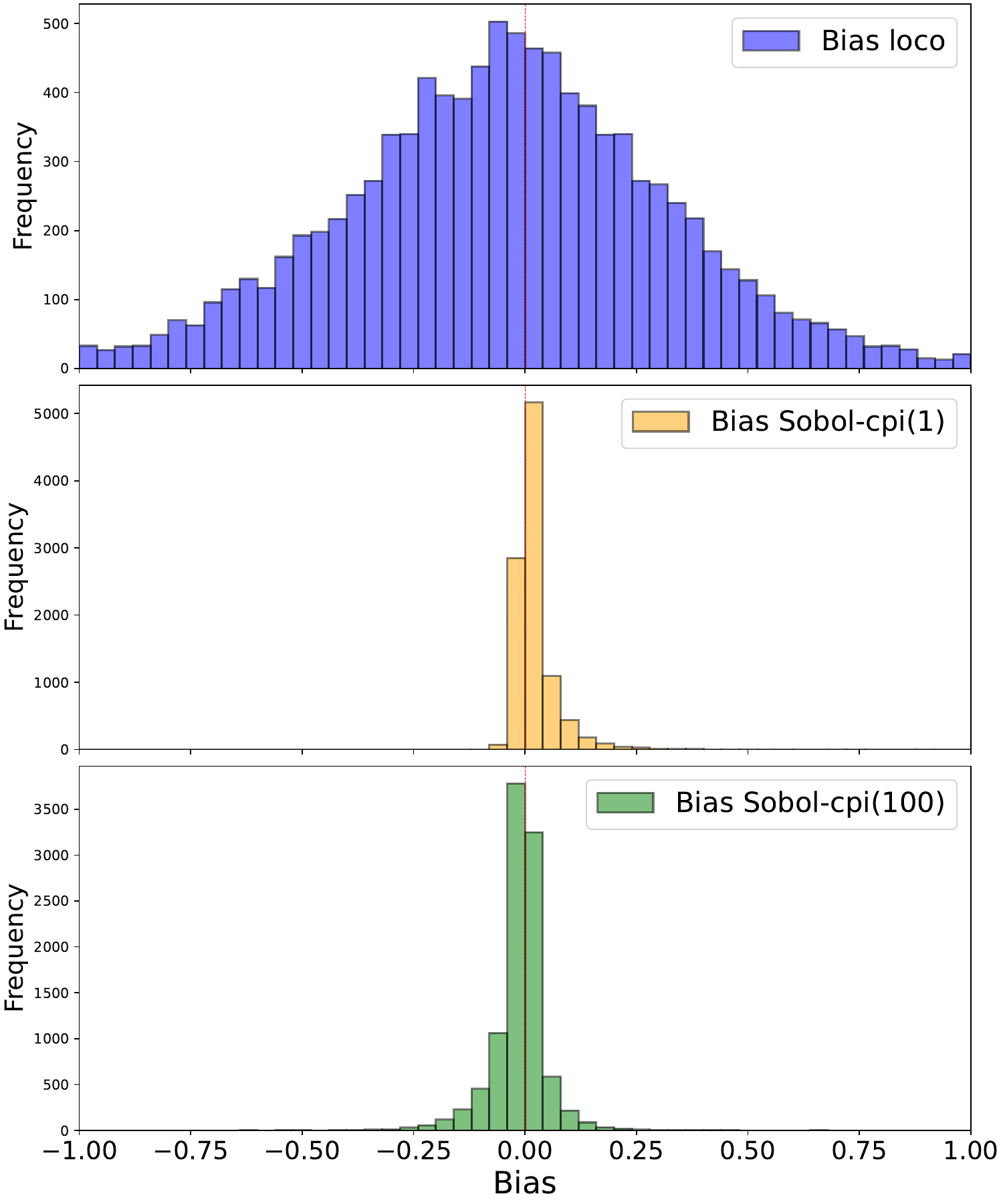}
    \caption{Different training sample, null covariate}
  \end{subfigure}

  \vskip\baselineskip

  \begin{subfigure}{0.45\textwidth}
    \centering
    \includegraphics[width=\linewidth]{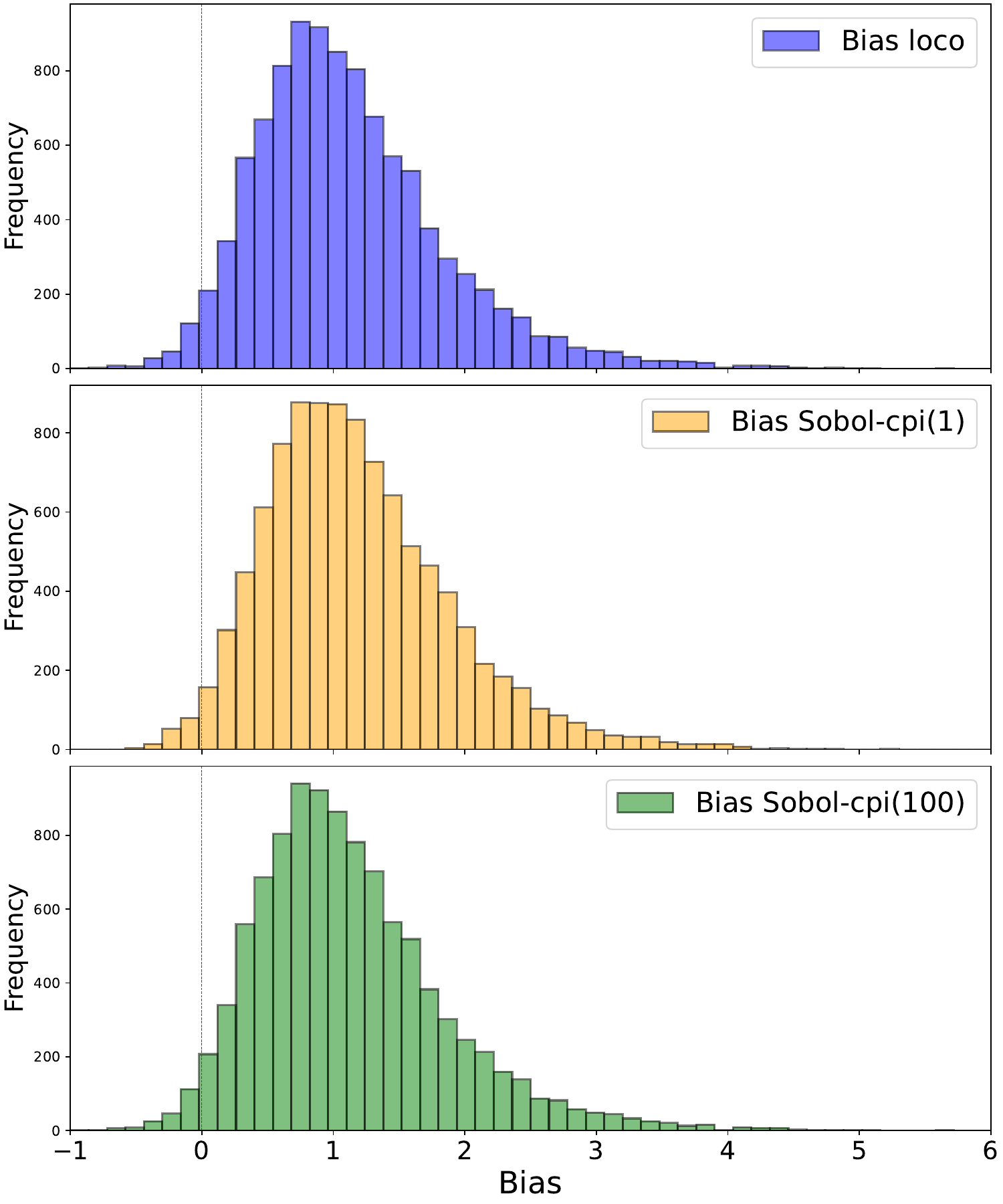}
    \caption{Same training sample, nonnull covariate}
  \end{subfigure}
  \hfill
  \begin{subfigure}{0.45\textwidth}
    \centering
    \includegraphics[width=\linewidth]{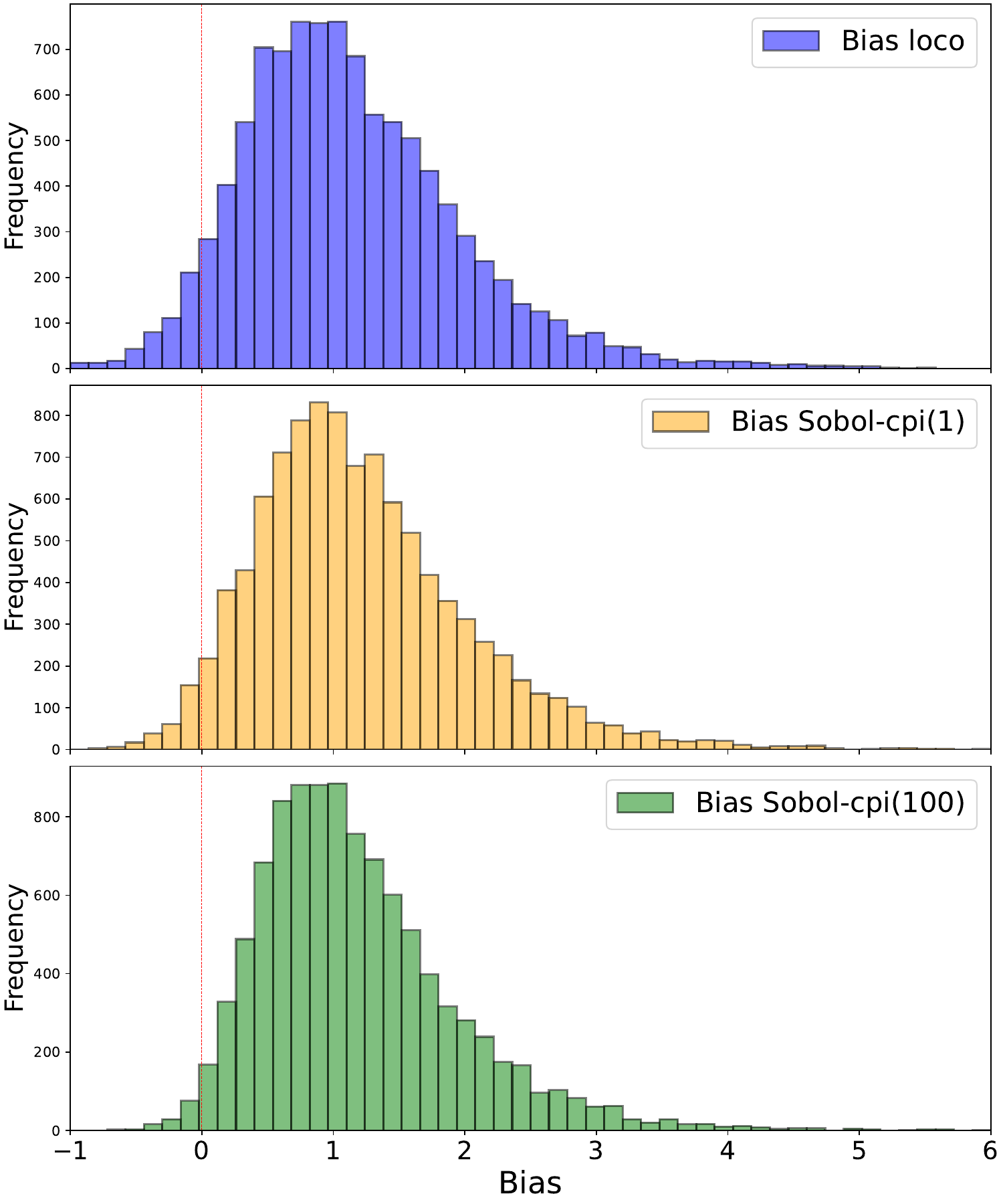}
    \caption{Different training sample, nonnull covariate}
  \end{subfigure}

  \caption{\textbf{Double robustness of the Sobol-CPI:} 
   The empirical bias distribution of LOCO, Sobol-CPI(1), and Sobol-CPI(100). From (a) and (b), we observe the benefits of using a CPI-based approach, as its double robustness results in lower bias. In (c) and (d), we see that the estimation error for a non-null covariate is similar. Comparing (a) and (b), we observe the negative effect of data splitting.
  }
  \label{fig:append_doub-rob}
\end{figure}

    



\subsection{Effect of $n_\mathrm{cal}$}\label{subsec:exp-n-cal}

In general, in order to get a valid variable importance that is asymptotically efficient, Theorem \ref{th:asymp_eff} shows that we need to take a calibration set size that increases with $n$. Nevertheless, as seen in Proposition \ref{prop:fix_n_cal}, with $\ell=\ell_2$, it is possible to fix the calibration set size and correct the bias generated. Moreover, when $n_\mathrm{cal}$ is fixed to 1, as it is just a correction of CPI, it benefits from its double robustness, making it easier to separate null covariates from important ones. In this way, this \textit{hyperparameter} represents a trade-off between \textit{variable selection} and \textit{variable importance}. With a large $n_\mathrm{cal}$, we benefit from the asymptotic efficiency, obtaining a more robust estimate for the important covariates; however, for the null covariates, it yields a worse estimate. Indeed, standard CPI does not converge and is not asymptotically efficient when the importance is not null. This can be seen because the optimality assumption (Assumption A1 in \cref{sect:proofTheo}) does not hold. 

To avoid the first-order contribution of having to estimate the regressor $m$, which is the maximizer of $f \mapsto \mathbb{E}[\ell(f(X), y)]$ over $\mathcal{F}$, it is reasonable to assume that 
\[
\frac{d}{d\epsilon}\mathbb{E}[\ell(m_\epsilon(X), y)]\big|_{\epsilon=0}=0,
\]
for any smooth path $\{m_\epsilon:-\infty<\epsilon<\infty\}\subset \mathcal{F}$. Nevertheless, this no longer holds for CPI. Indeed, we are not reoptimizing a learner with the empirical conditional distribution; we are only substituting the optimizer of the original distribution on another distribution. Therefore, this first-order term is not expected to vanish.

This trade-off between variable selection and variable importance can be observed, for instance, in \cref{fig:n_cal}. In this Figure, we study the estimated importance of two covariates in a modified version of the nonlinear setting from \citet{scornet2022mda}: $ y = X_0 X_1 \mathbb{I}_{X_2 > 0} + 2 X_3 X_4 \mathbb{I}_{X_2 < 0} $, with $X \sim \mathcal{N}(\mu, \Sigma) $, where $ \Sigma_{i,j} = \rho^{|i-j|}$, $p = 50 $, $n=10000$ and $\mu = \mathbf{0}$. The $x$-axis represents $\rho$, while the $y$-axis represents the estimated importance. The experiment is repeated 30 times.  

For the important covariate, we observe slightly better results in estimating importance. However, for the null covariates, the results are slightly worse when $n_\mathrm{cal}$ is larger. This is also remarkable in \cref{fig:corr_nonlin}.

\begin{figure}[htbp]
    \centering
        \includegraphics[width=0.8\textwidth]{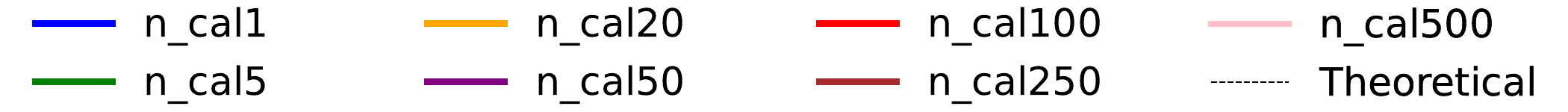}
        \includegraphics[width=0.9\textwidth]{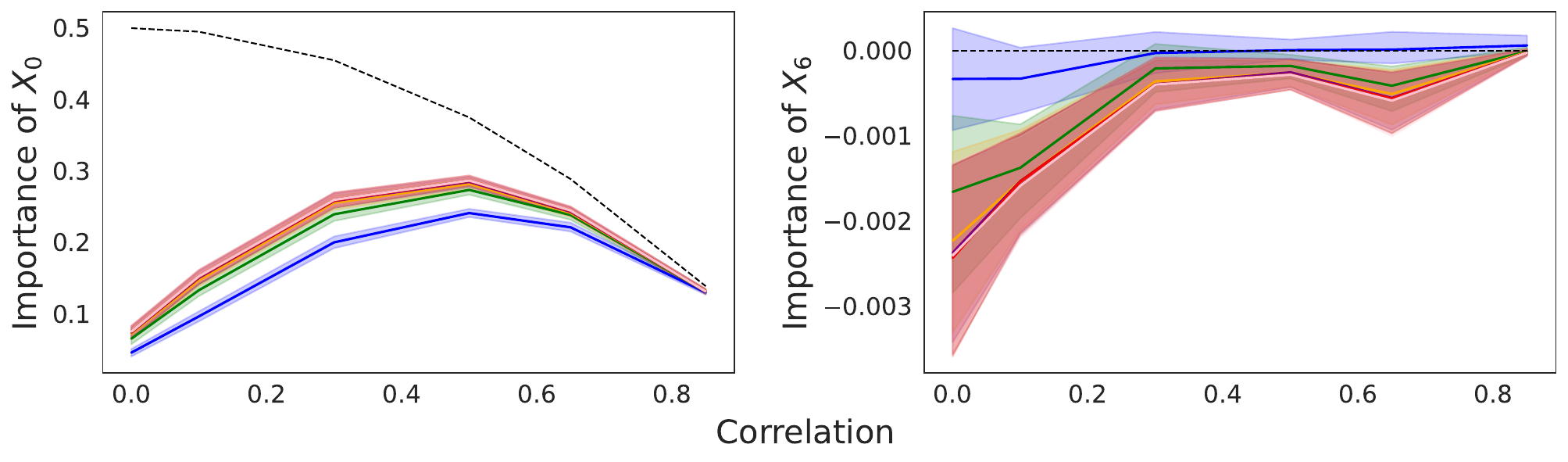}
    \caption{\textbf{Calibration set size effect as a trade-off between Variable Importance and Variable Selection:}  
Total Sobol Index estimation in a nonlinear setting. The first figure represents an important covariate ($X_0$), while the second represents a non-important covariate ($X_6$). We observe that with a larger $n_\mathrm{cal}$, the importance estimation of the non-null covariate is slightly improved, enhancing variable importance. However, for the null covariate, there is a slightly greater bias, making variable selection less accurate. }
     \label{fig:n_cal}
\end{figure}

\subsection{Correlation}\label{subsec:corr}
In this section, \( \widehat{m} \) and \( \widehat{m}_{-j} \) are Gradient Boosting models, while \( \widehat{\nu}_{-j} \) is a Lasso model. In this experiments, $X \sim \mathcal{N}(\mu, \Sigma) $, where $ \Sigma_{i,j} = \rho^{|i-j|}$, $p = 50 $ and $\mu = \mathbf{0}$. The $x$-axis is for $\rho$.

In \cref{fig:corr_poly}, we examine a polynomial setting where the important covariates are randomly sampled with a sparsity of $0.1$, with $n=1000$. On the left the AUC and the right is the mean bias across the null covariates. This experiment is run over 30 repetitions. We observe that the AUC values are similar, but the LOCO method fails to achieve null importance for the null covariates. Additionally, the LOCO method is significantly more computationally expensive.

\begin{figure}[htbp]
    \centering
        \includegraphics[width=0.4\textwidth]{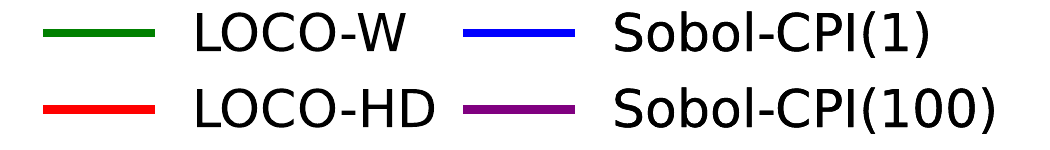}
        \includegraphics[width=0.7\textwidth]{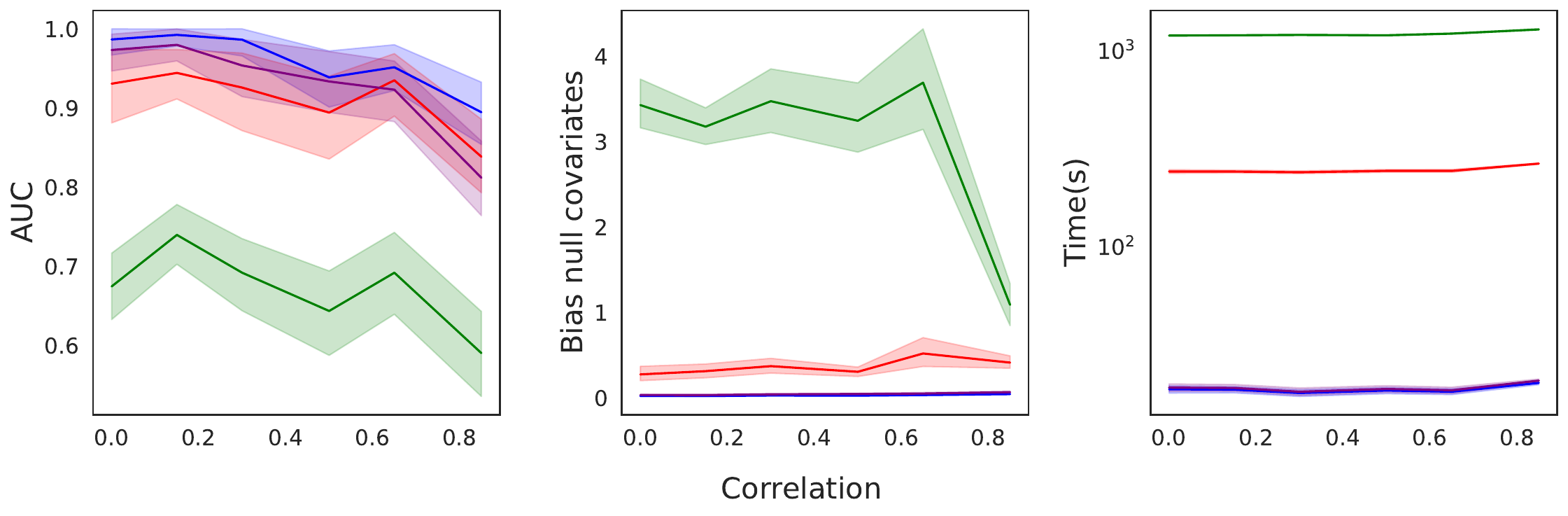}
     \caption{\textbf{Correlation Effect in a Polynomial Setting:} For all correlation values, Sobol-CPI achieves better discrimination of important covariates, assigns no importance to null covariates, and is significantly more computationally efficient. }
    \label{fig:corr_poly}
\end{figure}

In \cref{fig:corr_nonlin}, we use the same standard nonlinear setting: $y = X_0 X_1 \mathbb{I}_{X_2 > 0} + 2 X_3 X_4 \mathbb{I}_{X_2 < 0},$
with $n=10000$. On the top, the first two figures display the estimated importance of the important covariates, while the third figure represents an unimportant covariate. They are theoretically obtained in \cref{ex:LOCO-nonlin}. On the bottom, the left figure shows the AUC, the middle figure presents the mean bias across the null covariates, and the right figure is the computational time. This experiment is conducted over 10 repetitions.

The first two figures show that the importance scores from the Sobol-CPI method with a larger \( n_\mathrm{cal} \) are more accurate, though still comparable to the LOCO method. From the third figure, we observe that Sobol-CPI maintains double robustness with larger \( n_\mathrm{cal} \), while LOCO exhibits a large bias. 

In the bottom row, the first figure demonstrates the superior discriminant power of CPI-based methods. The second figure highlights the bias of the LOCO method for null covariates, and the third figure illustrates that CPI-based methods are much faster than the LOCO methods. This is because CPI-based methods avoid refitting a Gradient Boosting model for each covariate, using instead a Lasso model.


 \begin{figure}[htbp]
  \centering

  \begin{subfigure}{0.5\textwidth}
    \centering
    \includegraphics[width=\linewidth]{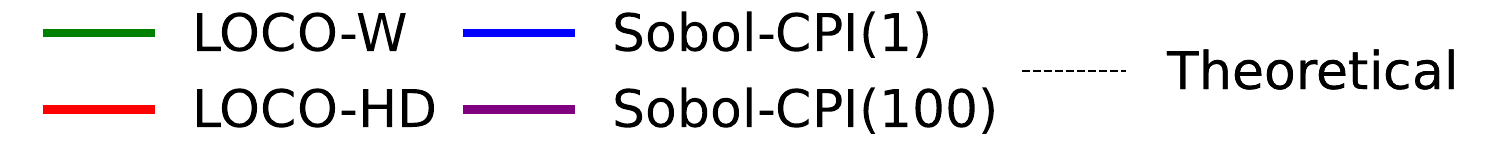}
  \end{subfigure}

  \vspace{0.8\baselineskip}

  \begin{subfigure}{0.7\textwidth}
    \centering
    \includegraphics[width=\linewidth]{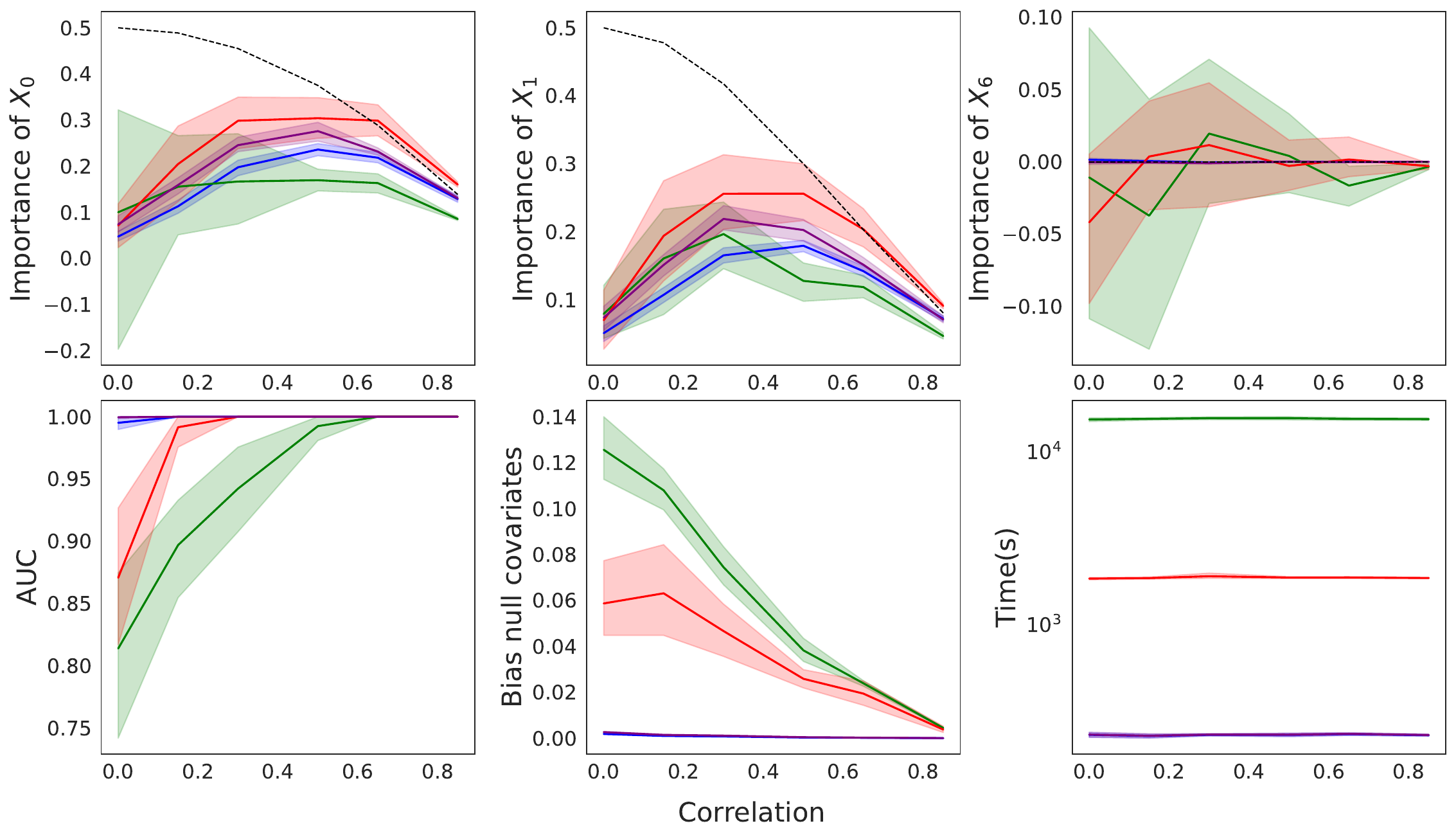}
  \end{subfigure}

  \caption{\textbf{Correlation effect in non-linear setting:} Sobol-CPI and LOCO achieve comparable results for important covariates while providing double robustness for null covariates in a significantly faster manner.}
  \label{fig:corr_nonlin}
\end{figure}




\subsection{Inference through variance correction}\label{subsec:app_inf}
In this section, we compare the power and type-I error of different variance correction methods across LOCO and Sobol-CPI, using various values of $n_\mathrm{cal}$ in both linear and polynomial correlated settings.

As stated in Appendix~\ref{app:inf_lin}, different variance estimators have been studied: bootstrap-based estimates for the methods with the $\_\mathrm{bt}$ and $\_\mathrm{n2}$ suffixes, and sample variance divided by the sample size for the others for the remaining methods.

When the method has no suffix (denoted by a circle marker), this indicates that no variance correction is applied. This approach was used by \cite{chamma2024statistically}, who ignored the vanishing variance; therefore, if the variance is zero, the null hypothesis is retained directly, without an statistical test. The suffix $\_\mathrm{sqrt}$ denotes a square root correction term, $\_\mathrm{n}$ and $\_\mathrm{bt}$ indicate a linear correction term, and $\_\mathrm{n2}$ represents a quadratic correction term. $\_\mathrm{bt}$ uses bootstrap from the variance estimation.

The experiments are run 100 times for the linear setting and 30 for the polynomial.

\paragraph{Linear setting:} We study a linear setting with varying correlations ($\rho \in \{0.3, 0.6, 0.8\}$) to examine their effect on the power of the methods. However, since the results are similar across different values of $\rho$, we present only the graphics for $\rho = 0.6$.

In these experiments, \( \widehat{m} \), \( \widehat{m}_{-j} \), and \( \widehat{\nu}_{-j} \) are linear models (see \cref{lemma:GaussianAss}).
More formally, $y = X\beta + \epsilon$ with $\beta$ sparse with sparsity $0.25$ and value 1 when non-null, $\epsilon \sim \mathcal{N}(0, \|X\beta\|^2/\mathrm{snr})$ with $\mathrm{snr}=2$, and $X \sim \mathcal{N}(0, \Sigma)$ where $\Sigma_{i,j} = 0.6^{|i-j|}$ and $p = 100$.

From Figure~\ref{fig:lin_0.6}, we observe that, despite the computational cost being similar across all methods due to the use of linear models (with the bootstrap-based covariance methods being more computationally intensive), Sobol-CPI demonstrates significantly smaller bias—not only for the null covariates but also for the non-null covariates.

We first note that \texttt{LOCO-W}~\citep{williamson2023general} requires a larger sample size to achieve type-I error control. Consequently, while LOCO-W is capable of identifying significant covariates, it also produces a substantial number of false positives. Furthermore, we observe that Sobol-CPI(1) is the most powerful method, benefiting from its double robustness property.



\begin{figure}[htbp]
    \centering
        \includegraphics[width=0.4\textwidth]{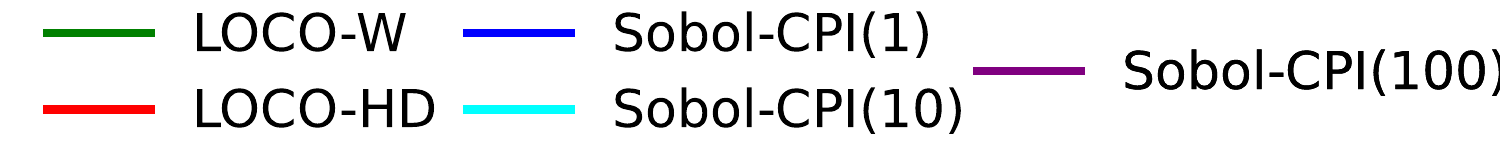}
        \includegraphics[width=0.8\textwidth]{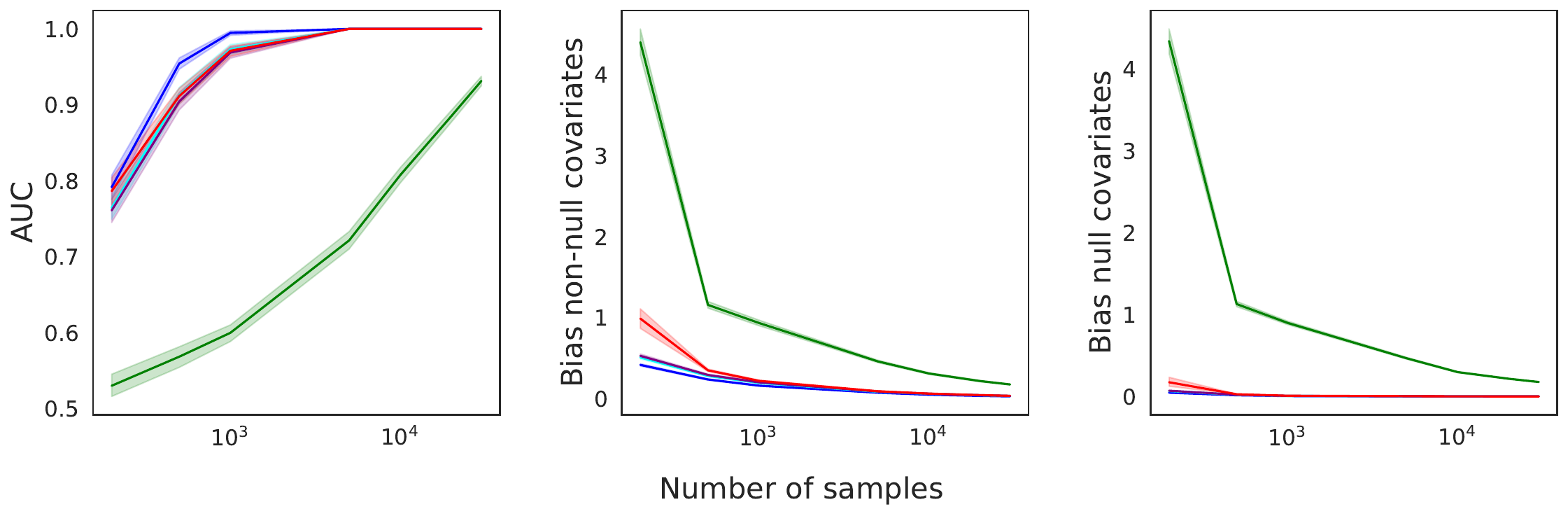}
        \includegraphics[width=0.7\textwidth]{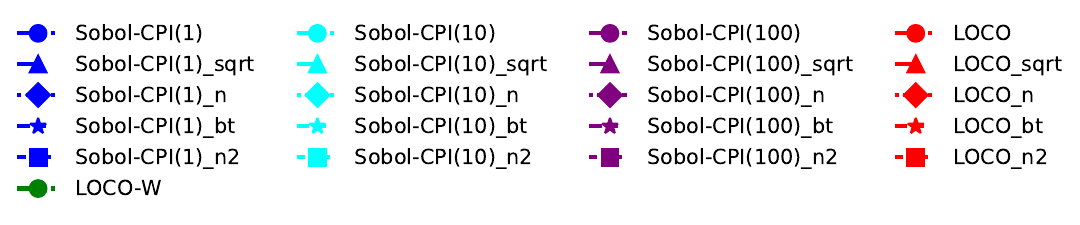}
\includegraphics[width=0.8\textwidth]{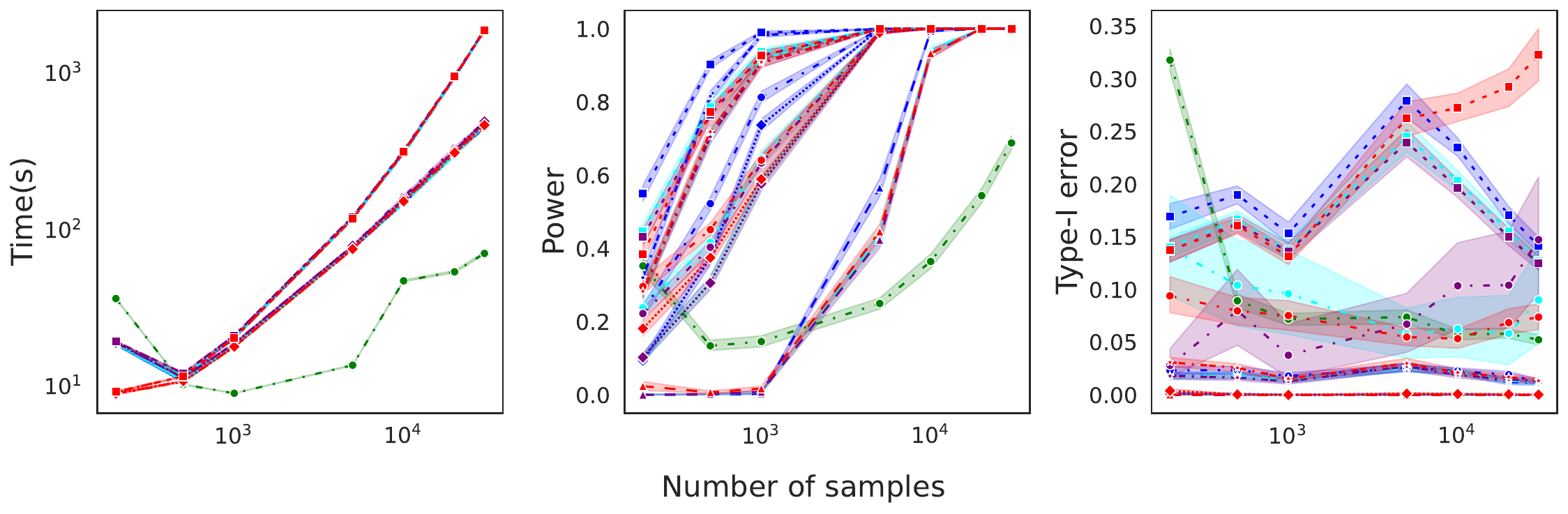}
     \caption{\textbf{General comparison of LOCO and Sobol-CPI in linear setting:}  From left to right, and top to bottom, we have the AUC to detect the null covariates, the bias in estimating the non-null covariates, the bias in estimating the null covariates, the time to compute the estimates and the tests, the power of the tests, and the type-I error. Sobol-CPI presents lower bias, more power with a valid type-I error with a linear additive correction. 
    }
    \label{fig:lin_0.6}
\end{figure}

In Figures~\ref{fig:type_lin_0.6} and~\ref{fig:power_lin_0.6}, we observe a detailed comparison of the type-I error and power, respectively.
Among the methods, the most powerful tests involve the quadratic correction. Nevertheless, these do not control the type-I error. Additionally, the uncorrected method (circle marker) also fails to control the type-I error due to vanishing variance. As shown in Section~\ref{subsec:app_inf}, the linear correction successfully preserves the type-I error. 

We observe that the most powerful method is $\mathrm{Sobol\text{-}CPI}(1)$. When using a larger calibration size $n_\mathrm{cal}$—which serves as a trade-off between variable selection and variable importance—the power slightly decreases due to the double robustness property.

Finally, we observe that in this case, there is a substantial gain when computing the variance using the bootstrap compared to using the sample variance in all the procedures. Both methods are valid, as the validity stems from the additive term, not from the specific choice of variance estimator.

\begin{figure}[htbp]
    \centering
    \includegraphics[width=0.7\textwidth]{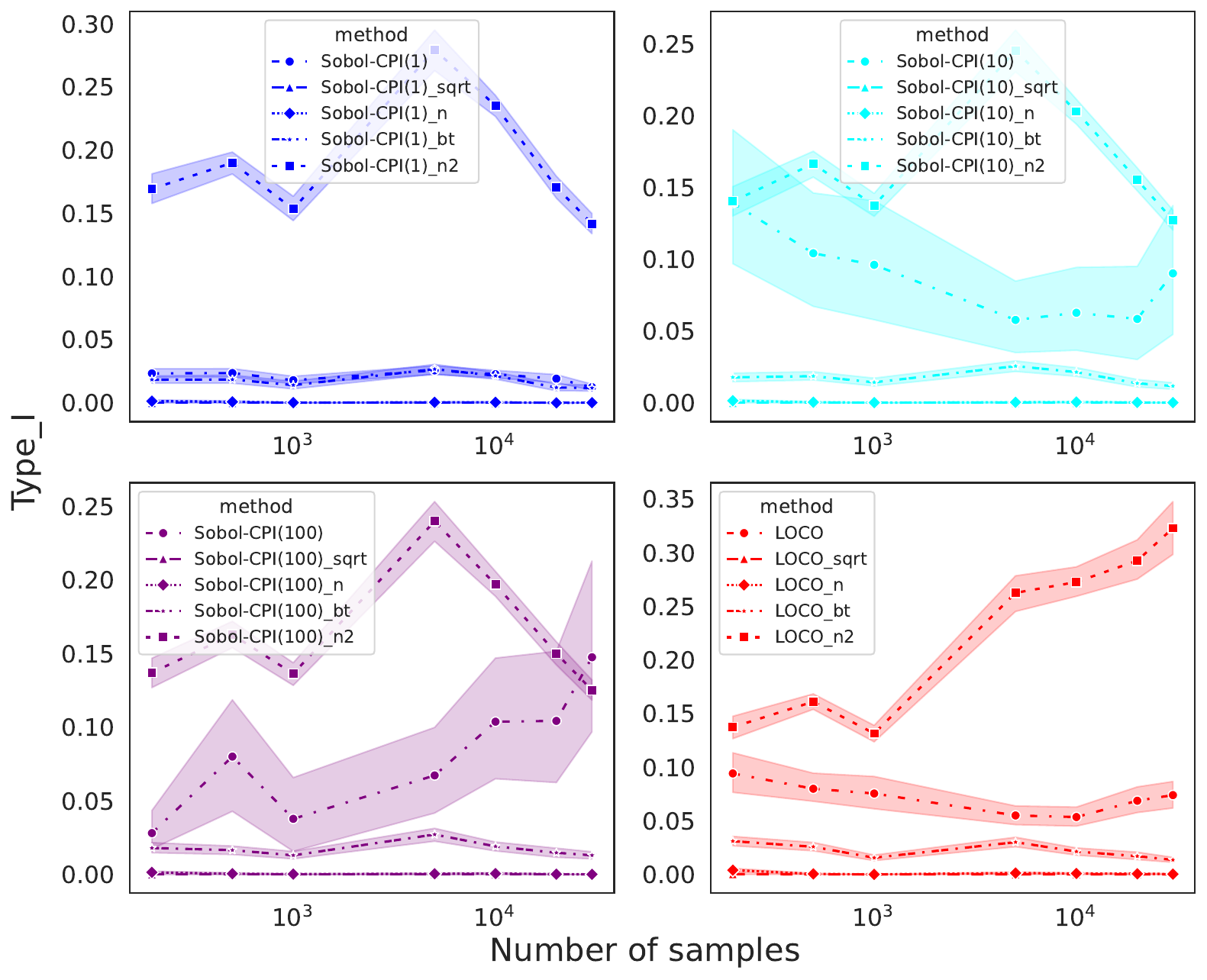}
    \caption{\textbf{Type-I error of Sobol-CPI and LOCO with different variance corrections in the Linear Setting:}  
    The first three figures show different corrections applied to Sobol-CPI with varying $n_\mathrm{cal}$, followed by the results for LOCO. The quadratic correction is not enough to control the type-I error.}
    \label{fig:type_lin_0.6}
\end{figure}

\begin{figure}[htbp]
    \centering
    \includegraphics[width=0.7\textwidth]{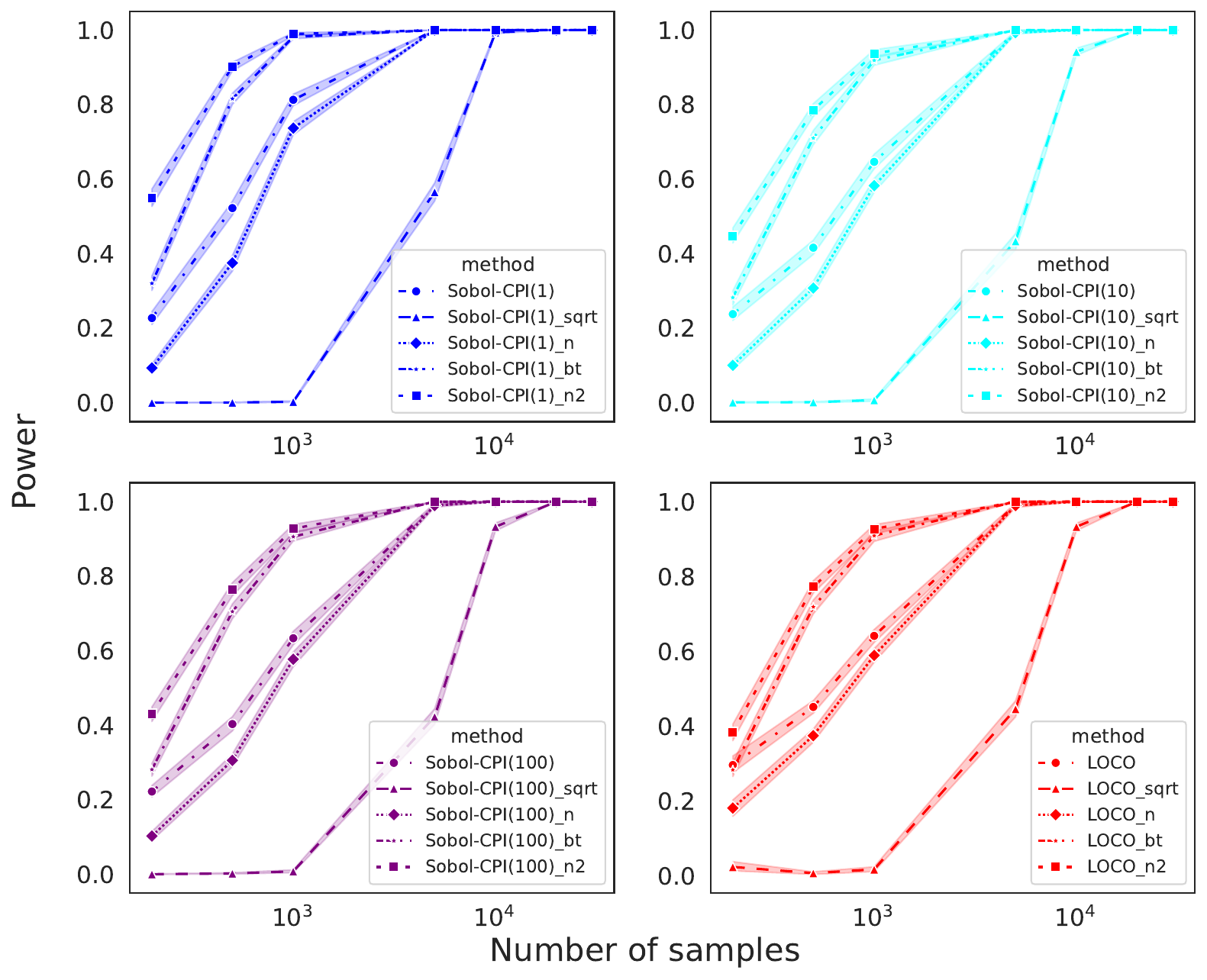}
    \caption{\textbf{Power of Sobol-CPI and LOCO with different variance corrections in the Linear Setting:}  
    The first three figures show different corrections applied to Sobol-CPI with varying $n_\mathrm{cal}$, followed by the results for LOCO. Sobol-CPI(1) is the most powerful method. Across the corrections, the linear decaying term is the least conservative, with variance estimated via bootstrap.}
    \label{fig:power_lin_0.6}
\end{figure}





\paragraph{Polynomial setting:} We study a polynomial setting with varying correlations($\rho \in \{0.3, 0.6, 0.8\}$) to examine their effect on the power of the methods. Nevertheless, we only report the results for $\rho=0.6$ as the other correlations give qualitatively similar results.  In this experiment, the input matrix is generated as before ( $X \sim \mathcal{N}(0, \Sigma)$ where $\Sigma_{i,j} = \rho^{|i-j|}$ and $p = 50$ ), but the output is polynomial in the input with interactions and degree three. The important covariates are randomly sampled with a sparsity of $0.25$. 

In all these experiments, \( \widehat{m} \) and \( \widehat{m}_{-j} \) are Gradient Boosting models, and \( \widehat{\nu}_{-j} \) is a Lasso. Therefore, in this setting, there is a clear computational benefit to using a permutation-based approach, as seen in the bottom-left panel of Figure~\ref{fig:poly_0.6}, where a substantial gain is observed. However, the benefits are not purely computational. At the top of the figure, we observe that the Sobol-CPI demonstrates superior classification performance, as indicated by a higher AUC. Additionally, it does not exhibit any bias on the null covariates, which is significant for the LOCO methods.

Regarding hypothesis testing, the conclusions are not entirely clear from this figure alone; we refer the reader to Figures~\ref{fig:type_poly_0.6} and~\ref{fig:power_poly_0.6} for a more detailed comparison. Nevertheless, it is evident that the Sobol-CPI methods are the most powerful, showing a clear separation in power from the LOCO methods. Moreover, the quadratic and linear corrections with bootstrap variance are not strict enough to control the type-I error. The \texttt{LOCO-W} method, in particular, fails to control the type-I error altogether—so even though it may yield discoveries, they are not reliable due to a high number of false positives. This can also be explained by the AUC: even with a large sample size, there is no clear distinction between important and null covariates.


\begin{figure}[htbp]
    \centering
        \includegraphics[width=0.4\textwidth]{Sobol-CPI/experiments/final-images/compact_images/legend_power_three_col.pdf}
        \includegraphics[width=0.6\textwidth]{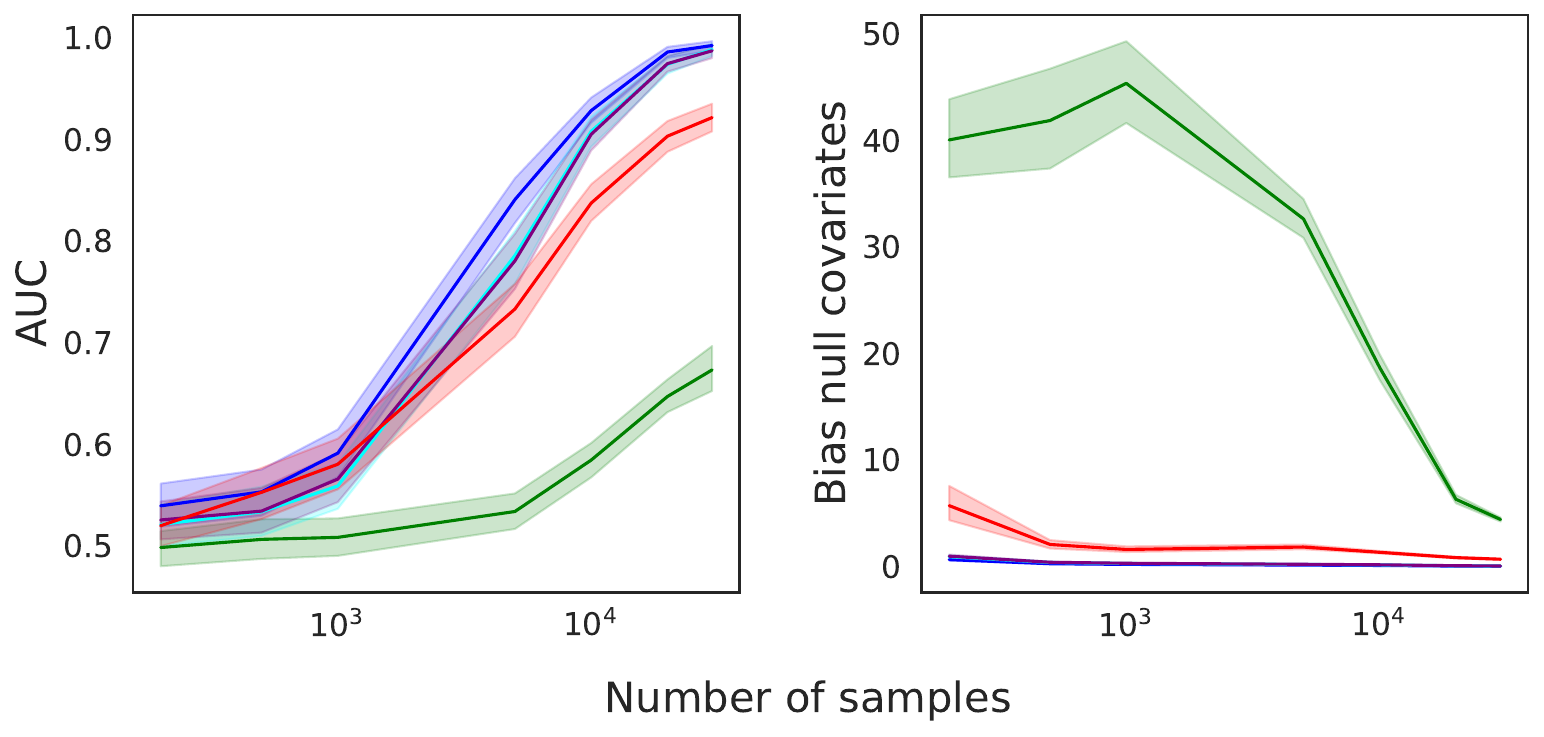}
        \includegraphics[width=0.7\textwidth]{Sobol-CPI/experiments/final-images/appendix/legend_inf_four_col.pdf}
        \includegraphics[width=0.8\textwidth]{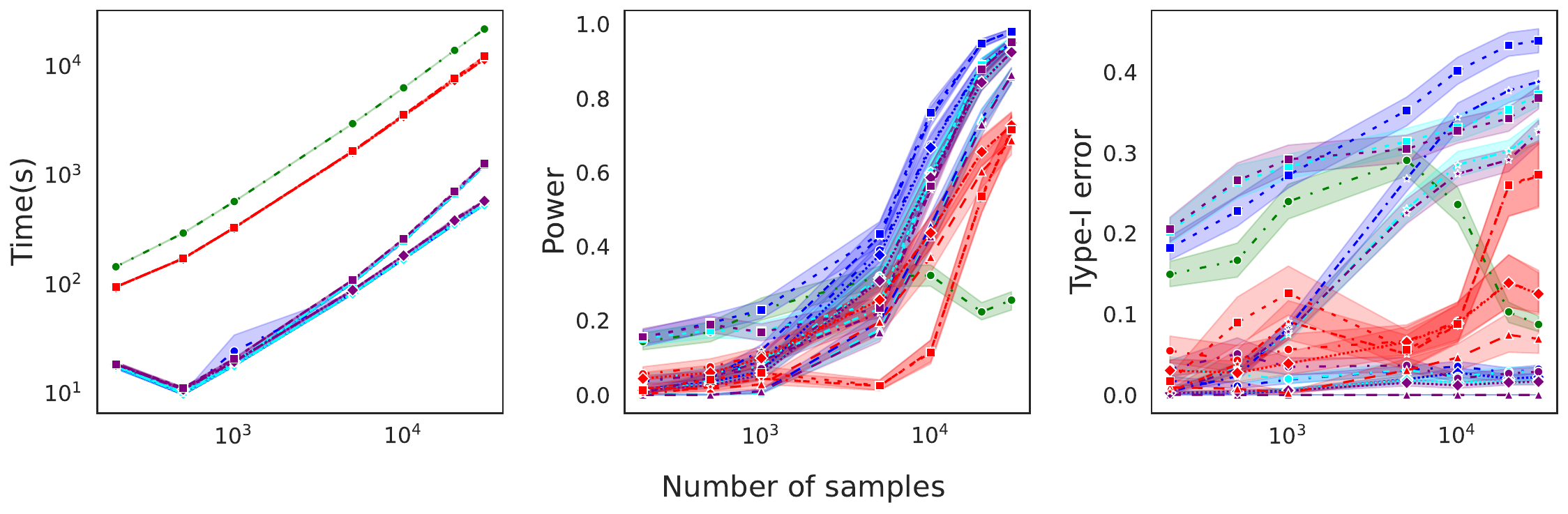}
      \caption{\textbf{General comparison of LOCO and Sobol-CPI in polynomial setting:} From left to right, and top to bottom, we have the AUC to detect the null covariates, the bias in estimating the null covariates, the time to compute the estimates and the tests, the power of the tests, and the type-I error. A greater gain is achieved with Sobol-CPI in these more complex settings, as it attains better results on the AUC without introducing bias on null covariates while requiring less computational effort.
      }
    \label{fig:poly_0.6}
\end{figure}

In Figure~\ref{fig:type_poly_0.6}, we observe, similarly to the linear case, that the quadratic correction is not sufficient to control the type-I error. However, we also see that the linear correction with the variance estimated via bootstrap is not sufficient either. This contrasts with the linear setting, and the reason lies in the fact that the convergence rate result (see \cref{lemm:doubl-rob-LM}), which allows the use of Markov's inequality (see \cref{app:inf_lin}) to guarantee type-I error control, is only available in the linear setting.

Additionally, we observe that for LOCO, even with the square root correction, the type-I error slightly exceeds the nominal level $\alpha = 0.05$, due to the high variability of the method. Finally, for Sobol-CPI, the square root correction does control the type-I error, but it may be overly conservative, as the observed error is exactly zero.

\begin{figure}[htbp]
    \centering
    \includegraphics[width=0.7\textwidth]{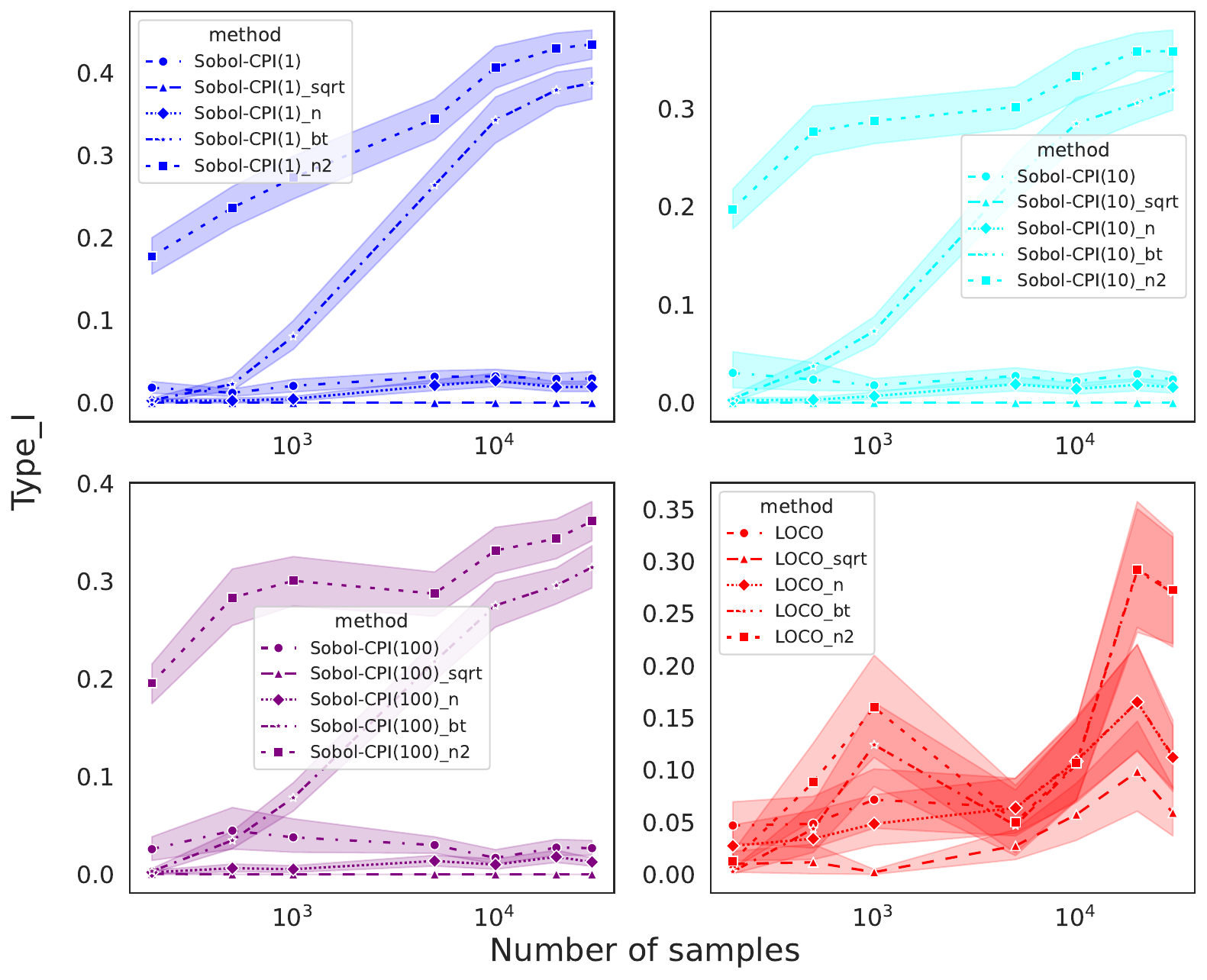}
    \caption{\textbf{Type-I error of Sobol-CPI and LOCO with different variance corrections in the Polynomial Setting:} The first three figures show different corrections applied to Sobol-CPI with varying $n_\mathrm{cal}$, followed by the results for LOCO. The linear correction with bootstrap variance is not enough to control the type-I error. 
}
    \label{fig:type_poly_0.6}
\end{figure}

In \cref{fig:power_poly_0.6}, we observe that Sobol-CPI(1) is the most powerful procedure, benefiting explicitly from double robustness. Additionally, across all methods, there is a decrease in power associated with more restrictive corrections. The gap observed with the square root correction arises from its overly conservative nature.

\begin{figure}[htbp]
    \centering
    \includegraphics[width=0.7\textwidth]{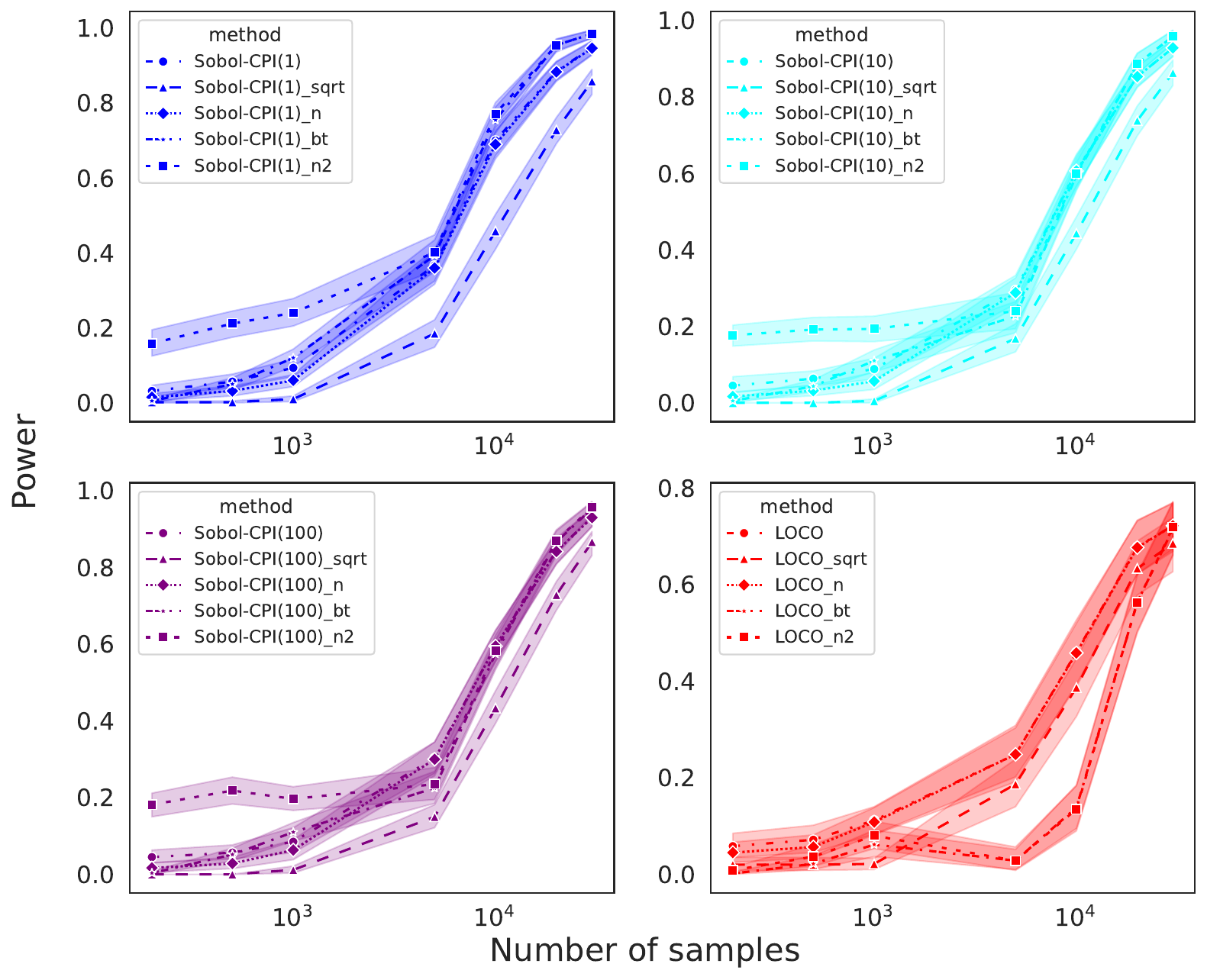}
    \caption{\textbf{Power of Sobol-CPI and LOCO with different variance corrections in the Polynomial Setting:}  
    The first three figures show different corrections applied to Sobol-CPI with varying $n_\mathrm{cal}$, followed by the results for LOCO. Sobol-CPI(1) is the most powerful method. Across the corrections, the linear decaying term is the least conservative, with variance estimated via bootstrap.}
    \label{fig:power_poly_0.6}
\end{figure}

\subsection{Real-data experiments}\label{app:real_data}

\subsubsection{Datasets}\label{app:datasets}

For empirical evaluation, this study uses five commonly referenced benchmark datasets from the machine learning literature.

The \emph{Wisconsin Diagnostic Breast Cancer} (WDBC) dataset consists of $n=569$ patient samples with $p=30$ numeric features of individual cells, obtained from a minimally invasive fine needle aspirate, to discriminate benign from malignant breast lumps. This dataset is widely used for binary classification tasks in medical informatics \citep{wdbc1995}.

The \emph{Wine Quality} dataset \citep{wine_quality_186} comprises two related subsets: \emph{Wine-Red} and \emph{Wine-White}. The red wine subset contains $n=1{,}599$ samples, and the white wine subset contains $n=4{,}898$ samples, each with $p=11$ physicochemical input variables and a sensory quality score as the target. 

The \emph{Diabetes} dataset (from scikit-learn’s \citep{scikit-learn} \texttt{load\_diabetes} function) is a regression dataset with $n=442$ samples and $p=10$ variables: age, sex, body mass index, average blood pressure, and six blood serum measurements, as well as the response of interest, a quantitative measure of disease progression one year after baseline.

Finally, the \emph{California Housing} dataset \citep{pace1997} contains $n=20{,}640$ samples of aggregated census data with $p=8$ numeric features describing socio-economic and geographic characteristics of California districts, and a continuous target variable representing median house value. 

\subsubsection{Double robustness and inference}

Since in real data the importance of a feature (or even knowing whether a feature is important or not) is unknown, we added an artificial feature which is a function of the original ones plus an independent noise term. Therefore, we can quantify the correlation with this feature to make the setting more complex, and we can also quantify the number of false positives by repeating the experiment across different seeds and counting the number of times this feature is selected. The discoveries are defined as the number of original features that are selected. We study the double robustness in practice, which is visible since the CFI methods assign a null importance to the artificial feature. We also study the power of the correction procedures and tests considered.

We investigate tests corrected by additive terms, which may be of order $\sqrt{n}$ as proposed by \citet{locoVSShapley}, or linear as we propose for the linear setting, where the faster convergence rate can be exploited (Lemma~\ref{lemm:doubl-rob-LM}). In particular, we also tested quadratic corrections, which failed to control the error. Additionally, we considered nonparametric tests, namely the Wilcoxon and the Sign test. We ran these simulations over all datasets from Appendix~\ref{app:datasets} and across standard ML models: Lasso, Gradient Boosting, Random Forest, Neural Network, and a SuperLearner combining the previous ones. Since the results were consistent across all experiments, we only report those for the SuperLearner on WDBC (Figure~\ref{fig:app_wdbc}) and red wine quality (Figure~\ref{fig:app_red_wine}).

In Figure~\ref{fig:app_wdbc}, we first observe that, due to double robustness, methods that do not rely on refitting are able to detect the null feature faster and with less variance. We also observe that Sobol-CPI(1) with the Wilcoxon test is the most powerful test for any correlation level, while methods based on additive corrections are too conservative. We further observe that the raw methods based on a t-test (\texttt{Sobol-CPI(1)} or \texttt{LOCO}) fail to control the error due to vanishing variance. Finally, we point out the clear time gain obtained by avoiding refitting a SuperLearner for each feature.

\begin{figure}[htbp]
    \centering
    \includegraphics[width=1.0\textwidth]{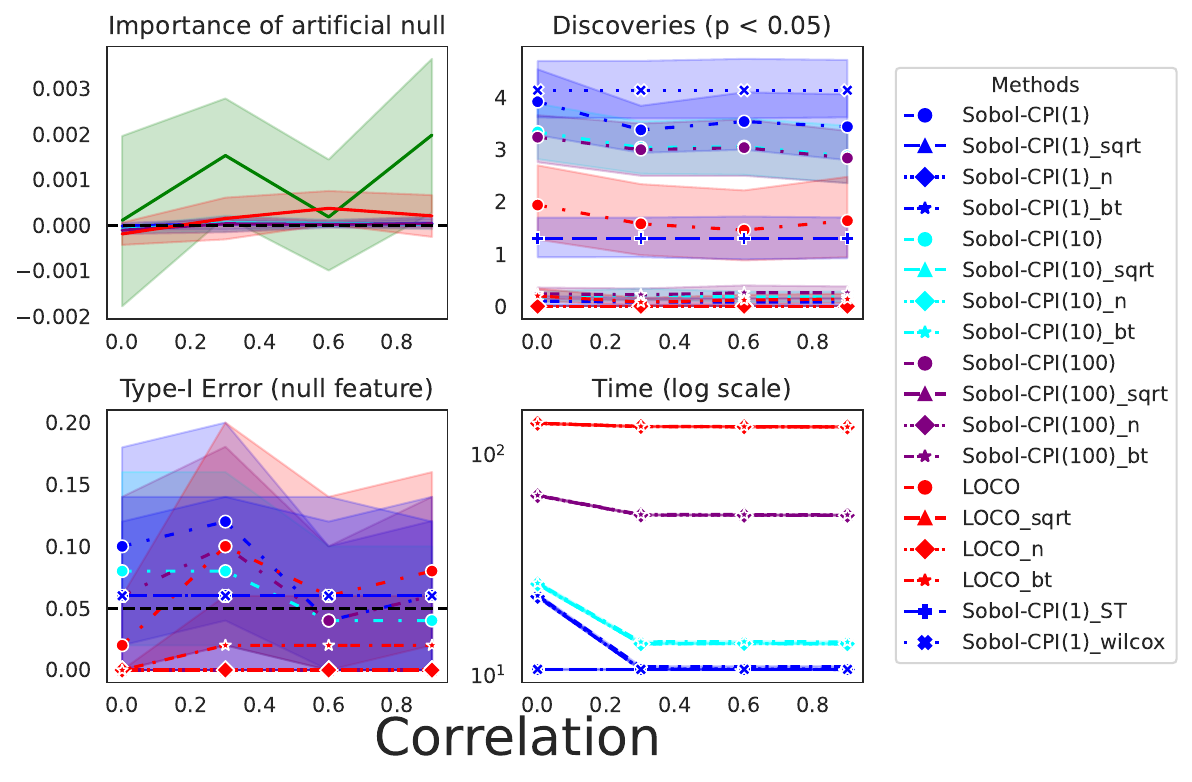}
    \caption{\textbf{Double robustness and inference on the WDBC dataset with SuperLearner:} Sobol-CPI-based methods do not assign any importance to the null artificial feature, and Sobol-CPI(1) combined with the Wilcoxon test is both the most powerful and the fastest thanks to avoiding the refitting of the SuperLearner. }
    \label{fig:app_wdbc}
\end{figure}

Similar conclusions can be drawn for the red wine dataset (Figure~\ref{fig:app_red_wine}). Indeed, it is clear not only that the Sobol-CPI methods do not assign any importance to the artificial feature, but also that they are, in general, much more powerful. Moreover, the Wilcoxon test applied to Sobol-CPI(1) is the most powerful.

\begin{figure}[htbp]
    \centering
    \includegraphics[width=1.0\textwidth]{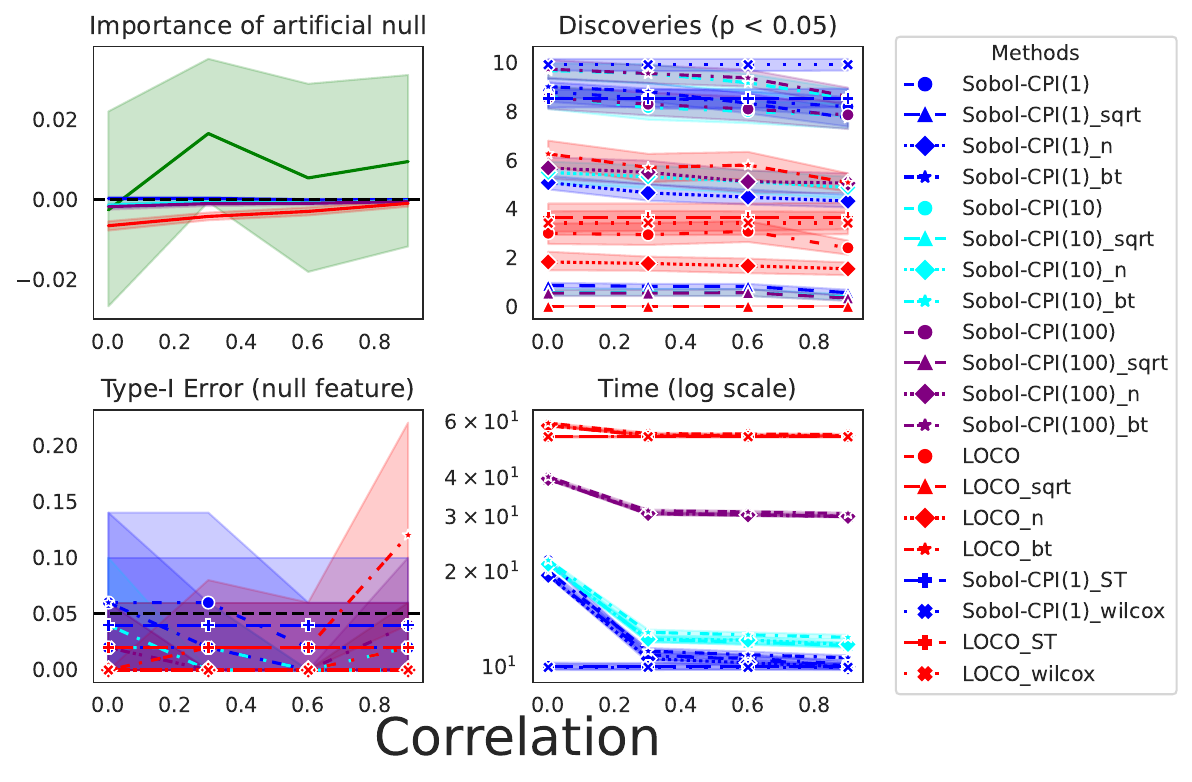}
    \caption{\textbf{Double robustness and inference on the Red Wine dataset with SuperLearner:} Sobol-CPI-based methods do not assign any importance to the null artificial feature, and Sobol-CPI(1) combined with the Wilcoxon test is both the most powerful and the fastest.}
    \label{fig:app_red_wine}
\end{figure}

\subsubsection{Double robustness and Asymptotic relevance }\label{app:asymp_rel_real_data}

In Figure~\ref{fig:app_real_asymp_rele} we report the importance of the additional null feature with respect to its correlation with the original features, across all real datasets from Appendix~\ref{app:datasets} and standard ML models (Lasso, Gradient Boosting, Random Forest, Neural Network, and their SuperLearner). We evaluate \texttt{Sobol-CPI(1/10/100)}, the standard \texttt{LOCO}, its data-splitting variant \texttt{LOCO-W} from \cite{williamson2023general}, and the Gaussian conditional samplers from \texttt{fippy} producing \texttt{CFI} and \texttt{scSAGEvf}, which correspond to \texttt{Sobol-CPI(1)} and \texttt{Sobol-CPI(50)} respectively, without the bias correction from Lemma~\ref{lemm:doubl-rob-LM}, which we manually add. 

The \emph{double robustness} property persists: Sobol--CPI-based methods exhibit a markedly smaller variance than refitting-based LOCO methods and assign (almost) zero importance to the null feature across all datasets and correlation levels, for both regression and Gaussian conditional samplers.

We also study \texttt{PFI} in connection with the asymptotic relevance assumption. \texttt{PFI} typically assigns no importance to the null feature (as visible from the scale), except for high correlations under Neural Networks. This indicates that while the model does not utilize the null feature, there can still be nontrivial gains from conditional samplers. In addition, we recall that \texttt{PFI} does not produce valid $p$-values, making it unsuitable for inference. In contrast, for the \texttt{CFI} we have provided valid inference procedures, such as the Wilcoxon test and the variance correction, enabling hypothesis testing for variable importance.

\begin{figure}[htbp]
    \centering
    \includegraphics[width=1.0\textwidth]{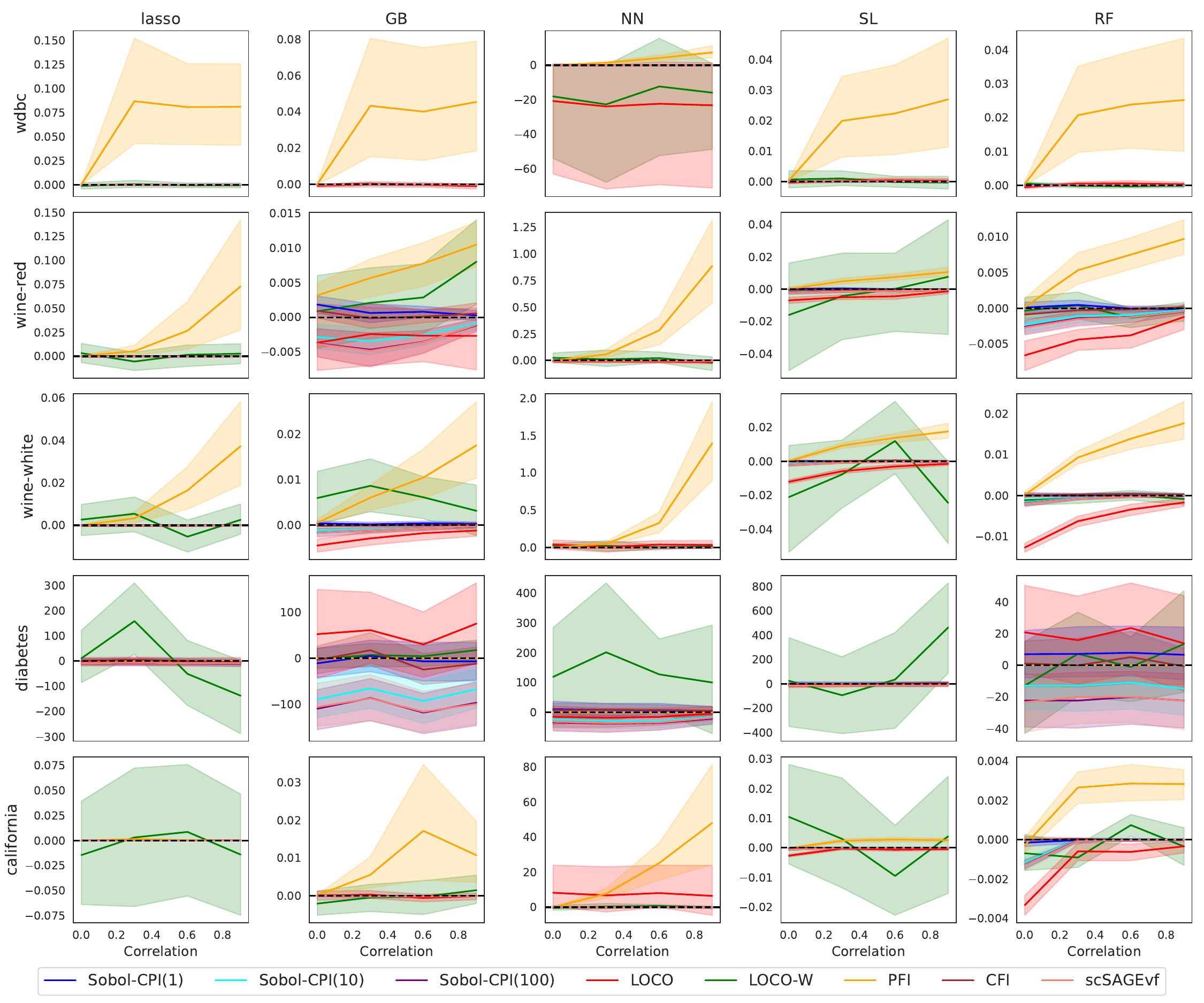}
    \caption{\textbf{Double robustness and asymptotic relevance in real data.}
        Across all real datasets and models, Sobol-CPI methods assign zero importance to the artificial null feature for every correlation level. This also holds for other sampling schemes such as the Gaussian one (\texttt{CFI} $\equiv$ \texttt{Sobol-CPI(1)} and \texttt{scSAGEvf} $\equiv$ \texttt{Sobol-CPI(50)}), illustrating the \textbf{double robustness} property. In contrast, \texttt{PFI}, which can be viewed as a \texttt{CFI} with a severely misspecified imputer, still yields small importance on this null feature, except for high correlations under neural networks, consistent with the \textbf{asymptotic relevance} assumption. }
    \label{fig:app_real_asymp_rele}
\end{figure}

\section{Explicit Total Sobol Index examples}\label{sect:expl_TSI}

\begin{examp}[LM with Gaussian covariates]\label{ex:LOCO-LM} Given $X\sim\mathcal{N}(\mu, \Sigma)$ and $y=\beta X+\epsilon$ we note that
\begin{align*}
    \psi_{\mathrm{TSI}}(j, P_0)&=\e{(m(X)-m_{-j}(X^{-j}))^2}\\
    &=\beta_{j}^2\e{(X^j-\e{X^{j}|X^{-j}})^2}\\
    &=\beta_j^2\e{\mathbb{V}(X^j|X^{-j})}\\
    &=\beta_j^2\Sigma^j_\mathrm{cond},
\end{align*}
with $\Sigma^j_\mathrm{cond}:=\Sigma_{j, j}-\Sigma_{j,-j}\Sigma^{-1}_{-j, -j}\Sigma_{-j, j}$.
\end{examp}

\begin{examp}[Non-linear setting]\label{ex:LOCO-nonlin}
In this example we will recover the example of no-linear setting from \citet{scornet2022mda} but changing the input covariance matrix to obtain more complex relationships between the covariates. Indeed, we will have $y=\alpha X^0X^1\mathbb{I}_{X^2>0}+\beta X^3X^4\mathbb{I}_{X^2<0}$, where $X$ is $p$-dimensional centered Gaussian with a Toeplitz covariance matrix where the $i,j$-th entry is given by $\rho^{|i-j|}$. In this setting, we are going to compute the $\psi_\mathrm{TSI}(j, P_0)$ for the covariate $X^0$ and $X_1$.

First, we observe that $$m_{-0}(X^{-0})=\e{m(X)|X^{-0}}=\alpha \e{X^0|X^{-0}}X^1\mathbb{I}_{X^2>0}+\beta X^3X^4\mathbb{I}_{X^2<0}.$$

Then, we can develop LOCO as

\begin{align*}
    \psi_\mathrm{TSI}(0, P_0)&=\e{(m(X)-m_{-0}(X^{-0}))^2}\\
    &=\e{\left(\alpha X^1\mathbb{I}_{X^2>0}\left(X^0-\e{X^0|X^{-0}}\right)\right)^2}\\
    &=\alpha^2 \e{(X^1)^2\mathbb{I}_{X^2>0}\left(X^0-\e{X^0|X^{-0}}\right)^2}\\
    &=\alpha^2 \e{(X^1)^2\mathbb{I}_{X^2>0}}\e{\left(X^0-\e{X^0|X^{-0}}\right)^2}.\text{ using $X^0-\e{X^0|X^{-0}}\indep X^{-0}$}
\end{align*}

The first term is exactly $\Sigma_{1,1}/2$. To see this, we first observe that as the covariates are centered and symmetrical, then $\e{(X^1)^2\mathbb{I}_{X^2>0}}=\e{(X^1)^2\mathbb{I}_{X^2<0}}$. Therefore, we have that
\begin{align*}
    \Sigma_{1,1}=\e{(X^1-\e{X^1})^2}=\e{(X^1)^2}=\e{(X^1)^2(\mathbb{I}_{X_2>0}+\mathbb{I}_{X_2<0})}=2\e{(X^1)^2\mathbb{I}_{X_2>0}},
\end{align*}
where we have used that $\e{X^1}=0$. We also observe that as it is a Toeplitz matrix, $\Sigma_{1,1}=1$. Then, $\psi_\mathrm{TSI}(0, P_0)=\alpha^2/2\e{\left(X^0-\e{X^0|X^{-0}}\right)^2}=\alpha^2/2\e{\mathbb{V}(X^0|X^{-0})}.$ Note that as it is a Gaussian vector, the variance is exactly $\Sigma_{0,0}-\Sigma_{0,-0}\Sigma^{-1}_{-0,-0}\Sigma_{-0,0}$. We also observe that as it is a Toeplitz matrix, we have the property that $\Sigma_{-0,0}=\rho\Sigma_{-0, 1}=\rho \Sigma_{-0,-0}(\mathbf{1}, \mathbf{0}, \ldots , \mathbf{0})^\top$. Thus, we can develop the last term as
\begin{align*}
    \e{\mathbb{V}(X^0|X^{-0})}&= \Sigma_{0,0}-\Sigma_{0,-0}\Sigma^{-1}_{-0,-0}\Sigma_{-0,0}\\
    &= 1-\rho\Sigma_{0,-0}\Sigma^{-1}_{-0,-0}\Sigma_{-0,-0}(\mathbf{1}, \mathbf{0}, \ldots , \mathbf{0})^\top\\
    &=1-\rho\Sigma_{0,-0}(\mathbf{1}, \mathbf{0}, \ldots , \mathbf{0})^\top\\
    &=1-\rho^2.
\end{align*}
Combining the previous, we conclude that, in this setting, $\psi_\mathrm{TSI}(0, P_0)=(1-\rho^2)/2.$
Similarly, for the first covariate we obtain $\psi_{\mathrm{TSI}}(1, P_0)=\rho^2/2(1-\Sigma_{1, -1}\Sigma_{-1, -1}^{-1}\Sigma_{-1, 1})$.  
\end{examp}

\end{document}